\documentclass[12pt]{article}
\usepackage{amsmath}
\usepackage{psfrag,epsf}
\usepackage{enumerate}
\usepackage{natbib}
\usepackage{url} 

\usepackage[utf8]{inputenc}
\usepackage[T1]{fontenc}
\usepackage{palatino, mathpazo,amsthm}
\usepackage{hyperref}
\usepackage{parskip}
\usepackage{graphicx,color,xcolor}
\definecolor{darkred}{RGB}{179,0,0} 
\usepackage{boxedminipage}
{\begin{boxedminipage}{#1}}
{\end{boxedminipage}}%

\newcommand{\blind}{0}

\usepackage{soul}
\definecolor{Dgreen}{HTML}{26b003}

\usepackage{enumitem}
\newcommand{\bfX}{{\bf X}}
\DeclareMathOperator*{\argmin}{argmin}
\newcommand{\Hh}{ \mathcal{H} }
\newcommand{\ra}{\rightarrow}
\newcommand{\R}{\mathbb R}
\newcommand{\MBD}{\mathrm{MBD}}

\def\bX{{\bf X}}

\hypersetup{
  colorlinks,
  citecolor=black,%
  filecolor=black,%
  linkcolor=black,%
  urlcolor=darkred
}

\newtheorem{theorem}{Theorem}[section]
\newtheorem{lemma}[theorem]{Lemma}

\newtheorem{definition}[theorem]{Definition}

\begin{document}

\def\spacingset#1{\renewcommand{\baselinestretch}%
{#1}\small\normalsize} \spacingset{1}

\if0\blind
{
  \title{\bf The accumulated persistence function,  a new useful functional summary statistic for topological data analysis, with a view to brain artery trees and spatial point process applications}
  \author{Christophe A.N. Biscio\hspace{.2cm}\\
    Department of Mathematical
    Sciences, Aalborg University, Denmark\\
    and \\
    Jesper Møller \\
    Department of Mathematical
    Sciences, Aalborg University, Denmark}
  \maketitle
} \fi

\if1\blind
{
  \bigskip
  \bigskip
  \bigskip
  \begin{center}
    {\LARGE\bf The accumulated persistence function,  a new useful functional summary statistic for topological data analysis, with a view to brain artery trees and spatial point process applications}
\end{center}
  \medskip
} \fi

\bigskip
\begin{abstract}
We start with a simple introduction to topological data analysis where the most popular tool is called a 
 persistent diagram. Briefly, a persistent diagram is a multiset of points in the plane 
describing the persistence of topological features of a compact set when 
 a scale parameter varies.  
Since statistical methods are difficult to apply directly 
on persistence diagrams, various alternative functional 
summary statistics have been suggested, but either
they do not contain  the full information of the persistence diagram 
or they are two-dimensional functions. 
 We suggest a new functional summary statistic that is one-dimensional and hence easier to handle, and 
which under mild conditions contains the full information of the persistence diagram. 
Its usefulness is illustrated in statistical settings concerned with point clouds and brain artery trees. 
The appendix includes additional methods and examples, together with technical details. 
The {\sf {\bf R}}-code used for all examples is available at \url{http://people.math.aau.dk/~christophe/Rcode.zip}.
\end{abstract}

\noindent%
\vfill

\newpage
\spacingset{1.45} 

\section{Introduction} 

Statistical methods that make use of algebraic topological
ideas to summarize and visualize complex data 
are called 
topological data analysis (TDA). In particular persistent homology and a method called the persistence diagram are used to measure the persistence of topological features. 
As we expect many readers may not be familiar with these concepts,  Section~\ref{s:simple-ex} 
discusses two examples without going into technical details though a few times it is unavoidable to refer to the terminology used in persistent homology. Section~\ref{s:1:background} discusses the use of the persistence diagram and related summary statistics
and motivates why in Section~\ref{s:1:APF} 
a new  
functional summary statistic called the accumulative persistence function (APF) is introduced. The remainder of the paper demonstrates the use of the APF in various statistical settings concerned with point clouds and brain artery trees.

\subsection{Examples of TDA}\label{s:simple-ex} 

The mathematics underlying TDA 
uses technical definitions
and results from persistent homology, see \cite{fasy:etal:14} and the references therein. This theory will not be needed for the present paper. Instead we provide an understanding through examples of the notion of persistence of 
 $k$-dimensional topological features for a sequence of compact subsets $C_t$ of the $d$-dimensional Euclidean space  $\mathbb R^d$, where $t\ge0$, $k=0,1,\ldots,d-1$, and either $d=2$ (Section~\ref{s:toy}) or $d=3$ (Section~\ref{s:brain}).
Recalling that a set $A\subseteq \mathbb R^d$ is path-connected if any two points in $A$ are connected by a curve in $A$,  
a $0$-dimensional topological feature of a compact set $C\subset\mathbb R^d$ is a maximal path-connected subset of $C$, also called a  {\it connected component} of $C$. The meaning of a $1$-dimensional topological feature is simply understood when $d=2$, and we appeal to this in a remark at the end of Section~\ref{s:toy} when $d=3$.

\subsubsection{A toy example}\label{s:toy}
 
Let $C\subset \mathbb R^2$ be the union of the three circles  depicted in the
top-left panel of Figure~\ref{fig:TDA_tuto}. 
The three circles are the $0$-dimensional topological features (the connected components) of $C$, as  any curve that goes from a circle to another will be outside $C$. The complement $\mathbb R^2\setminus S$ has four connected components, one of which is unbounded, whilst the three others are the $1$-dimensional topological features of $C$, also called the 
{\it loops} of $C$ (the boundary of each bounded connected component is a closed curve with no crossings; in this example the closed curve is just a circle).

 \begin{figure}
 \centering 
 \begin{tabular}{ccc}  
\includegraphics[scale=0.15]{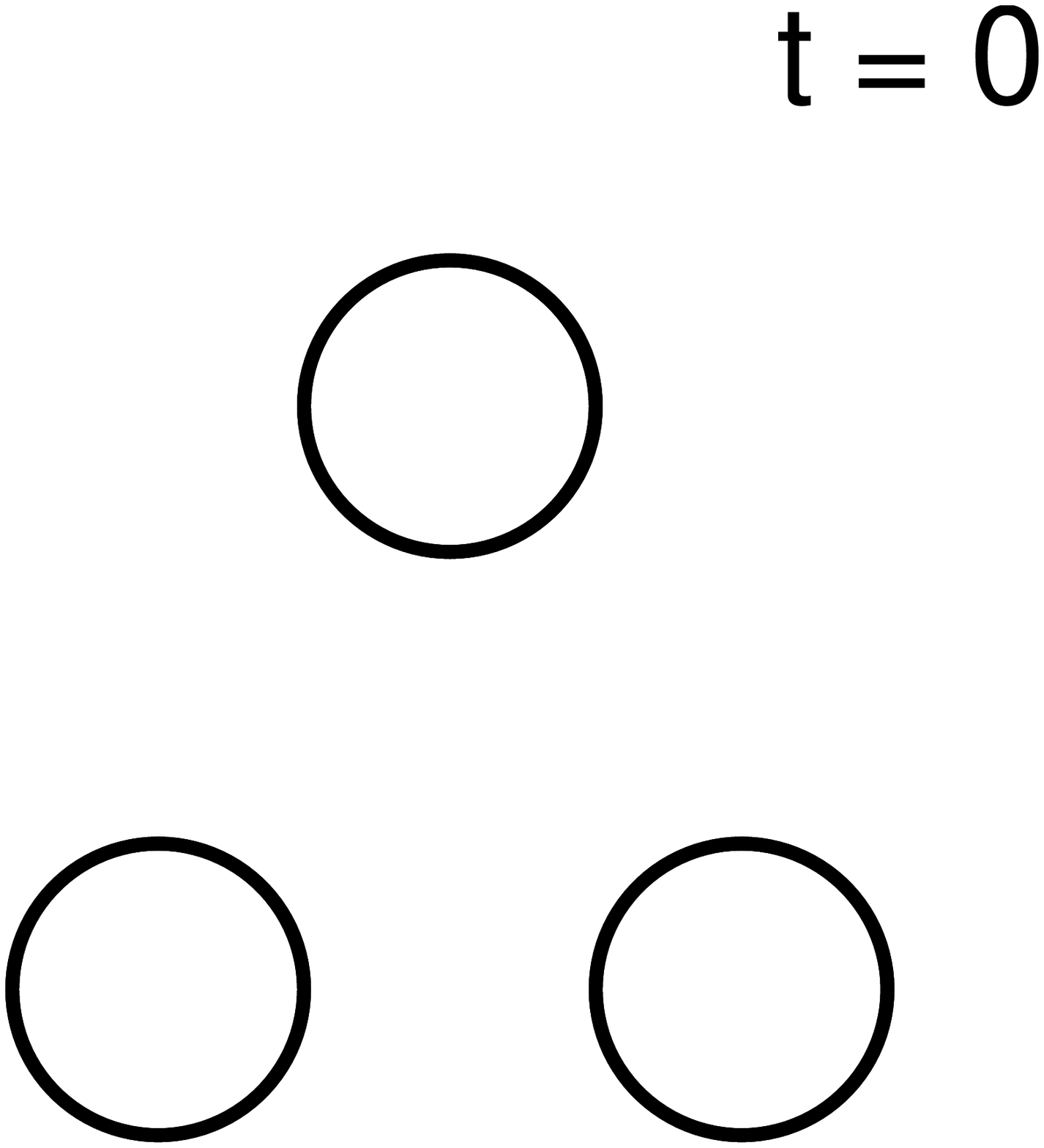}  & \includegraphics[scale=0.15]{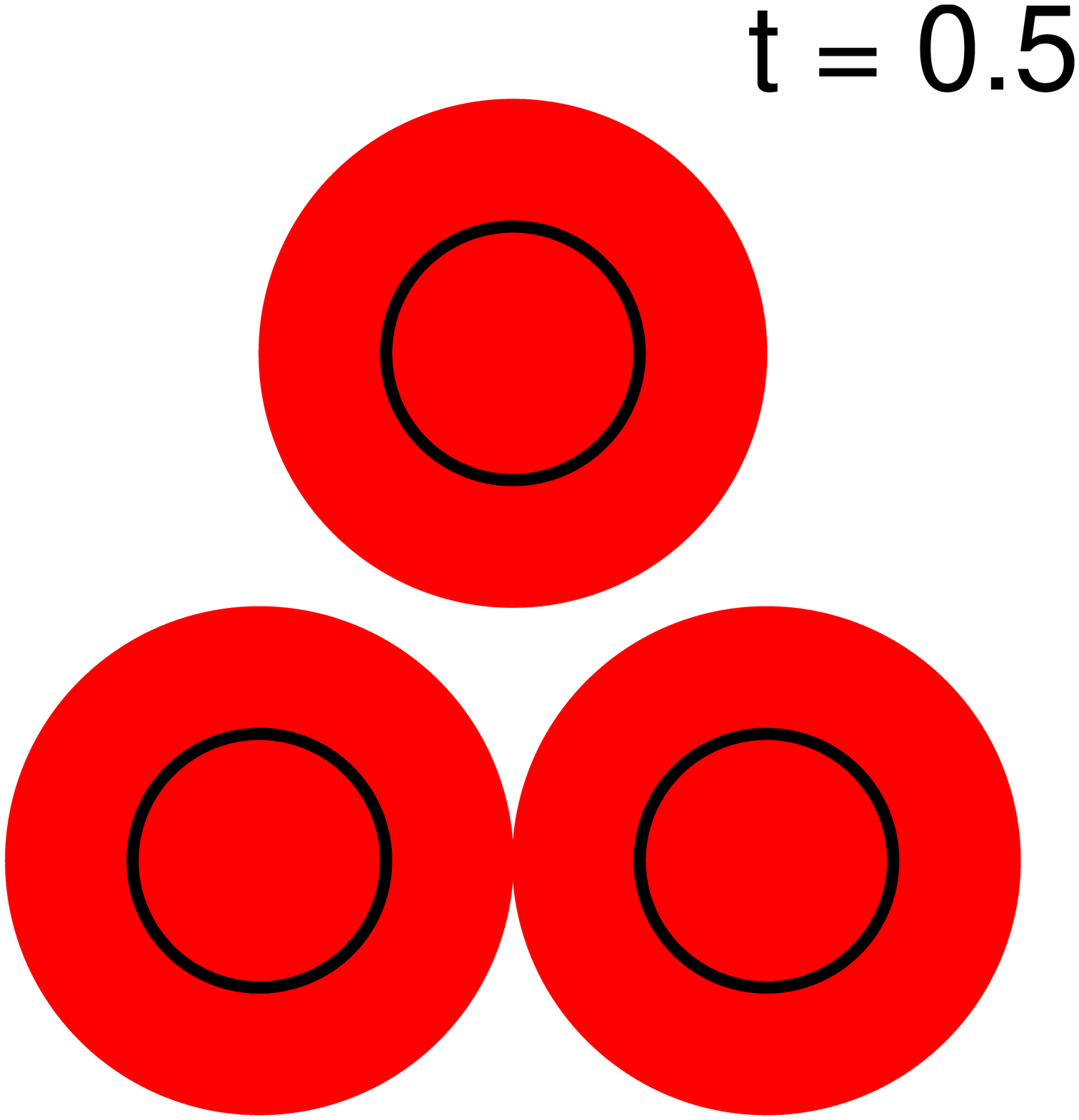} & 
\includegraphics[scale=0.15]{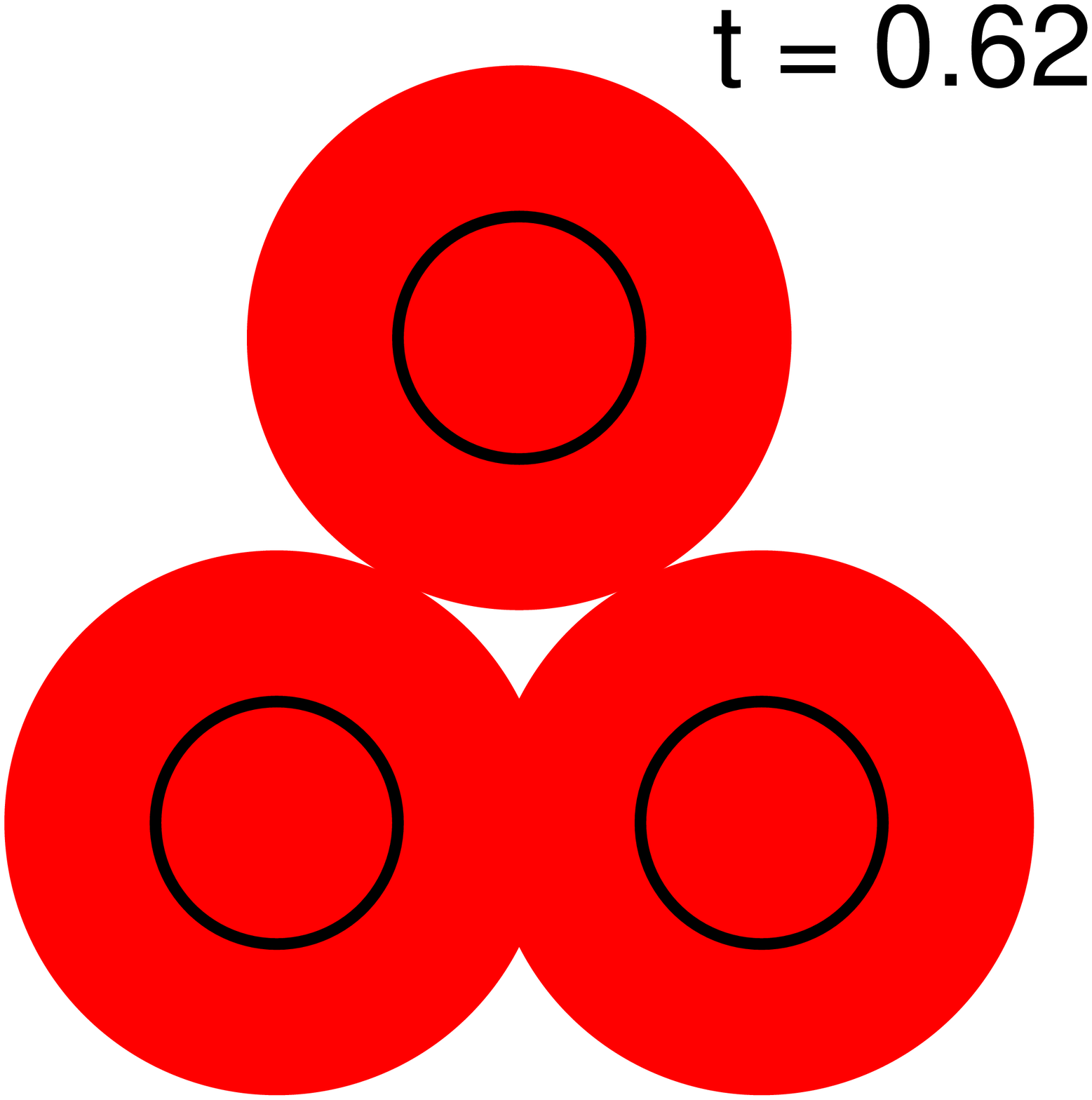} \\ 
\includegraphics[scale=0.15]{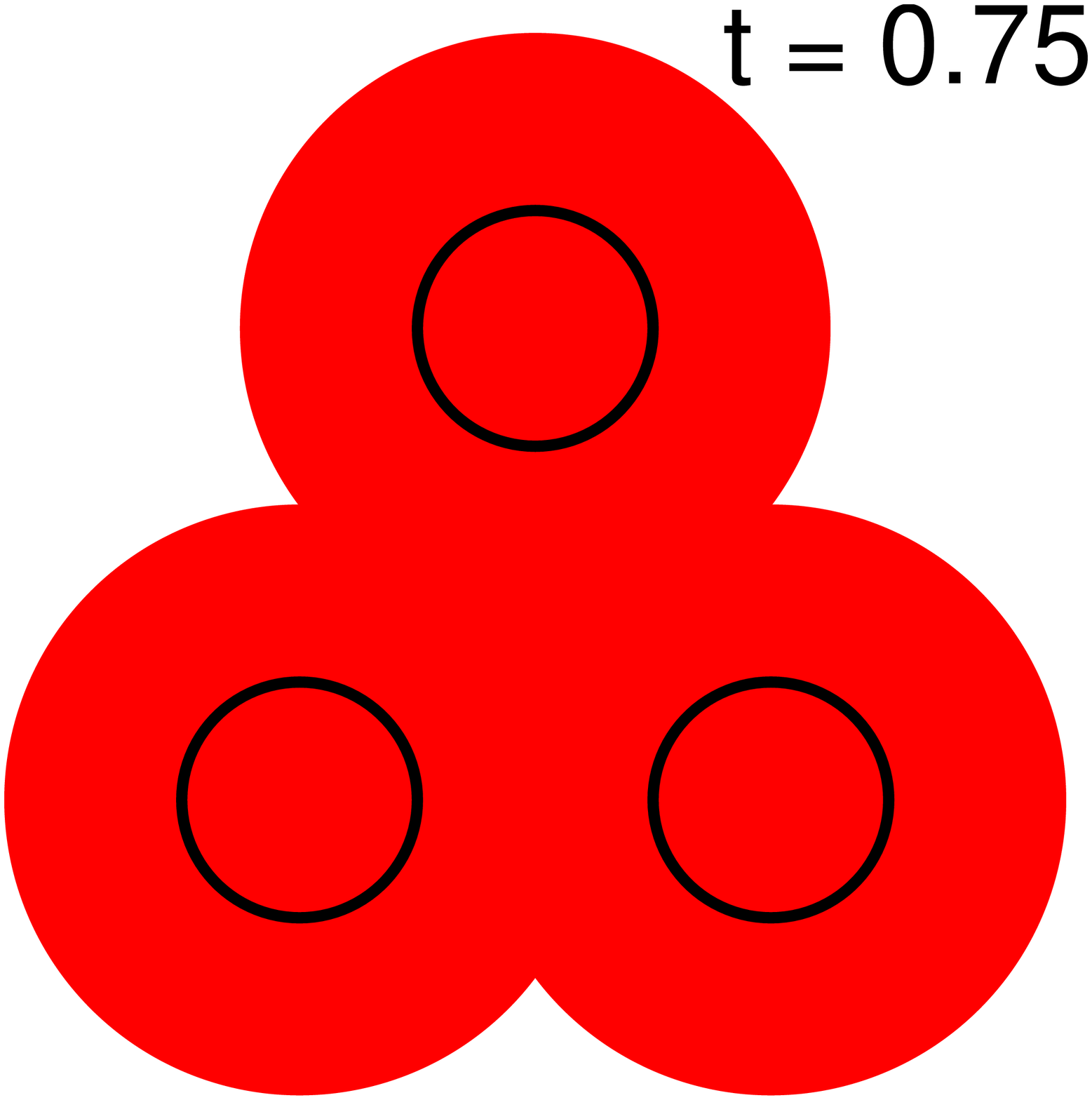} & \includegraphics[scale=0.15]{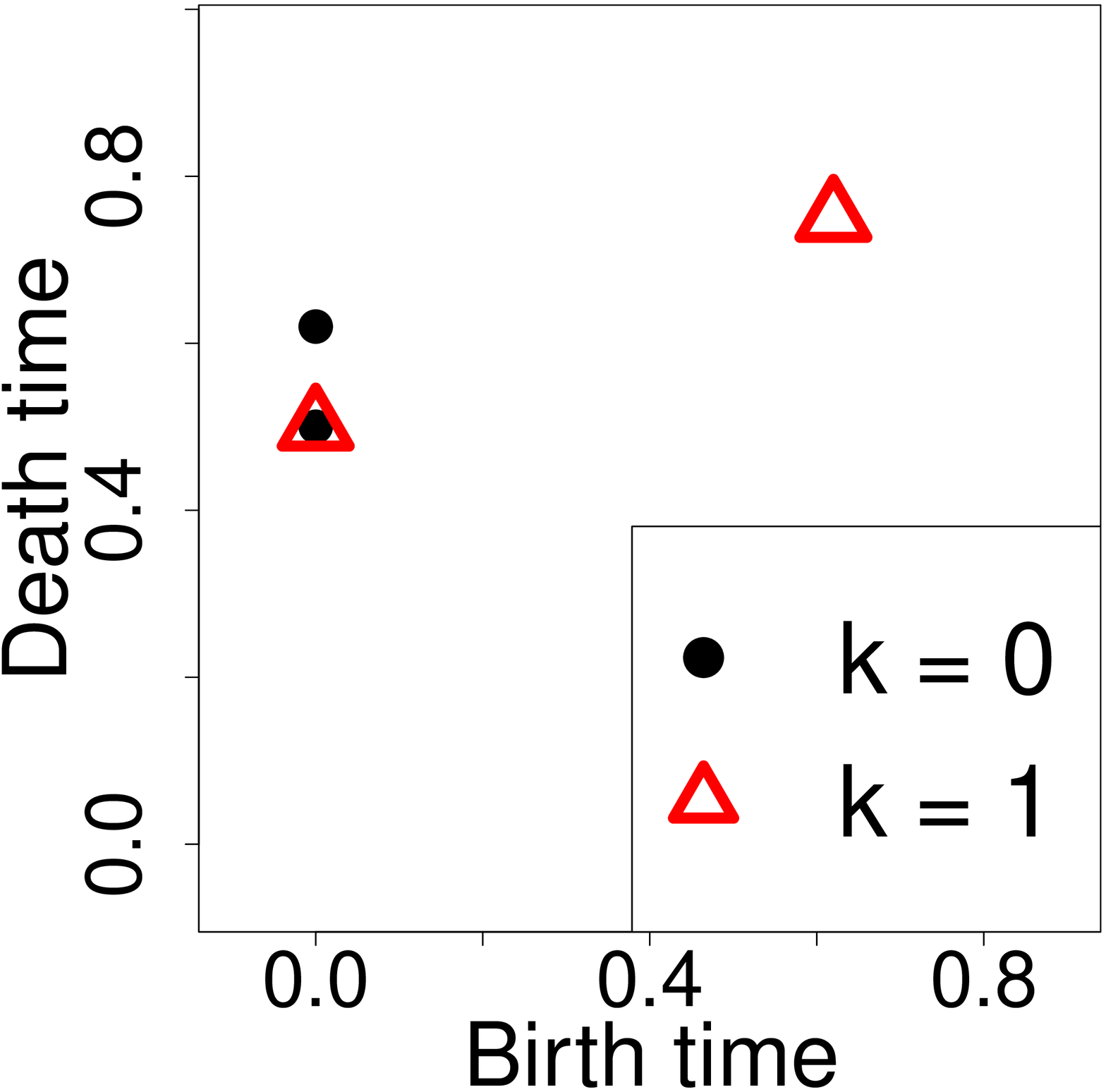} & 
\includegraphics[scale=0.15]{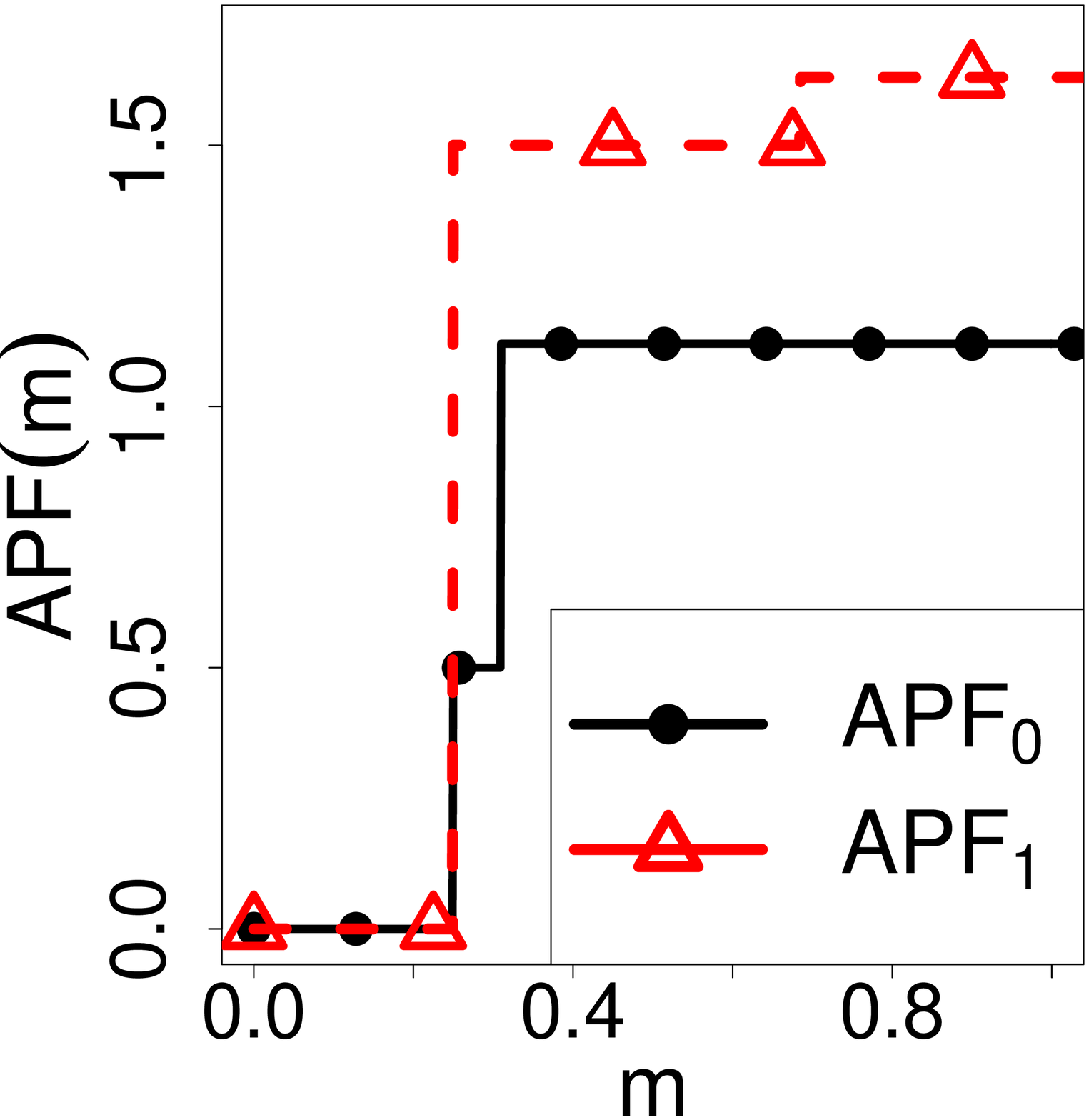} 
\end{tabular} 
\caption{The four first panels show a simple example of spherically growing circles centred at (-1,-1), (1,-1), and (0,1) and all with radii 0.5 when time is $0$. The fifth panel shows
the persistence diagram for the connected components ($k=0$) and the loops ($k=1$). The
final panel shows the corresponding accumulated persistence functions.} 
\label{fig:TDA_tuto}
 \end{figure}  

For $t\ge0$, let $C_{t}$ be the subset of points in $\mathbb R^2$ within distance $t$ from $C$.  
Thinking of $t$ as time, $C_{t}$ results when each point on $C$ grows as a disc with constant speed one. The algebraic topology (technically speaking the Betti numbers) changes exactly at the times $t=0,0.5,0.62,0.75$, see the first four panels of Figure~\ref{fig:TDA_tuto}: For each topological dimension $k=0,1$, let $t_i^{(k)}$ denote the time of the $i$th change. First, $C_{0}=C$ has three connected components and three loops as 
given above; we say that they are {\it born} at time $t_1^{(0)}=t_1^{(1)}=0$ (imaging there was nothing before time 0). 
Second, the loops disappear and two connected components merge into one connected component; we say that the loops and one of the connected components {\it die} at time $t_2^{(0)}=t_2^{(1)}=0.5$; 
since the two merging connected components were born at the same time, it is decided uniformly at random
which one should die respectively survive; the one which survives then represent the new merged connected component.
Third, at time $t_3^{(0)}=t_3^{(1)}=0.62$, a new loop is born and the two connected component merge into one which is represented by the oldest (first born) connected component whilst the other connected component (the one which was retained when the two merged at time $t_2^{(0)}=0.5$) 
 dies; the remaining connected component "lives forever" after time $0.62$ and it will be discarded in our analysis. 
Finally, at time $t_4^{(1)}=0.75$, the loop dies.     

Hence, for each $k=0,1$, there is a multiset of points specifying the appearance and disappearance of each $k$-dimensional topological feature as $t$ grows. A scatter plot of these points is called a {\it persistence diagram}, see Figure~\ref{fig:TDA_tuto} (bottom-middle panel): For $k=0$ (the connected components), the points are $(0,0.5)$ and $(0,0.62)$ (as $(0,\infty)$ is discarded from the diagram) with multiplicities 1 and 1, respectively; and for $k=1$ (the loops), the points are $(0,0.5)$ and $(0.62,0.75)$ with multiplicities 3 and 1, respectively. The term "persistence" refers to that 
 distant connected components and large 
loops are present for a long time; which here of course just corresponds to the three circles/connected components of $C$ and their loops persist for long whilst the last appearing loop has a short lifetime and hence is considered as "noise".

Usually in practice $C$ is not known but a finite subset of points $\{x_1,\ldots,x_N\}$ has been collected as a
sample on $C$, possibly 
with noise. Then we redefine $C_{t}$ as the union of closed discs  
of radius $t$ and with centres given by the point cloud. 
Hence the connected components of $C_0$ are just the points $x_1,\ldots,x_N$, and $C_0$ has no loops. 
For $t>0$, it is in general difficult to directly compute the connected components and loops of $C_t$, but
a graph in $\mathbb R^m$ (with $m\ge2$) can be constructed 
so that its connected components correspond to those of $C_t$ and moreover
the triangles of the graph may be filled or not in a way so that the loops of the obtained triangulation correspond to those of $C_t$. 

\noindent {\it Remark:} Such a construction can also be created in the case where  $C_{t}$ is the union of $d$-dimensional closed balls 
of radius $t$ and with centres given by a finite point pattern $\{x_1,\ldots,x_N\}\subset\mathbb R^d$. The construction is a so-called simplicial complex such as the \v{C}ech-complex, where $m$ may be much larger than $d$, or the Delaunay-complex (or alpha-complex), where $m=d$, and a technical result (the Nerve Theorem) establishes that it is possible to identify the topological features of $C_t$ by the \v{C}ech or Delaunay-complex, see e.g.\ \cite{Edelsbrunner:Harer:10}. It is unnecessary for this paper to understand the precise definition of these notions, but as $d=2$ or $d=3$ is small in our examples, it is computationally convenient to use the  Delaunay-complex. 
When $d=3$, we may still think of a 1-dimensional topological feature as a loop, i.e.\ a closed curve with no crossings; again the simplicial complex is used for the "book keeping" when determining the persistence of       
a loop. For example, a 2-dimensional sphere has no loops, and a torus in $\mathbb R^3$ has two. 
Finally, when $d\ge3$, a $k$-dimensional topological feature is a $k$-dimensional manifold (a closed surface if $k=2$) that cannot "be filled in", but for this paper we omit the precise definition since it is technical and not needed.

\subsubsection{Persistent homology for brain artery trees}\label{s:brain}

The left panel of Figure~\ref{fig:brain_artery_tree} shows an example of one of the 98 brain artery trees analysed in
\cite{steve:16}. The data for each tree specifies a graph in $\mathbb R^3$ consisting of a dense cloud of about $10^5$ vertices (points) together with the edges (line segments) connecting the neighbouring vertices; further details about the data are
given in Section~\ref{s:brains}.  
As in \cite{steve:16}, for each tree
we only consider the $k$-dimensional topological features when $k=0$ or $k=1$, using different types of data and sets $C_t$ as described below. Below we consider the tree in Figure~\ref{fig:brain_artery_tree} and let $B\subset\mathbb R^3$ denote 
the union of its edges.
 
Following~\cite{steve:16}, if $k=0$, let 
 $C_t=\{(x,y,z)\in B: z\le t\}$ be the sub-level set of the height function for the tree at level $t\ge0$ (assuming $C_t$ is empty for $t<0$).
Thus the 0-dimensional topological features at "time/level" $t$ are the connected components of $C_t$.  
 As illustrated in the left panel of Figure~\ref{fig:brain_artery_tree}, instead of time we may think of $t$ as "water level": As the water level increases,  
 connected components of the part of $B$ surrounded by water (the part in blue) may be born or die; we refer to this as sub-level persistence. 
 As in Section~\ref{s:toy}, we represent the births and deaths of the connected component in a persistence diagram which is shown in Figure~\ref{fig:brain_artery_tree} (middle panel).  
 The persistence of the connected components for all brain artery trees will be studied in several examples later on.

As in \cite{steve:16}, if $k=1$, we let $B$ be represented by a point pattern $C$ of 3000 points subsampled from $B$, and redefine $C_t$ to be the union of balls of radii $t\ge0$ and centres given by $C$ (as considered in the remark at the end of Section~\ref{s:toy}). The loops of $C_t$ are then determined by the corresponding Delaunay-complex.
The right panel of Figure~\ref{fig:brain_artery_tree} shows the corresponding persistence diagram.
 The persistence of the loops for all trees will be studied in the examples to follow.

\begin{figure}
\centering \setlength{\tabcolsep}{0.03\textwidth}
\begin{tabular}{ccc} 
 \includegraphics[scale=0.23]{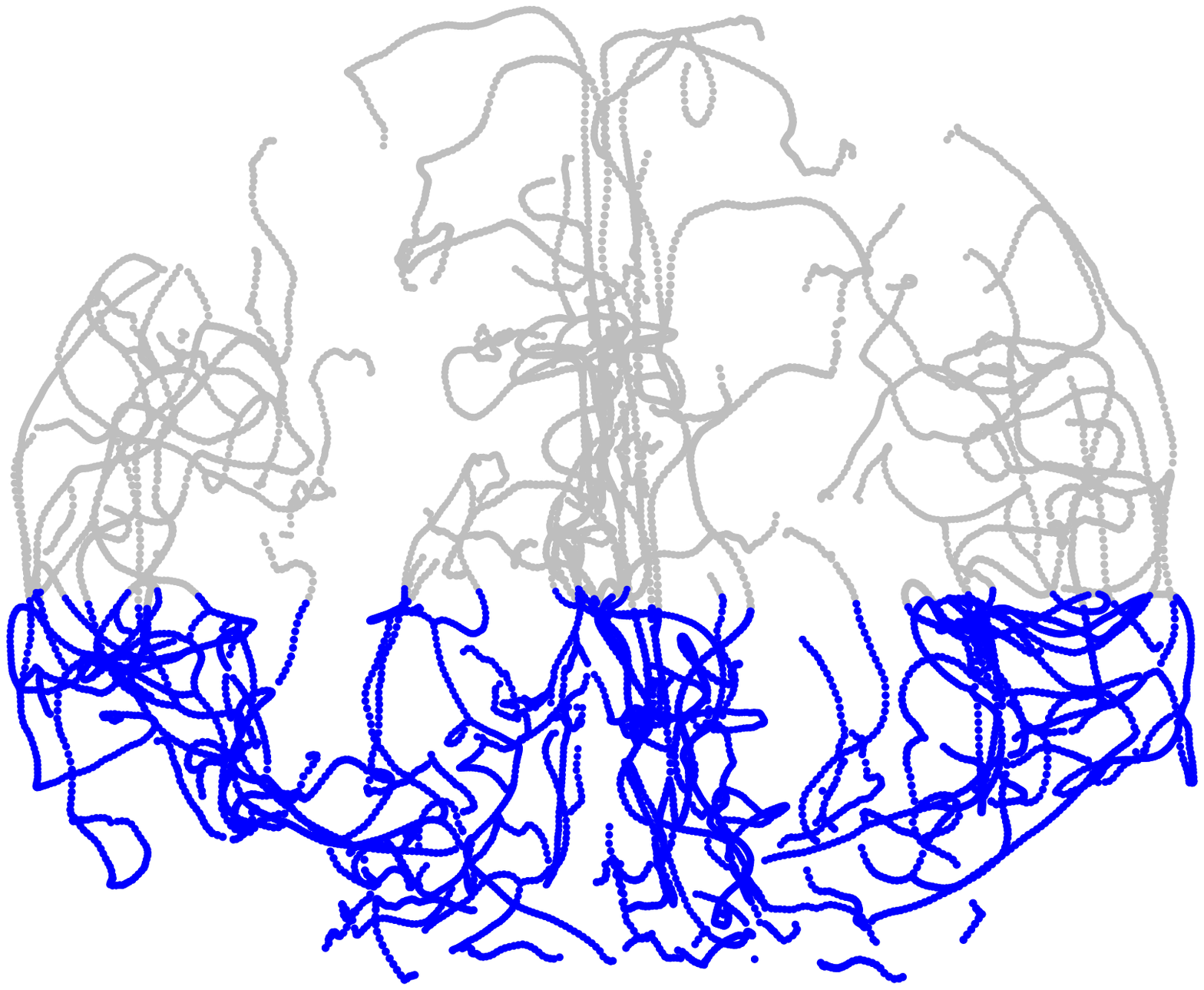}&  \includegraphics[scale=0.15]{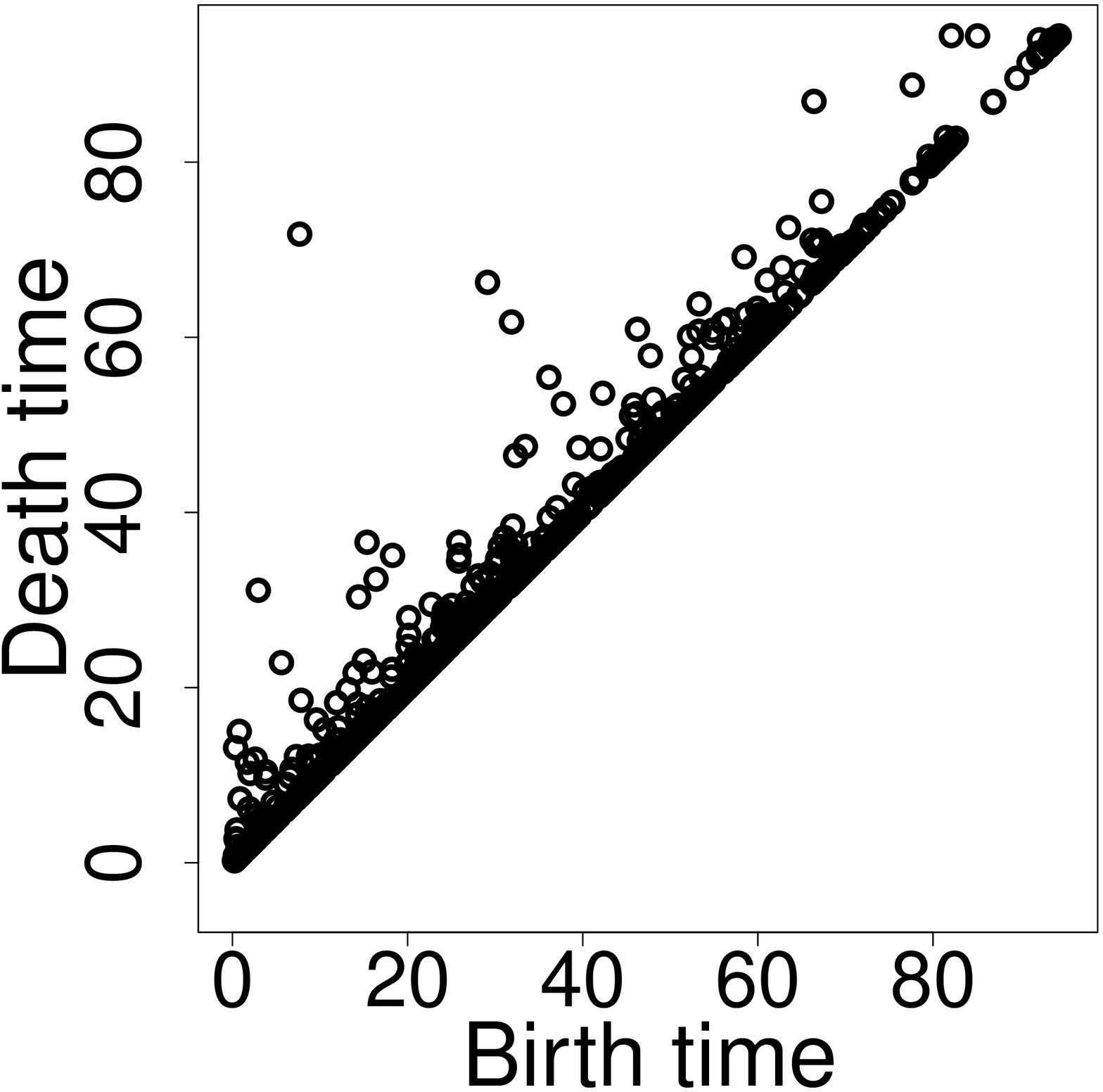} &   \includegraphics[scale=0.15]{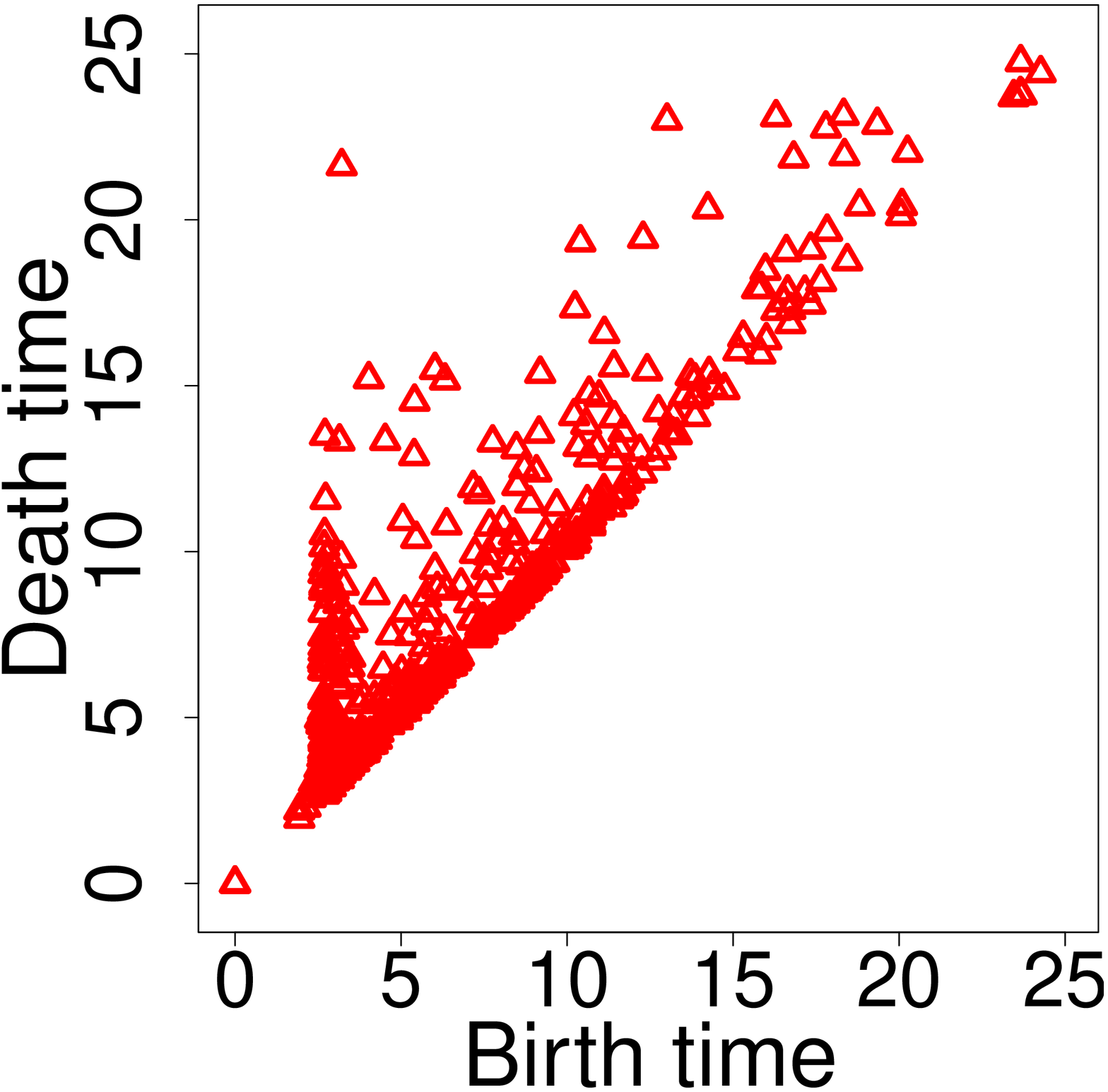}   
\end{tabular}
\caption{A brain artery tree with the "water level" indicated (left panel) and the persistence diagrams of connected components (middle panel) and loops (right pane).}
\label{fig:brain_artery_tree}
 \end{figure}

\subsection{Further background and objective}\label{s:1:background}

The persistence diagram is a popular graphical representation of the persistence of the topological features of a sequence of compact sets $C_t\subset\mathbb R^d$, $t\ge0$.
As exemplified above it consists 
for each topological dimension $k=0,\ldots,d-1$ 
of a multiset $\mathrm{PD}_k$ of points $(b_i,d_i)$ with multiplicities $c_i$, where $b_i$ and $d_i$ is a pair of birth-death times for a $k$-dimensional topological feature obtained as time $t$ grows.  

In the majority of literature on TDA, including the analysis in \cite{steve:16} of brain artery trees, long lifetimes  
are of main interest whereas short lifetimes are considered as topological noise.
Short lifetimes are of interest in the study of complex structures 
such as branch polymers and fractals, see \cite{macpherson:12}; and for brain artery trees \cite{steve:16} noticed in one case that "not-particularly-high persistence have the most distinguishing power in our specific application". 
In our examples we demonstrate that short lifetimes will also be of key interest in many  situations, including when analysing the brain artery trees dataset from \cite{steve:16}.

\cite{chazal:etal:13}~and \cite{chen:etal:15} note that  it is
difficult to apply statistical methodology to persistent diagrams. 
Alternative functional summary statistics have been 
suggested: \cite{bubenik:12:publish} introduces a sequence of 
one-dimensional functions called the persistent landscape, 
where his first function is denoted $\lambda_1$ and is considered to be of main interest, 
since it provides a measure of the dominant topological
features, i.e.\ the longest lifetimes; therefore we call $\lambda_1$ 
the dominant function. 
\cite{chazal:etal:13} introduce 
the silhouette which is 
a weighted average of the functions of the persistent landscape, 
where the weights control whether the focus is on topological features with long or short lifetimes. 
Moreover, \cite{chen:etal:15} consider a kernel estimate of 
the intensity function for the persistent diagram viewed as a point pattern. 
The dominant function, the silhouette, and the intensity estimate are one-dimensional functions and hence easier to handle than the persistence diagram, however, they provide selected and not
 full information about the persistence diagram. 
In Section~\ref{s:1:APF}, we introduce another one-dimensional 
functional summary statistic called the accumulative persistence function 
and discuss its advantages and how it differs from the existing functional summary statistics.  

 \subsection{The accumulated persistence function}\label{s:1:APF}
 
For simplicity and specificity, for each topological dimension $k=0,1,\ldots,d-1$, we always assume that the persistence diagram $\mathrm{PD}_k=\{(b_1,d_1,c_1),\ldots,(b_n,d_n,c_n)\}$ is such that $n<\infty$ and $0\le b_i<d_i<\infty$ for $i=1,\ldots,n$.
This assumption will  
be satisfied in our examples (at least with probability one).  
 Often in the TDA literature, $\mathrm{PD}_k$ is transformed to the \textit{rotated and rescaled persistence diagram} (RRPD) given by  
 $\mathrm{RRPD}_k= \lbrace 
(m_1,l_1,c_1), \ldots, (m_n,l_n,c_n) \rbrace$, where $m_i= (b_i+d_i)/2$ is the meanage and $l_i=d_i-b_i$ is the lifetime. 
This transformation is useful when 
 defining our \textbf{\textit{accumulative persistence function} (APF)} by 
\begin{equation}\label{e:(a)}
\mathrm{APF}_k(m)=\sum_{i=1}^n c_i l_i 1(m_i\le m),\quad m\geq 0,
\end{equation}
where $1(\cdot)$ is the indicator function and we suppress in the notation 
that $\mathrm{APF}_k$ is a function of $\mathrm{RRPD}_k$. The remainder of this section comments on this definition. 

Formally speaking, when $\mathrm{RRPD}_k$ is considered to be random, it is viewed as
a finite point process with multiplicities, see e.g.\ \cite{daley:vere-jones:03}. 
It what follows it will always be clear from the context whether $\mathrm{PD}_k$ and $\mathrm{RRPD}_k$ 
are considered as being random or observed, and hence whether  $\mathrm{APF}_k$ is a deterministic or random function. 
In the latter case, because $\mathrm{APF}_k(m)$ is an accumulative function, its random fluctuations typically increase as $m$ increases.  

Depending on the application, the jumps and/or the shape of $\mathrm{APF}_k$ may be of interest as demonstrated later in our examples.
A large jump of $\mathrm{APF}_k$ corresponds to a large lifetime (long persistence). In the simple example shown in Figure~\ref{fig:TDA_tuto}, both jumps of $\mathrm{APF}_0$ are large and indicate the three connected components (circles), whilst only the first jump of $\mathrm{APF}_1$ is large and indicates the three original loops. For the more complicated examples considered in the following it may be hard to recognize the individual jumps. 
In particular, as in the remark at the end of Section~\ref{s:toy}, suppose $C_t$ is the union of $d$-dimensional balls of radius $t$ and with centres given by a finite point pattern $\{x_1,\ldots,x_N\}\subset\mathbb R^d$. Roughly speaking we may then have the following features as illustrated later in Example~1. 
For small meanages $m$, jumps of $\mathrm{APF}_0(m)$  correspond to balls that merge together for small values of $t$.
Thus, if the point pattern is aggregated (e.g.\ because of clustering), we expect that
$\mathrm{APF}_0(m)$ has jumps and is hence large for small meanages $m$, whilst
if the point pattern is regular (typically because of inhibition between the points), we expect the jumps of $\mathrm{APF}_0(m)$ to happen and to be large for modest values of $m$ (as illustrated later in the middle panel of Figure~\ref{fig:simEx2} considering curves for the Mat{\'e}rn cluster process and the determinantal point process). 
For large meanages, jumps of $\mathrm{APF}_0(m)$ are most likely to happen in the case of aggregation. Accordingly, the shape of $\mathrm{APF}_0$ can be very different for these two cases (as illustrated in the first panel of Figure~\ref{fig:simEx2}). 
Similar considerations lead us to expect different shapes of $\mathrm{APF}_1$ for different types of point patterns; we expect that 
$\mathrm{APF}_1(m)$ is large respective small for the case of aggregation respective regularity when $m$ is small, and the opposite happens when $m$ is large (as illustrated in the last panel of Figure~\ref{fig:simEx2}).  

Clearly, $\mathrm{RRPD}_k$ is in one-to-one correspondence to $\mathrm{PD}_k$. In turn,
 if all $c_i=1$ and the $m_i$ are pairwise distinct, 
 then 
 there is a one-to-one correspondence 
between $\mathrm{RRPD}_k$ and its corresponding $\mathrm{APF}_k$. For $k=0$, 
this one-to-one correspondence would easily be lost if we had used $b_i$ in place of $m_i$ in \eqref{e:(a)}. 
We need to be careful with not over-stating this possible one-to-one correspondence. 
For example, imagine we want to compare two APFs with respect to $L^q$-norm ($1\le q\le\infty$) and let 
$\mathrm{PD}_k^{(1)}$ and $\mathrm{PD}_k^{(2)}$ be the underlying persistence diagrams. However, when points $(b,d)$ close to the diagonal are considered as topological noise (see Section~\ref{s:toy}), usually the  
so-called bottleneck distance $W_\infty(\mathrm{PD}_k^{(1)},\mathrm{PD}_k^{(2)})$ is used, see e.g.\ \cite{fasy:etal:14}. Briefly, for $\epsilon>0$, let 
$\mathcal N=\{(b,d):\,b \le d,\,l\le 2 \epsilon\}$ be the set of points at distance $\sqrt{2}\epsilon$ of the diagonal in the persistence diagram, and let $S(b,d)=\{(x,y):\,|x-b|\le \epsilon,\,|y-d|\le \epsilon\}$ be the square with center $(b,d)$, sides parallel to the $b$- and $d$-axes, and of side length $2\epsilon$. Then
$W_\infty(\mathrm{PD}_k^{(1)},\mathrm{PD}_k^{(2)})\le \epsilon$ if $\mathrm{PD}_k^{(2)}$ has exactly one point in each square $S(b_i,d_i)$, with $(b_i,d_i)$ a point of $\mathrm{PD}_k^{(1)}$ (repeating this condition $c_i$ times).  
However, small values 
of $W_\infty(\mathrm{PD}_k^{(1)},\mathrm{PD}_k^{(2)})$ does not correspond to 
closeness of the two
corresponding APFs with respect to $L^q$-norm.

Note that
the dominant function, the silhouette, 
and the intensity estimate (see Section~\ref{s:1:background})
are in general not in a one-to-one correspondence with $\mathrm{RRPD}_k$. 
Like these functions, $\mathrm{APF}_k$ is a one-dimensional function, 
and so it is easier to handle than the sequence of functions 
for the persistent landscape in~\cite{bubenik:12:publish}
and the intensity estimate in~\cite{chen:etal:15}
--- e.g.\ confidence regions become easier to plot.
Contrary to the dominant function and the silhouette, 
the APF provides information about topological features 
without distinguishing between long and short lifetimes.

 \subsection{Outline}\label{s:1:outlines}

Our paper discusses various methods based on APFs in different contexts and illustrated  
by simulation studies related to spatial point process applications
and 
by re-analysing the brain artery trees dataset previously analysed in~\cite{steve:16}.  
Section~\ref{s:1:dataset} specifies the setting for these examples.
Sections~\ref{s:conf for APFS}, \ref{s:one sample}, and \ref{s:two sample problem} consider the case of a single APF, a sample of APFs, and two samples of APFs, respectively. Further examples and details appear in Appendix~A-F.

\section{Datasets}\label{s:1:dataset}

\subsection{Simulated data}\label{s:1:simulated dataset}

In our simulation studies  
we consider a planar point cloud, i.e.\  a finite point pattern $\{x_1,\ldots,x_N\} \subset \R^2$,
and study as at the end of Section~\ref{s:toy} how the topological features of $C_t$, the union of closed discs of radii $t$ and centred at $x_1,\ldots,x_N$, change as $t$ grows.  
Here
$\lbrace x_1,\ldots,x_N \rbrace$ will be a realisation of a point process 
$\bX\subset \R^2$, where the count $N$ is finite. 
Thus $\mathrm{PD}_k$ and $\mathrm{RRPD}_k$ can be viewed as finite planar point processes (with multiplicities) and
 $\mathrm{APF}_k$ as a random function. Note that $N$ may be random, and conditional on $N$, the points in $\bX$ are not necessarily  
independent and identically distributed (IID). This is a common situation in spatial statistics, e.g.\ if the focus is on the point process $\bX$ and the purpose 
is to assess the goodness of fit for a specified point process model of $\bX$ when $\lbrace x_1,\ldots,x_N \rbrace$ is observed.

  \subsection{Brain artery trees dataset}\label{s:brains}

The dataset in~\cite{steve:16} comes from 98 brain artery trees which can be included  
 within a cube of side length at most 206 mm; one tree is excluded 
"as the java/matlab function crashed" (e-mail correspondence with Sean Skwerer). 
They want to capture how the arteries bend through space and to detect age and gender effects.
For $k=0$, sub-level persistence of the connected components of each tree represented by a union of line segments is considered, 
cf.\ Section~\ref{s:brain}; then for all meanages, $m_i\le 137$; and the number of connected components is always below 3200. 
 For $k=1$,  
 persistence of the loops for the union of growing balls with centres at a point cloud representing the tree is considered,  
 cf.\ Section~\ref{s:brain}; 
the loops  have a finite death time but 
some of them do not die during the allocated time $T=25$ 
(that is, \cite{steve:16} stop the growth of balls when $t>25$). Thus we shall only consider meanages $m_i\le 25$; then the number of loops is always below 2700. 
For each tree and $k=0,1$, most $c_i=1$ and sometimes $c_i>1$.

For each tree and $k=0,1$, 
\cite{steve:16} use only the 100 largest lifetimes in their analysis.  
Whereas their principal component analysis clearly reveal age effects, 
their permutation test based on the mean lifetimes for the male and females subjects
only shows a clear difference when considering $\mathrm{PD}_1$. 
Accordingly, when demonstrating the usefulness of $\mathrm{APF}_0$ and $\mathrm{APF}_1$, we will focus on the gender effect and  consider the same 95 trees as in \cite{steve:16} (two transsexual subjects are excluded) obtained from $46$ female
subjects and $49$ male subjects; in contrast to \cite{steve:16}, we consider all observed meanages and lifetimes. 
In accordance to the allocated time $T=25$, we need to redefine $\mathrm{APF}_1$ by 
\begin{equation}\label{e:(b)}
\mathrm{APF}_1(m)=\sum_{i=1}^n c_i l_i 1(m_i\le m,\,m_i+l_i/2\le T),\quad m\geq 0.
\end{equation}
For simplicity we use the same notation $\mathrm{APF}_1$ in \eqref{e:(a)} and \eqref{e:(b)}; although all methods and results in this paper will be presented with the definition \eqref{e:(a)} in mind, they apply as well when considering \eqref{e:(b)}. 

Finally, we write 
$\mathrm{APF}^F_k$ and $\mathrm{APF}^M_k$ to distinguish between APF's for females and males, respectively.

\section{A single accumulated persistence function} \label{s:conf for APFS}

There exists several constructions and
results on confidence sets for persistence diagrams when the aim is to separate topological signal from noise, see~\cite{fasy:etal:14}, \cite{chazal:fasy:14}, and the references therein.
Appendix~\ref{s:sep} and its accompanying Example~5 discuss 
the obvious idea of transforming such a confidence region into one for an accumulate persistence function, where the potential problem is that the bottleneck metric is used for persistence diagrams and this is not corresponding to closeness of APFs, cf.\ Section~\ref{s:1:APF}.
In this section we focus instead on spatial point process model assessment using APFs or more traditional tools.

Suppose a realization of a finite spatial point process $\bX_0$ has been observed and copies $\bX_1,\ldots,\bX_r$ have been simulated under a claimed model for $\bX_0$ 
so that the  joint distribution of $\bX_0,\bX_1,\ldots, \bX_r$ should be exchangeable. That is, for any permutation 
$(\sigma_0,\ldots,\sigma_r)$ of $(0,\ldots,r)$, $(\bX_{\sigma_0},\ldots, 
\bX_{\sigma_r})$ is claimed to be distributed as $(\bX_0,\ldots, \bX_r)$; e.g.\ this is the case if $\bX_0,\bX_1,\ldots, \bX_r$ are IID. This is a 
common situation for model assessment in spatial point process analysis when a distribution for $\bX_0$ has been specified (or estimated),
see 
e.g.\   
\cite{Baddeley:Rubak:Wolf:15} and \cite{moeller:waagepetersen:16}. 
Denote the $\mathrm{APF}_k$s for $\bX_0,\ldots, \bX_r$  by 
$A_0,\ldots,A_r$, respectively, and the null hypothesis that the joint distribution of  
$A_0,\ldots,A_r$ is exchangeable by $\Hh_0$. Adapting ideas from \cite{myllymaki:etal:16}, 
we will discuss how 
to construct a goodness-of-fit test for $\Hh_0$ based on 
a so-called global rank envelope 
 for 
 $A_0$; their usefulness will be demonstrated in Example~1.

In functional data analysis, to measure how extreme $A_0$ is in comparison 
to $A_1,\ldots,A_r$, a so-called depth function is used for ranking 
$A_0,\ldots,A_r$, see e.g.\ \cite{lopez:romo:09}.  
We suggest using a depth ordering called extreme rank in \cite{myllymaki:etal:16}:
Let $T>0$ be a 
user-specified parameter chosen such that it is the behaviour of $A_0(m)$ for $0\le m\le T$
which is of interest. 
 For $l=1,2,\ldots$, define the $l$-th bounding curves of $A_0,\ldots,A_r$ 
 by 
\[ A^{l}_{\mathrm{low} } ( m) = \min_{i = 0,\ldots, r} {} \hspace{-1.5mm} ^{l}  A_i(m)  \quad \mathrm{and} \quad   A^{l}_{\mathrm{upp} } ( m) =\max_{i = 0,\ldots, r} {} \hspace{-1.5mm} ^{l}  A_i(m), \quad 0\leq m \leq T, \]
where $ \min {} \hspace{-0.5mm} ^{l}$ and $ \max  {} \hspace{-0.5mm} ^{l}$ 
denote the $l$-th smallest  and largest values, respectively, and where $l\le r/2$.  Then, for $i=0,\ldots,r$, the extreme rank of $A_i$ 
with respect to $A_0,\ldots,A_r$  is 
\begin{equation*}
R_i = \max \left \lbrace l \ : \  A^{l}_{\mathrm{low} } ( m)  \leq A_i(m) \leq  A^{l}_{\mathrm{upp} } (m)  \quad  \mbox{for all } m \in [0,T]    \right \rbrace.
\end{equation*} 
The larger $R_i$ is, the deeper or more central $A_i$ is among  $A_0,\ldots,A_r$.

Now, for a given  $\alpha \in (0,1)$, the extreme rank ordering is used to define the $100(1-\alpha)\%$-global rank envelope as the band delimited by the curves $A^{l_\alpha}_{\mathrm{low} }$ and $A^{l_\alpha}_{\mathrm{upp} }$ where 
\begin{equation*}
l_\alpha = \max \left\lbrace l \ : \ \frac{1}{r+1} \sum_{i=0}^{r}  1 (  R_i <l ) \leq  \alpha    \right\rbrace.
\end{equation*}
Under $\Hh_0$,  with probability at least $ 1-\alpha$,
\begin{equation}\label{e:APFenvelope}
A^{l_\alpha}_{\mathrm{low} } (m)   \leq A_0(m) \leq A^{l_\alpha}_{\mathrm{upp} } (m) \quad  \mbox{for all } m \in [0,T],
\end{equation}
see~\cite{myllymaki:etal:16}. Therefore,  the $100(1-\alpha)\%$-global rank envelope is specifying a  
conservative statistical test called the extreme rank envelope test and which
accepts  $\Hh_0$ at level $100\alpha\%$ if  \eqref{e:APFenvelope} is satisfied or equivalently if 
\begin{align}\label{e:APFtest}
 \frac{1}{r+1} \sum_{i=0}^{r}  1 (  R_i <R_0 ) >\alpha,
\end{align}
cf.\ \cite{myllymaki:etal:16}.  
A plot of the extreme rank envelope allows a graphical interpretation of the 
extreme rank envelope test and may in case of rejection suggest an alternative  model for $\bX_0$.

There exist alternatives to the extreme rank envelope test,
in particular a liberal extreme rank envelope test and a so-called global scaled maximum absolute difference envelope, see \cite{myllymaki:etal:16}. 
It is also possible to combine several extreme rank envelopes, for instance by combining $\mathrm{APF}_0$ and $\mathrm{APF}_1$, see \cite{myllymaki:etal:16b}.
In the following example we focus on \eqref{e:APFenvelope}-\eqref{e:APFtest} and briefly remark on results obtained by combining $\mathrm{APF}_0$ and $\mathrm{APF}_1$.

\paragraph*{Example 1 (simulation study).} Recall that a homogeneous Poisson process is a model for 
complete spatial randomness (CSR), see e.g.\ \cite{moeller:waagepetersen:00} and the simulation in the 
first panel of Figure~\ref{fig:simEx2}. 
Consider APFs $A_0, A_1, \ldots, A_r$ corresponding 
to  independent point processes $\bfX_0, \bfX_1, \ldots, \bfX_r$ defined on a unit square and where $\bX_i$ for $i>0$ is CSR with a given intensity $\rho$ (the mean number of points). Suppose $\bfX_0$ 
is claimed to be CSR with intensity $\rho$, however, the model for $\bfX_0$ is given by one of the following four point process models, which we refer to as the true model:
\begin{enumerate} 
\item[(a)] CSR; hence the true model agrees with the claimed model. 
\item[(b)] A Baddeley-Silverman cell process; this has the same 
second-order moment properties as under CSR, see \cite{baddely:silverman:84}. Though from a 
mathematical point of view, it is a cluster process, simulated realisations 
will exhibit both aggregation and regularity at different 
scales, see the second panel of Figure~\ref{fig:simEx2}.  
\item[(c)] A Matérn cluster process; this is a model for clustering where each cluster is a homogenous Poisson process within a disc and the centers of the discs are not observed and constitute a stationary Poisson process, see \cite{matern:86}, \\ \cite{moeller:waagepetersen:00}, and the third panel of Figure~\ref{fig:simEx2}.
\item[(d)] A most repulsive Bessel-type determinantal point process (DPP); 
this is a model for regularity, see \cite{LMR15}, \cite{biscio:lavancier:16}, and the fourth panel of Figure~\ref{fig:simEx2}. 
\end{enumerate}
We let $\rho=100$ or $400$. This specifies completely the 
models in (a) and (d), whereas the remaining parameters in the cases (b)-(c) are defined to be the same as those 
used in~\cite{robins:turner:16}. In all cases of
Figure~\ref{fig:simEx2},  $\rho=400$. Finally, following the recommendation in~\cite{myllymaki:etal:16}, we let $r=2499$. 

\begin{figure}
\begin{center}
 \resizebox{\columnwidth}{!}{
\begin{tabular}{cccc} 
\includegraphics[scale=0.2]{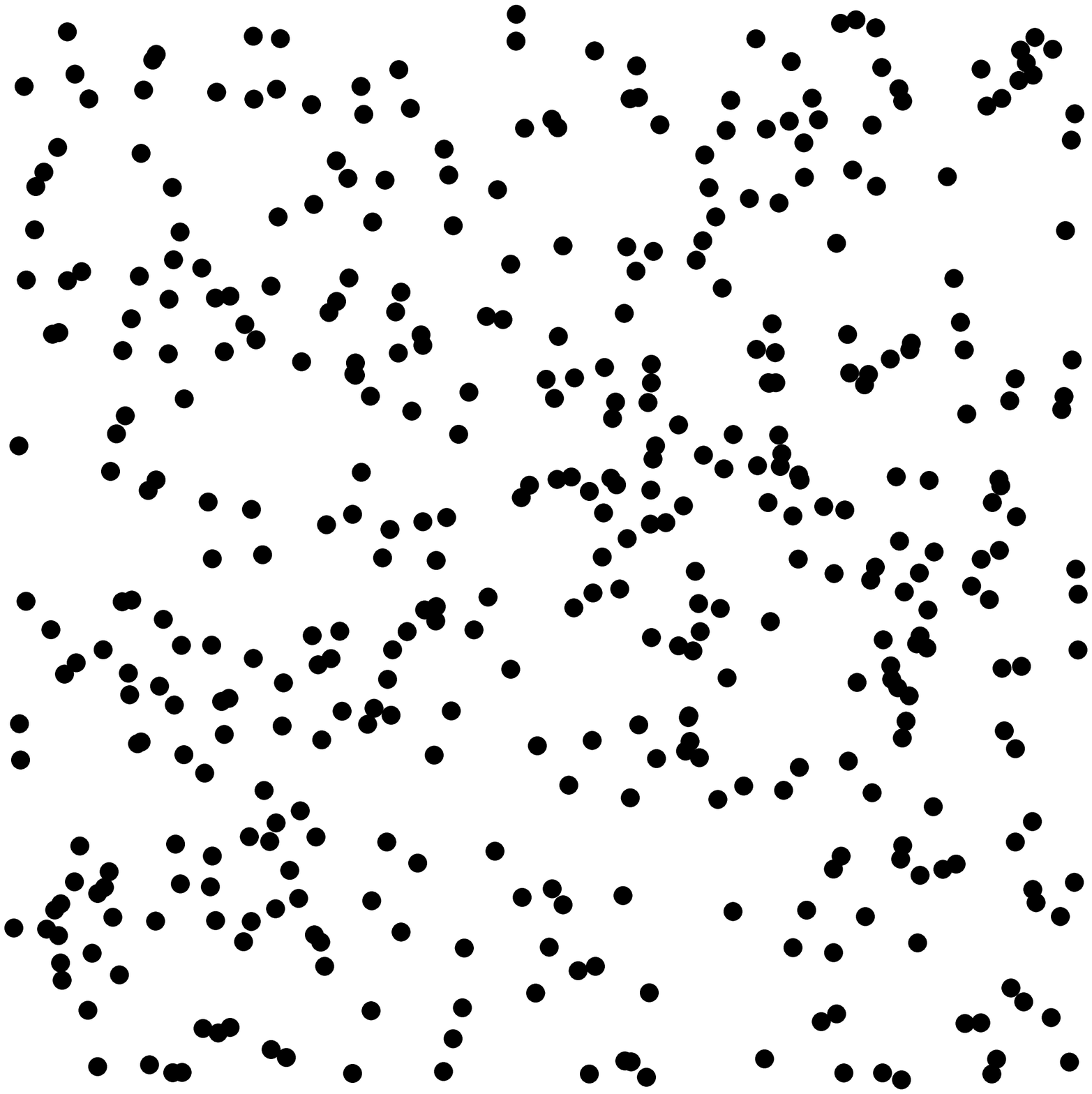}& \includegraphics[scale=0.2]{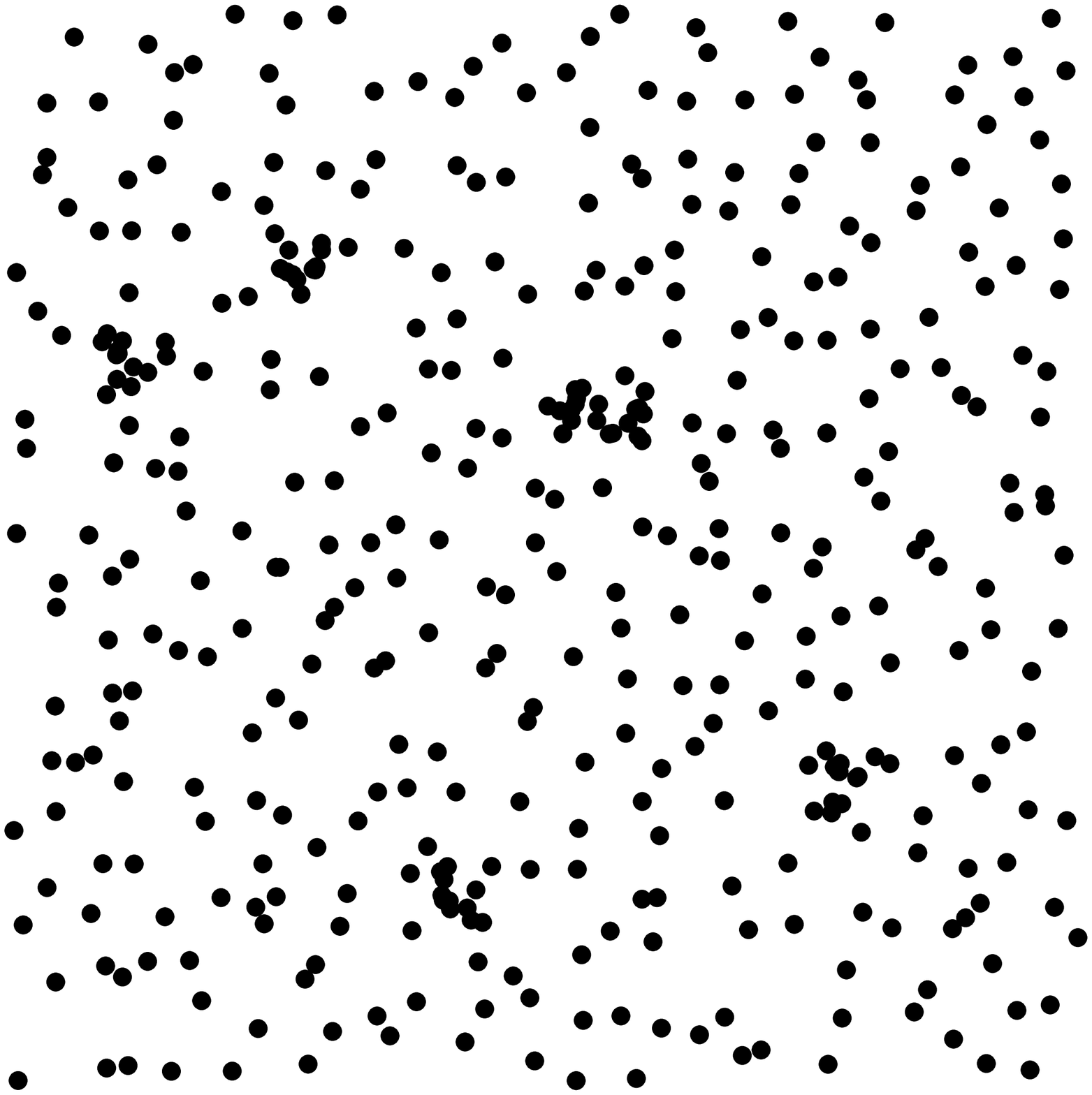} &
\includegraphics[scale=0.2]{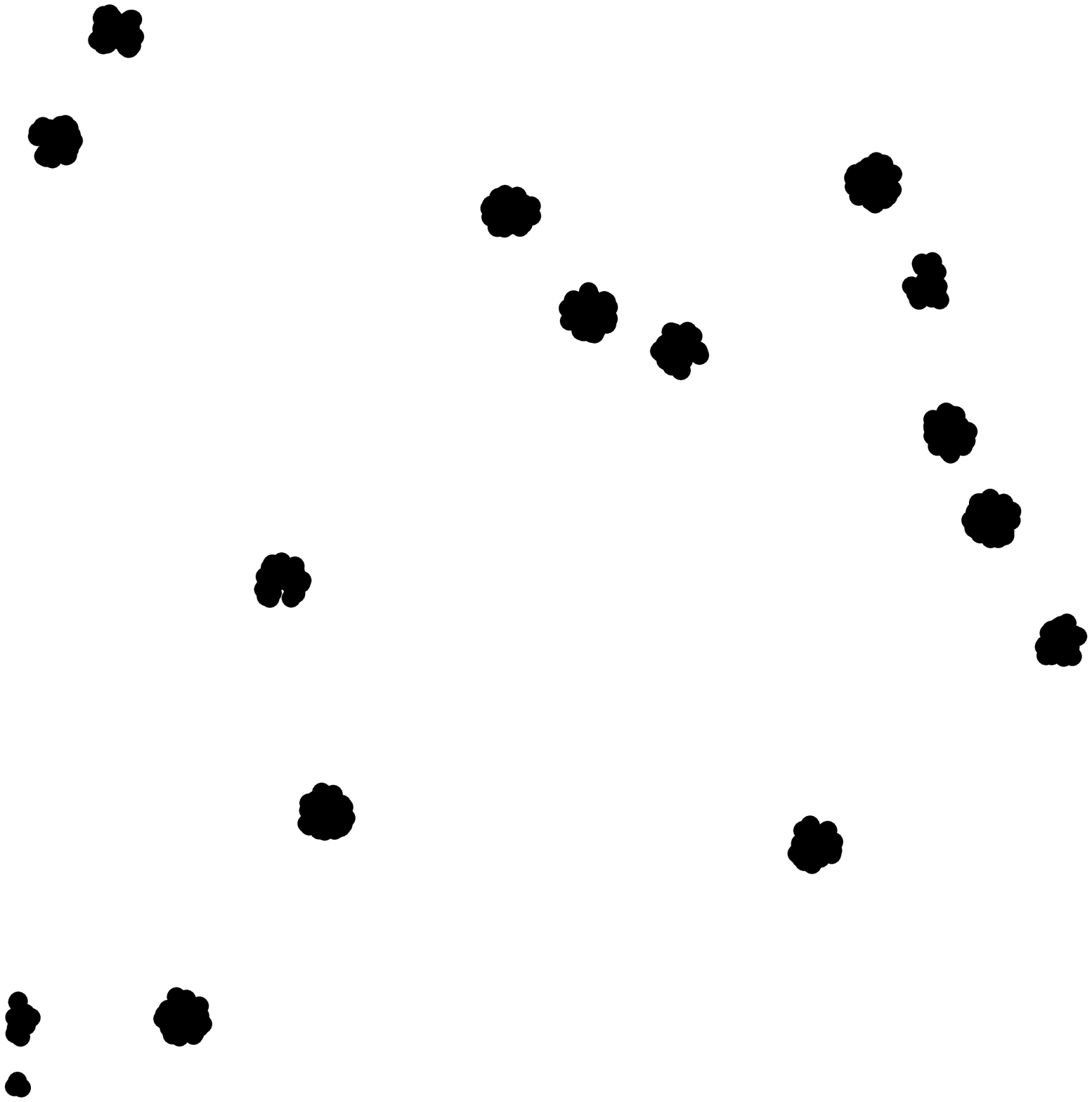}  & \includegraphics[scale=0.2]{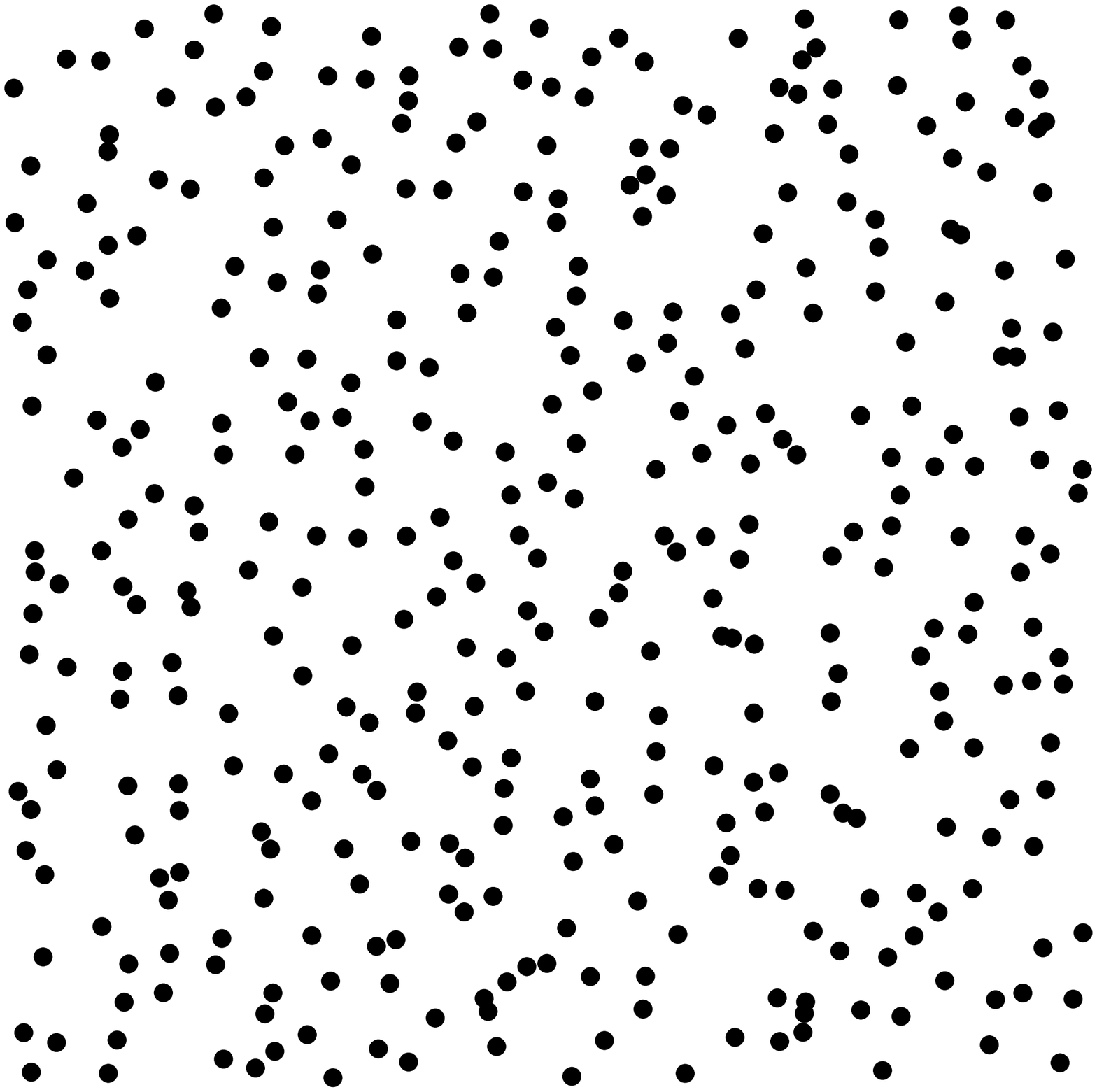}
\end{tabular} }
\caption{Simulated point patterns for a homogeneous Poisson process (first panel), 
a  Baddeley-Silverman cell process (second panel), 
a Matérn cluster process (third panel), 
and a most repulsive Bessel-type DPP (fourth panel).}\label{fig:simEx2}
\end{center}
\end{figure}

For each value of $\rho=100$ or $400$, we simulate each point process in (a)-(d)
 with the {\sf \bf R}-package 
{\tt spatstat}. Then, for each dimension $k=0$ or $1$, we compute the extreme rank envelopes 
and extreme rank envelope tests with the {\sf \bf R}-package {\tt spptest}. We repeat all this $500$ times. 
Table~\ref{table:reject_envelope_intensite}  shows for each case (a)-(d) the percentage of rejection of the hypothesis that $\bfX_0$ 
is a homogeneous Poisson process with known intensity $\rho$. 
In case of CSR,
the type one error of the test is small except when $k=0$ and $\rho=100$. As expected in case of (b)-(d), the power of the test is increased when $\rho$ is increased. 
For both the Baddeley-Silverman process and the DPP, when $k=0$ and/or $\rho=400$, the power is high and even 100\% in two cases.
For the Matérn cluster process, the power is 100\% when both $\rho=100$ and 400; 
this is also the case when instead the radius of a cluster becomes 10 times larger and hence it is not so easy to distinguish the clusters as in the third panel of Figure~\ref{fig:simEx2}. 
 When we combine the extreme rank envelopes for $\mathrm{APF}_0$ and $\mathrm{APF}_1$, the results are better or close to the best results obtained when considering only one extreme rank envelope.

  \begin{table}
\newcommand{\mc}[3]{\multicolumn{#1}{#2}{#3}}
\def\arraystretch{0.9}
 \resizebox{\columnwidth}{!}{
\begin{tabular}[c]{l|c c c c c c c c}
\mc{1}{c|}{}                          &  \mc{2}{c|}{CSR}       & \mc{2}{c|}{DPP} & \mc{2}{c|}{Matérn cluster} & \mc{2}{c|}{Baddeley-Silverman}   \\ 
\mc{1}{c|}{}                          &  $\rho=100$ & \mc{1}{c|}{$\rho=400$}       &  $\rho=100$ & \mc{1}{c|}{$\rho=400$}   &  $\rho=100$ & \mc{1}{c|}{$\rho=400$}  &  $\rho=100$ & \mc{1}{c|}{$\rho=400$}   \\ \cline{1-9}
\mc{1}{l|}{ $\mathrm{APF}_0$  }          						& 3.6  					&  \mc{1}{c|}{4}         		& 77.4                                      &  \mc{1}{c|}{100}              & 100  &  \mc{1}{c|}{100}        & 45.6  &  \mc{1}{c|}{99.6}  \\ 
\mc{1}{l|}{ $\mathrm{APF}_1$ }           						& 3.8 					&  \mc{1}{c|}{4.6}      		&  28.2 					&  \mc{1}{c|}{57.8}             & 100  &   \mc{1}{c|}{100}       & 65.8   &     \mc{1}{c|}{100} \\
\mc{1}{l|}{ $\mathrm{APF}_0$, $\mathrm{APF}_1$ }            & 4.8 &  \mc{1}{c|}{3.6}       &  82.4 &  \mc{1}{c|}{100}     & 100  &   \mc{1}{c|}{100}       & 60.8   &     \mc{1}{c|}{100}
\end{tabular}
 }
\caption{Percentage of point patterns for which the $95\%$-extreme rank envelope test rejects the hypothesis of CSR (a homogeneous Poisson process on the unit square with intensity $\rho=100$ or  $\rho=400$)
when the true model is either CSR or one of three alternative point process models.}\label{table:reject_envelope_intensite}
\end{table}

Figure~\ref{fig:graph:extreme_rank_test} illustrates for one of the $500$ repetitions and for each dimension $k=0$ and $k=1$
the deviation of $\mathrm{APF}_k$ from the extreme rank envelope obtained when the true model is not CSR. 
For each of the three non-CSR models, $\mathrm{APF}_k$ is outside the extreme rank envelope, in particular when $k=0$ and  
both the meanage and lifetime are small, cf.\ the middle panel. This means that 
small lifetimes are not noise but of particular importance, cf.\ the discussion  
in Section~\ref{s:1:background}. 
Using an obvious notation, for small $m$, we may expect that
$\mathrm{APF}^{\mathrm{DPP}}_0(m)<\mathrm{APF}^{\mathrm{CSR}}_0(m)<\mathrm{APF}^{\mathrm{MC}}_0(m)$ which is
in agreement with the middle panel. 
For large $m$, we may expect that
$\mathrm{APF}^{\mathrm{DPP}}_0(m)>\mathrm{APF}^{\mathrm{CSR}}_0(m)$ and $\mathrm{APF}^{\mathrm{CSR}}_0(m)>\mathrm{APF}^{\mathrm{MC}}_0(m)$, but only the last relation is detected by the extreme rank envelope in the left panel. 
Similarly, we may expect $\mathrm{APF}^{\mathrm{MC}}_1(m)> \mathrm{APF}^{\mathrm{CSR}}_1(m)$ for small $m$, whereas $\mathrm{APF}^{\mathrm{MC}}_1(m)< \mathrm{APF}^{\mathrm{CSR}}_1(m)$ 
for large $m$, and both cases 
are detected in the right panel. Note that 
for the Baddeley-Silverman cell process and $k=0,1$, $\mathrm{APF}^{\mathrm{BS}}_k$  has a rather similar behaviour as $\mathrm{APF}^{\mathrm{DPP}}_k$, 
 i.e.\ like a regular point process and probably because clustering is a rare phenomena.

\begin{figure}
\centering
\begin{tabular}{ccc} 
\includegraphics[scale=0.22]{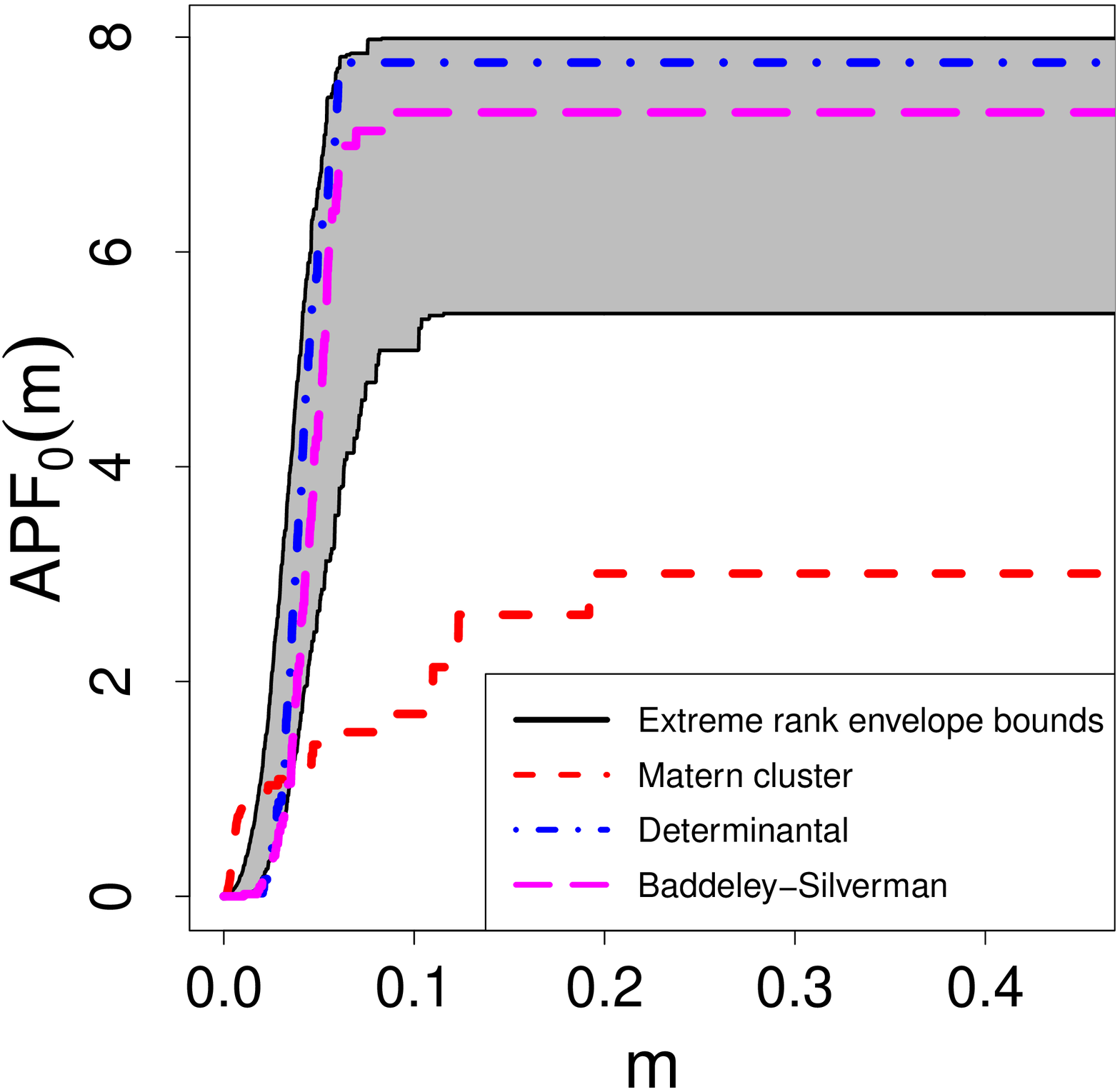}& \includegraphics[scale=0.22]{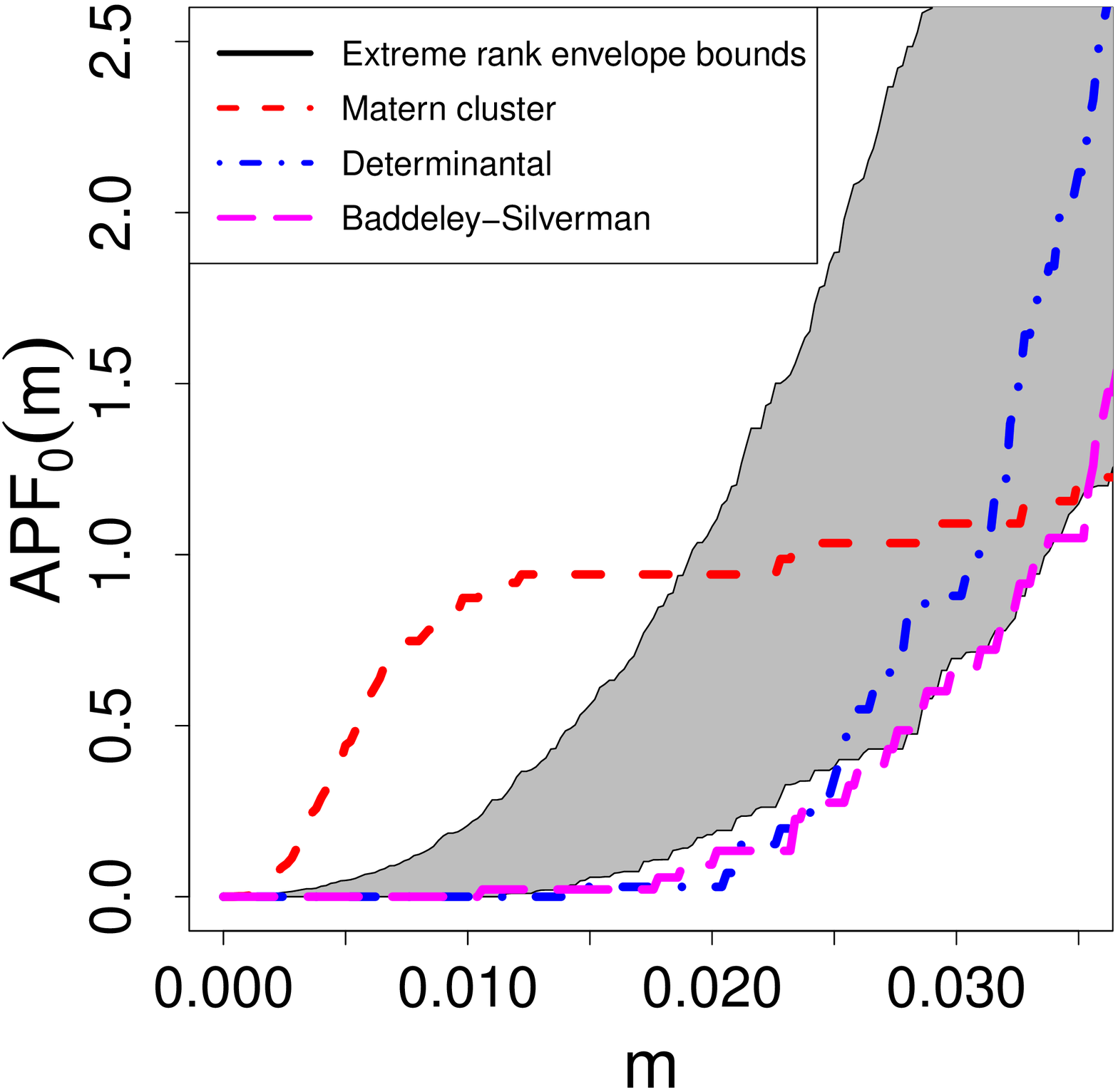}    & \includegraphics[scale=0.22]{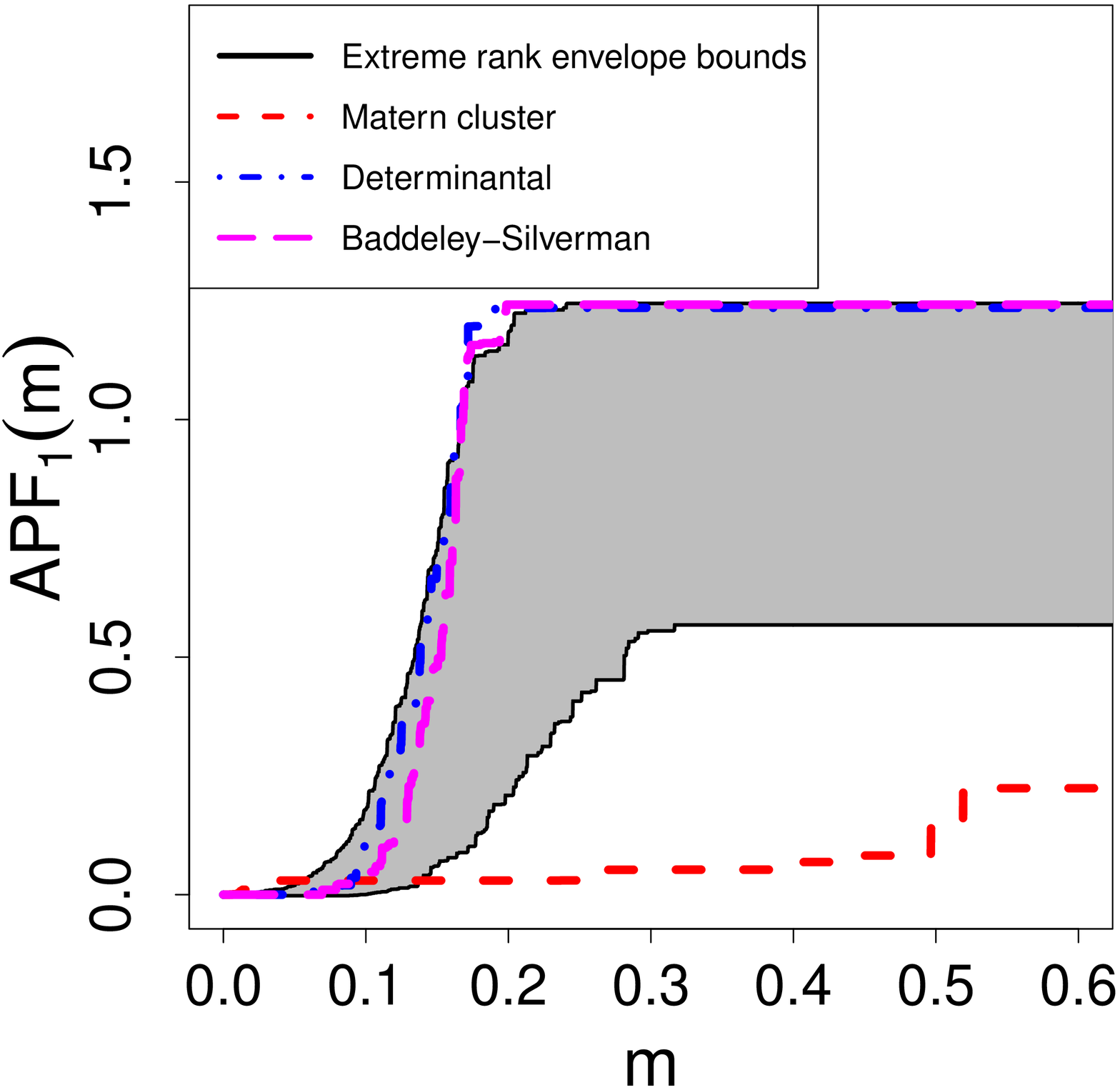} \end{tabular}
\caption{ $95\%$-extreme rank envelope for $\mathrm{APF}_i$ when $i=0$ (left panel and the enlargement shown in the middle panel) or $i=1$ (right panel) together with the curves for the three non-CSR models (Baddeley-Silverman cell process, Matérn cluster process, and Bessel-type DPP). The envelope is 
obtained from $2499$ realisations of a CSR model on the unit square and with intensity $100$.}\label{fig:graph:extreme_rank_test}
 \end{figure}       

A similar simulation study is discussed in~\cite{robins:turner:16} for the models in (a)-(c), but notice that they fix the number of points to be $100$ and 
they 
use a testing procedure based on the persistent homology rank function, which 
in contrast to our one-dimensional APF is a two-dimensional 
function and is not summarizing all the topological features represented in a persistent diagram.
\cite{robins:turner:16} show that a test for CSR based on the  persistent homology rank function is  
useful as compared to  various tests implemented in {\tt spatstat} 
and which only concern first and second-order moment properties.
Their method is in particular useful, when the true model is a  Baddeley-Silverman cell process 
with the same first and second-order moment properties as under CSR. 
Comparing Figure~4 in~\cite{robins:turner:16} with the results in Table~\ref{table:reject_envelope_intensite}  when the true model is a Baddeley-Silverman cell process and $\rho=100$, 
the extreme rank envelope test seems less powerful than the  
test they suggest.
On the other hand,~\cite{robins:turner:16} observe that
the latter test performs poorly when the true model is  a Strauss process (a model for inhibition) or a Matérn cluster process; 
as noticed for the Matérn cluster process, we obtain a perfect power when using the extreme rang envelope test.

\section{A single sample of accumulated persistence functions}\label{s:one sample}

\subsection{Functional boxplot} \label{s:2:functional boxplot}

This section discusses the use of a functional boxplot \citep{sun:genton:11} for a sample $A_1,\ldots,A_r$ of $\mathrm{APF}_k$s those joint distribution is exchangeable. The plot 
provides a representation of the variation of the curves given by $A_1,\ldots,A_r$ around the most central curve, and it can be used for outlier detection, i.e.\ detection of curves that are too extreme with respect to the others in the sample. 
This is illustrated in Example~2 for the brain artery trees dataset and in Appendix~\ref{s:appendix functional boxplot} and its accompanying Example~6 concerning a simulation study.

The functional boxplot is based on an 
ordering of the $\mathrm{APF}_k$s obtained using a so-called depth function. For specificity we make the standard choice called the modified band depth function (MBD), cf.\ \cite{lopez:romo:09} and \cite{sun:genton:11}: 
 For a user-specified parameter $T>0$ and $h,i,j=1,\ldots,r$ with $i<j$,  define 
 \[B_{h,i,j}= \left\lbrace m \in [0,T]: \ \min \left\lbrace A_i\left(  m  
\right), A_j\left(  m   \right) \right\rbrace    \leq  A_h(m) \leq \max 
\left\lbrace A_i\left(  m  \right), A_j\left(  m   \right) \right\rbrace    
\right\rbrace, \]
 and denote the Lebesgue measure on $[0,T]$ by   $\left| \cdot \right|$. Then the MBD of 
$A_h$ with respect to $A_1,\ldots,A_r$ is 
\begin{align}\label{e:definition MBD}
 \MBD_r(A_h )  = \frac{2}{r(r-1)} \sum_{ 1 \leq i < j \leq r }  \left|  B_{h,i,j}  \right| . 
\end{align}
This is the average proportion of   $A_h$ on $[0,T]$  between all possible pairs of  $A_1,\ldots,A_r$.  
Thus, the larger the value of  the MBD of a curve is, the more central or deeper it is in the sample. We call the region delimited by the $50\%$ 
 most central curves the central envelope. It is often assumed that a curve outside the central envelope inflated by $1.5$ times 
  the range of the central envelope is an outlier or abnormal curve ---  
  this is just a generalisation of a similar criterion for the boxplot of a sample of real numbers --- 
  and the range may be changed if it is more suitable for the application at hand,  see the discussion in~\cite{sun:genton:11} and Example~2 below.

\paragraph*{Example 2 (brain artery trees).} 

For the brain artery trees dataset (Section~\ref{s:brains}), Figure~\ref{fig:functional boxplot brain} shows
the functional boxplots of $\text{APF}_k$s for females (first and third panels) respective males (second and fourth panels) when $k=0$ (first and second panels) and $k=1$ (third and fourth panels): The most central curve is plotted in black, the central envelope in purple, and the upper and lower bounds obtained from all the curves except the outliers in dark blue. 
Comparing the two left panels (concerned with connected components), the shape of the central envelope is clearly different for females and males, in particular on the interval $[40,60]$, and
the upper and lower bounds of the non-outlier are closer to the central region for females, in particular on the interval $[0,50]$. For the two right panels (concerned with loops), the main difference 
is observed on the interval $[15,25]$ where the central envelope is larger for females than for males.

\begin{figure}
\centering
 \begin{tabular}{cccc}
\includegraphics[scale=0.17]{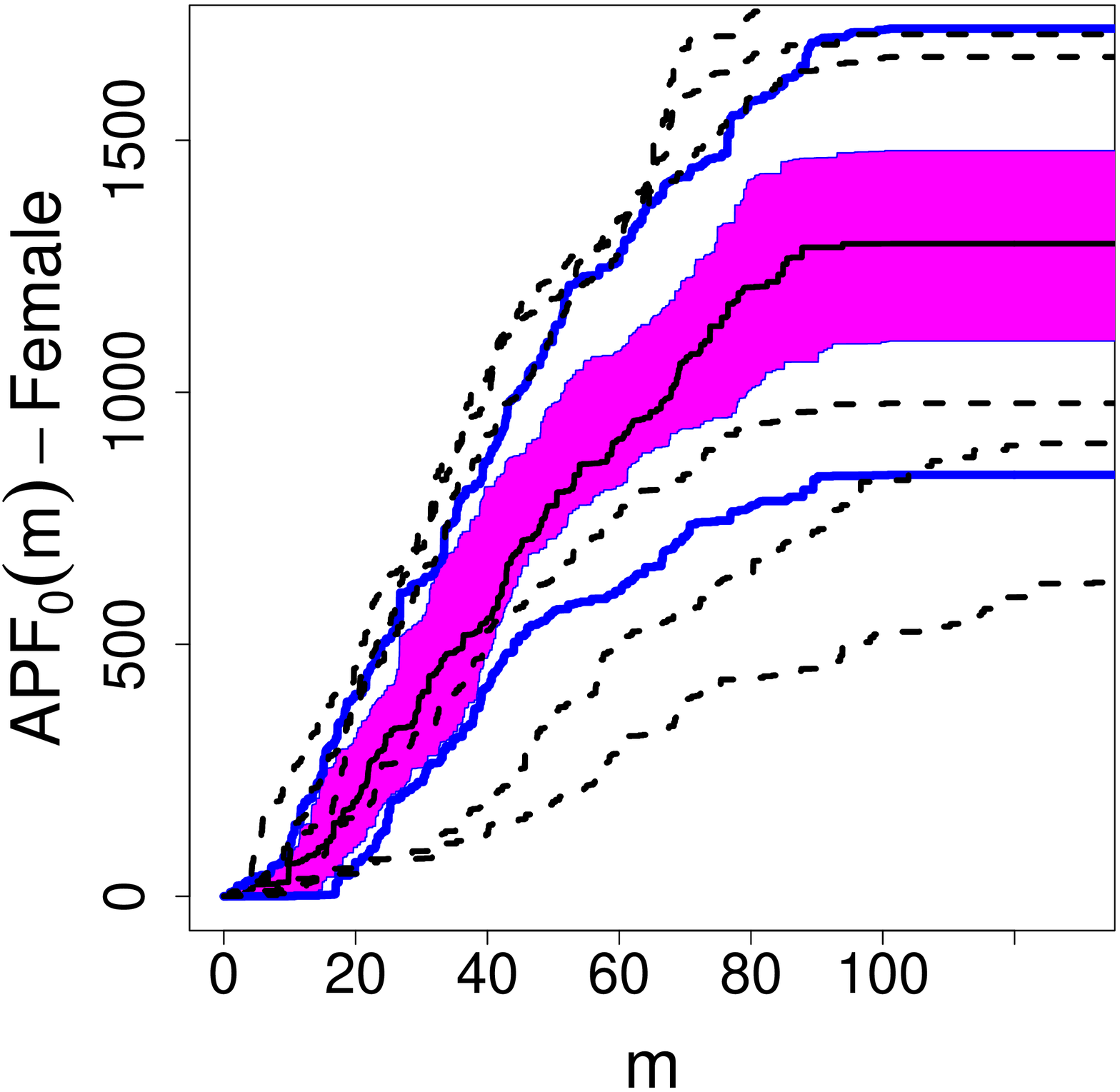} & \includegraphics[scale=0.17]{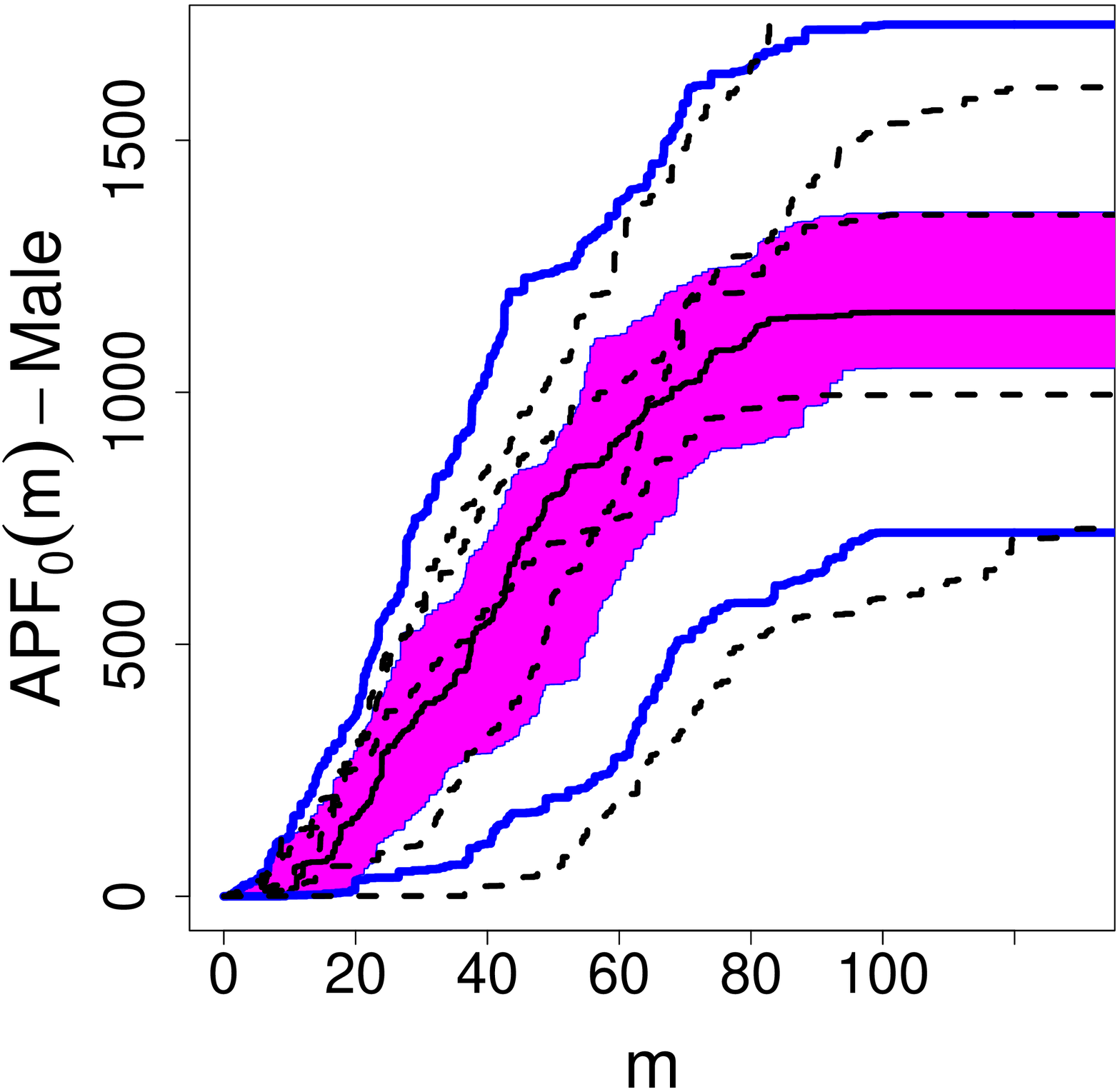}  
\includegraphics[scale=0.17]{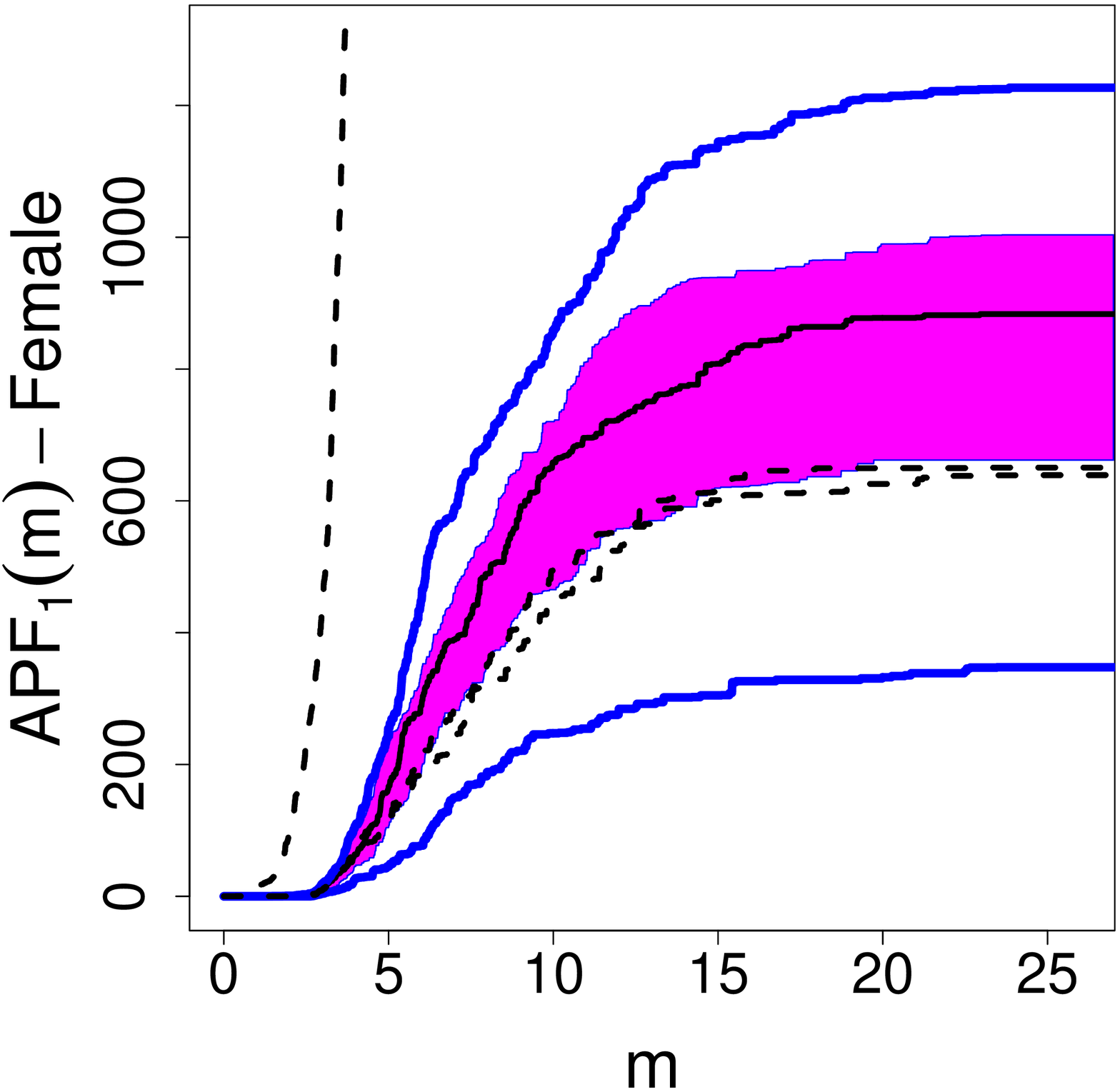} & \includegraphics[scale=0.17]{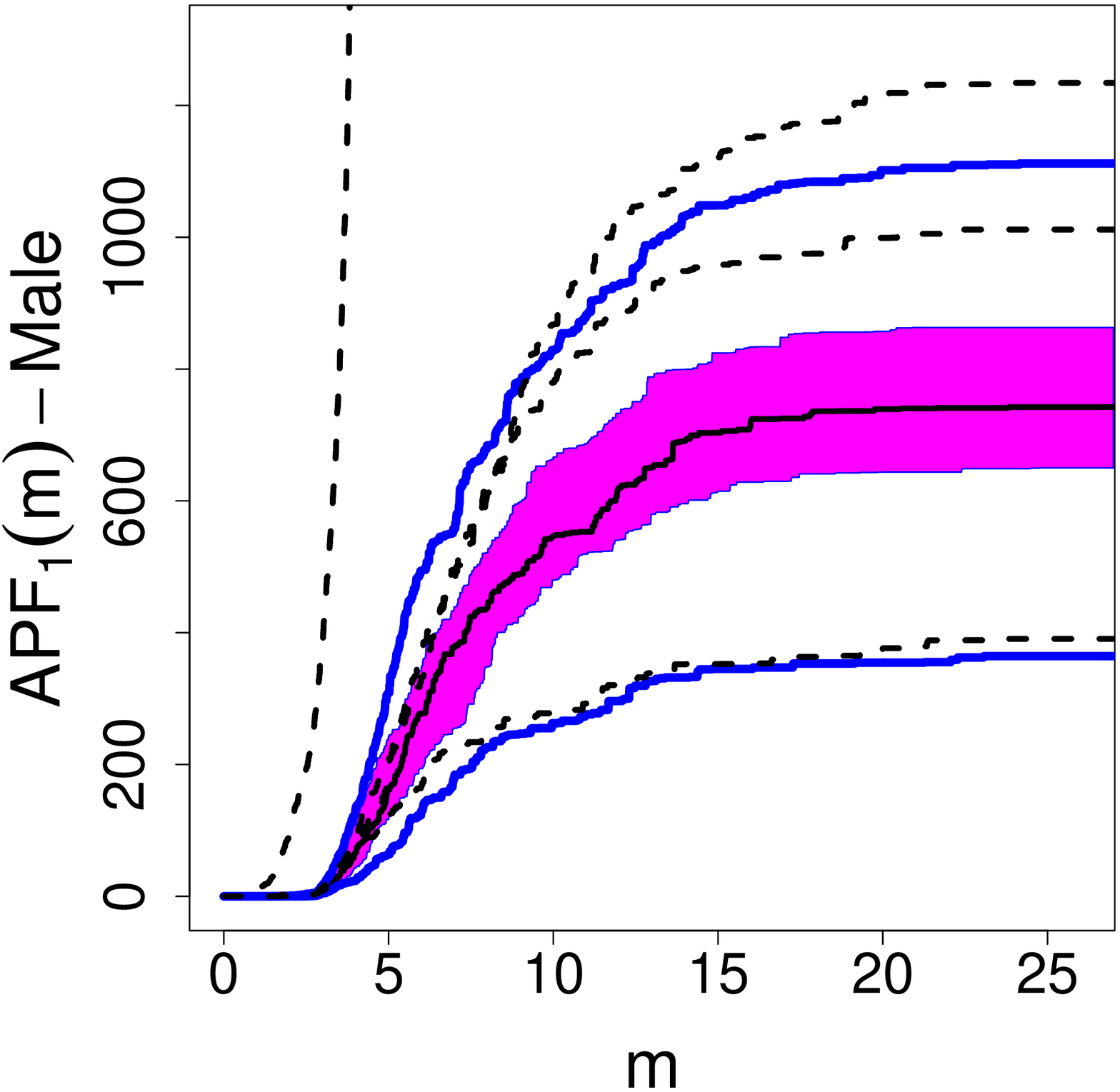}
 \end{tabular}
\caption{Functional boxplots of APFs for females and males obtained from the brain  
artery trees dataset: $\mathrm{APF}^F_0$ (first panel), $\mathrm{APF}^M_0$ (second panel), 
 $\mathrm{APF}^F_1$ (third panel), $\mathrm{APF}^M_1$ (fourth panel). 
 The dashed lines show the outliers detected by the 1.5 criterion.}\label{fig:functional boxplot brain}
 \end{figure}

The dashed lines in Figure~\ref{fig:functional boxplot brain} show the APFs detected as outliers by the 1.5 criterion, that is $6$ $\mathrm{APF}^F_0$s (first panel), $3$ $\mathrm{APF}^F_1$s (third panel), 
$6$ $\mathrm{APF}^M_0$s (second panel), and $4$ $\mathrm{APF}^M_1$s (fourth panel). 
For the females, only for one point pattern both $\mathrm{APF}^F_0$ and $\mathrm{APF}^F_1$ are outliers, where 
$\mathrm{APF}^F_1$ is the steep dashed line in the bottom-left panel; and
for the males, only for two point patterns
both $\mathrm{APF}^M_0$ and $\mathrm{APF}^M_1$ are outliers, where in one case
$\mathrm{APF}^M_1$ is the steep dashed line in the bottom-right panel. For this case, Figure~\ref{fig:functional boxplot brain - male outlier} reveals an obvious issue: A large part on the right 
of the corresponding tree
is missing!
 
Examples 3 and 4 discuss to what extent our analysis of the brain artery trees will 
  be sensitive to whether we include or exclude the detected outliers.

\begin{figure}
\begin{center}
\vspace*{0.8cm}
\includegraphics[scale=0.23]{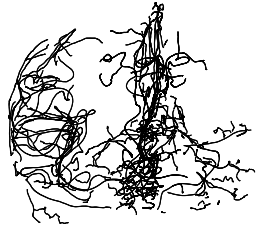}
\caption{Brain artery tree of a male subject with $\mathrm{APF}^M_0$ and $\mathrm{APF}^M_1$ 
detected as outliers by the 1.5 criterion.}\label{fig:functional boxplot brain - male outlier}
\end{center}
 \end{figure}

\subsection{Confidence region for the mean function}\label{s:confidence region mean}\label{s:crmean}

This section considers an asymptotic confidence region for the mean function 
of a sample $A_1,\ldots,A_r$ of IID $\mathrm{APF}_k$s. We assume that $D_1,\ldots,D_r$ are the underlying IID $\mathrm{RRPD}_k$s 
for the sample so that with probability one, there exists an upper bound $T<\infty$ on the death times
and there exists an upper bound  $n_{\mathrm{max}}<\infty$ on the 
number of $k$-dimensional topological features. Note that
the state space for such $\mathrm{RRPD}_k$s is 
\begin{align*}
\mathcal{D}_{k,T,n_{\mathrm{max}}}=\{\lbrace (m_1,l_1,c_1),\ldots,(m_n,l_n,c_n) \rbrace: \sum_{i=1}^n c_i \leq n_{\mathrm{max}},\, 
m_i + l_i/2 \leq T,\, i=1,\ldots,n\}
\end{align*} 
and
 only the existence and not the actual values of 
$n_{\mathrm{max}}$ and $T$ play a role when applying our method below. 
For example, in the settings (i)-(ii) of 
Section~\ref{s:1:simulated dataset}
it suffices to assume that $\mathbf X$ is included in a bounded region of $\R^2$ and that the number of points $N$ is bounded by a constant; 
this follows from the two versions of the Nerve Theorem presented in~\cite{fasy:etal:14}
and~\cite{Edelsbrunner:Harer:10}, respectively.

We adapt an empirical bootstrap procedure (see e.g.\ \cite{Vaart:Wellner:96}) which in  
\cite{chazal:etal:13} is used for a confidence region for the mean of the dominant function
of the persistent landscape and which in our case works as follows.  
 For $0\le m\le T$, the mean function is given by 
$\mu(m)=\mathrm E\left\{A_1(m)\right\}$ and estimated by the empirical 
mean function $\overline{A}_r(m) =\frac{1}{r}  \sum_{i=1}^r  A_{i}(m)$. 
Let $A^*_1,\ldots,A^*_r$ be independent uniform draws with replacement from the set $\{A_1,\ldots, A_r\}$
and set $\overline{A^*_r}=\frac{1}{r}  \sum_{i=1}^r  A^*_{i}$
and $\theta^*=\sup_{m\in[0,T]} \sqrt{r} \left|\overline{A_r}(m)- \overline{A^*_r}(m)\right|$.  
For a given integer $B>0$, independently repeat this procedure $B$ times to obtain  $\theta_1^*,\ldots,\theta_B^*$. 
Then, for $0<\alpha<1$,  the $100(1-\alpha)\%$-quantile in the distribution of $\theta^*$ is estimated 
by 
 \begin{align*}
  \hat{q}^B_\alpha = \inf \lbrace q\geq 0: \, \frac{1}{B} \sum_{i=1}^B 1( \theta_i^*>q) \leq \alpha \rbrace.
 \end{align*}
 The following theorem is verified in Appendix~\ref{s:proof of theorem confidence region}. 
 
\vspace{0.5cm}

\begin{theorem}\label{t:1}
Let the situation be as described above.
For large values of $r$ and $B$,  the functions $\overline{A_r} \pm \hat{q}^B_\alpha/\sqrt r$
provide the bounds for an asymptotic conservative $100(1-\alpha)\%$-confidence region for the mean APF, that is
\begin{equation*}
 \lim_{r \ra \infty} \lim_{B\ra \infty} \mathrm P\left( 
     \mu(m) \in [\overline{A_r}(m) - \hat{q}_\alpha^B / \sqrt{r}, \overline{A_r}(m) + \hat{q}_\alpha^B / \sqrt{r}] \mbox{ for all } m \in [0,T]  \right)  \geq 1-\alpha.
\end{equation*}
\end{theorem}

 \paragraph*{Example 3 (brain artery trees).} 

 The brain artery trees are all contained in a bounded region 
 and presented by a bounded number of points, 
 so it is obvious that $T$ and $n_{\mathrm{max}}$ exist for $k=0,1$. 
To establish confidence regions for the mean of the
$\mathrm{APF}^M_k$s respective $\mathrm{APF}^F_k$s, we  apply the bootstrap procedure with $B=1000$.
The result is shown in Figure~\ref{fig:confidence region brain} when all 95 trees 
are considered: 
In the left panel, $k=0$ and approximatively half of each confidence region overlap with the other confidence region; it is not clear if there is a difference between genders. In the right panel, $k=1$ and the difference is more pronounced, in particular on the interval $[15,25]$. 
Similar results and conclusions are obtained 
if we exclude the APFs detected as outliers in Example~2. 
Of course we should supply with a statistical test to assess the gender effect and such a test is established in Section~\ref{s:two sample problem} and applied in Example~4.

Appendix~\ref{s:appendix crmean}  provides the additional Example~7 for a simulated dataset along with a discussion on the geometrical interpretation of the confidence region obtained.

 \begin{figure}
\centering 
\begin{tabular}{cc}
\includegraphics[scale=0.2]{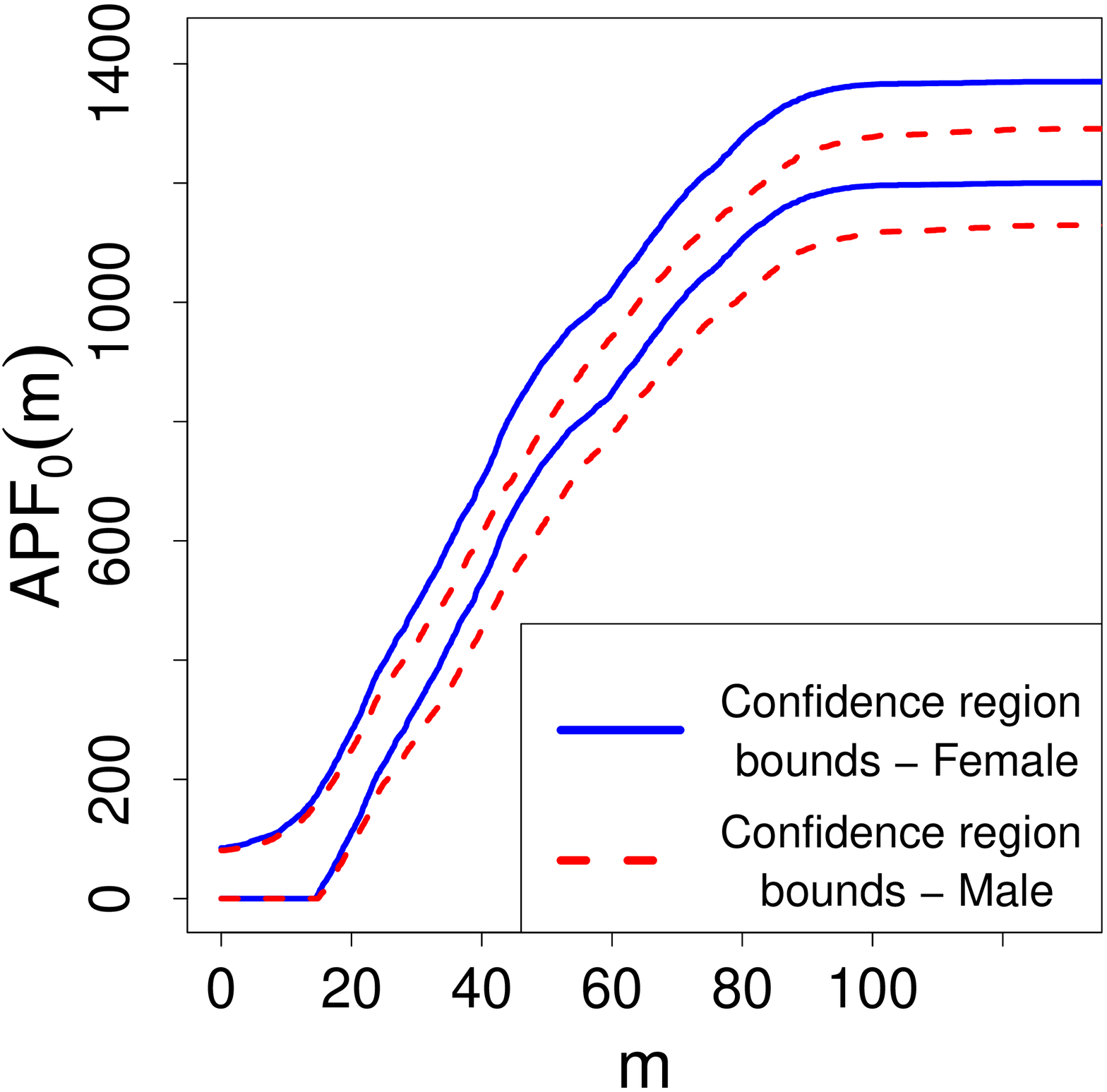} & \includegraphics[scale=0.2]{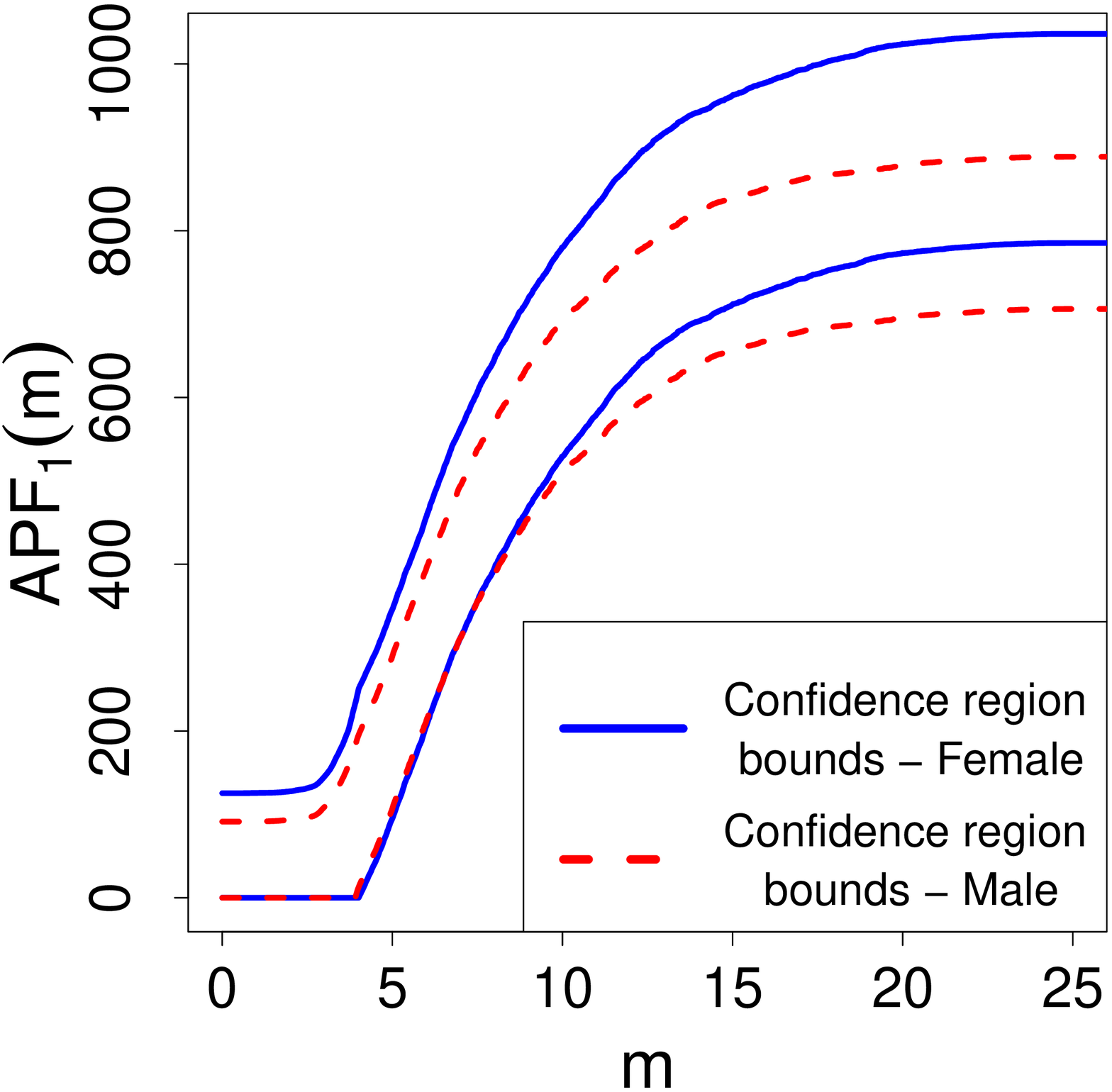}
\end{tabular}
\caption{
Bootstrap confidence regions for the mean $\mathrm{APF}^M_k$ and the mean $\mathrm{APF}^F_k$
when $k=0$ (left panel) and $k=1$ (right panel).}\label{fig:confidence region brain}
 \end{figure}

\section{Two samples of  accumulated persistence functions}\label{s:two sample problem}
 
This section concerns a two-sample test for comparison of two samples of APFs. 
Appendix~\ref{s:group of APFs} presents both a clustering method (Appendix~\ref{s:clusteringappendix}, including Example~9) and a unsupervised classification method (Appendix~\ref{s:supervised classification}, including Example~10) for two or more samples.

Consider two samples of independent $\mathrm{RRPD}_k$s $D_1,\ldots,D_{r_1}$ and $E_1,\ldots,E_{r_2}$, 
where each $D_i$ ($i=1,\ldots,r_1$) has distribution $\mathrm{P}_D$ 
and each $E_j$ has distribution $\mathrm{P}_E$ ($j=1,\ldots,r_2$), and suppose we  
want to test the null hypothesis $\mathcal{H}_0$: $\mathrm{P}_D = \mathrm{P}_E=\mathrm{P}$. 
Here, the common distribution $\mathrm P$ is unknown and as in Section~\ref{s:confidence region mean} we assume it is concentrated 
on $\mathcal{D}_{k,T,n_{\mathrm{max}}}$ for some integer $n_{\mathrm{max}}>0$ and number $T>0$. 
Below, we adapt a two-sample test statistic studied in~\cite{praestgaard:95} and \cite{Vaart:Wellner:96}.

Let $r = r_1+r_2$. Let $A_1,\ldots,A_r$ be the $\mathrm{APF}_k$ corresponding to $( D_1, \ldots,D_{r_1},E_1,\ldots,E_{r_2} )$, and denote by $\overline{ A_{r_1}}$ and $\overline{ A_{r_2}}$ 
 the empirical means of $A_1,\ldots, A_{r_1}$ and $A_{r_1+1},\ldots, A_{r_1+r_2}$, respectively.
Let $I=[T_1,T_2]$ be a user-specified interval with $0\le T_1<T_2\le T$ and used for
defining a two-sample test statistic by
\begin{align}\label{e:KSfirst} 
  KS_{r_1,r_2} &= \sqrt{\frac{r_1 r_2}{r}} \sup_{ m \in I} \left|\overline{ A_{r_1}} (m) - \overline{ A_{r_2}} (m) \right|,
\end{align}
where large values are critical for $\mathcal{H}_0$.
This may be rewritten as 
\begin{align} \label{e:stat KS alternative}
  KS_{r_1,r_2} = \sup_{ m \in I} \left|\sqrt{\frac{r_2}{r}} G^{r_1}_D(m) - \sqrt{\frac{r_1}{r}}G^{r_2}_E(m) + \sqrt{\frac{r_1 r_2}{r}} \mathrm E \left\{ A_D-A_E \right\}(m) \right|,
\end{align}
where  $G_D^{r_1} = \sqrt{r_1} \left( \overline{ A_{r_1}} - \mathrm E\left\{A_D\right\} \right)$ 
and $G_E^{r_2} = \sqrt{r_2} \left( \overline{ A_{r_2}} - \mathrm E\left\{A_E\right\}\right)$.  
By Lemma~\ref{lemma APF donsker} 
in Appendix~\ref{s:proof of theorem confidence region} and by the independence of the samples, $G_D^{r_1}$ and $G_E^{r_2}$ 
converge in distribution   to two independent zero-mean Gaussian processes on $ I$, denoted $G_D$ and $G_E$,  respectively.   
Assume that $ r_1/r \ra \lambda\in (0,1)$ as $r \ra \infty$. 
 Under 
$\mathcal{H}_0$, in the sense of convergence in distribution,
\begin{align}\label{e:stat KS convergence}
\lim_{r \ra \infty}  KS_{r_1,r_2} = \sup_{ m \in I} \left|\sqrt{1-\lambda} G_D(m) - \sqrt{\lambda}G_E(m) \right|,
\end{align}
where $\sqrt{1-\lambda} G_D - \sqrt{\lambda}G_E $ follows the same distribution as $G_D$. 
If $\mathcal{H}_0$ is not true and $\sup_{ m \in I} \left| \mathrm E \left\{ A_1-A_{r_1+1}\right\}(m)\right|  >0$, then 
$KS_{r_1,r_2}\ra\infty$ as $r \ra \infty$, see~\cite{Vaart:Wellner:96}. 
Therefore, for $0<\alpha<1$ and letting $q_\alpha = \inf \lbrace q: \, \mathrm P( \sup_{ m \in I} \left| G_D(m) \right| > q ) \leq \alpha \rbrace$,
the asymptotic test that rejects $\mathcal{H}_0$ if $KS_{r_1,r_2} \le q_\alpha$
is of level $100\alpha\%$  and of power $100\%$. 

As $q_\alpha $ depends on the unknown distribution $\mathrm P$, we
estimate $q_\alpha $ by  a bootstrap method:  
Let  $A^*_1,\ldots, A^*_r$  be independent uniform draws with replacement from $\lbrace A_1,\ldots, A_r \rbrace$.   
For $0\le m\le T$, define the empirical mean functions 
$ \overline{  A^*_{r_1}} (m) = \frac{1}{r_1} \sum_{i=1}^{r_1} A^*_{i} (m)$ and  
$ \overline{  A^*_{r_2}} (m) = \frac{1}{r_2} \sum_{i=r_1+1}^{r_1+r_2} A^*_{i} (m)$,  
and compute 
\begin{align}
 \theta^*  &= \sqrt{\frac{r_1 r_2}{r}} \sup_{ m \in I} \left|\overline{ A^*_{r_1}} (m) - \overline{ A^*_{r_2}} (m) \right|.  \label{critical value KS} 
\end{align}
For a given integer $B>0$, independently repeat this procedure $B$ times to obtain $\theta_1^*,\ldots,\theta_B^*$.  
Then we  estimate $q_\alpha $ by the $100(1-\alpha)\%$-quantile of the empirical distribution of $\theta_1^*,\ldots,\theta_B^*$, that is
\begin{align*}
  \hat{q}^B_\alpha = \inf \lbrace q\geq 0: \, \frac{1}{B} \sum_{i=1}^B 1( \theta_i^*>q) \leq \alpha \rbrace.
 \end{align*}

The next theorem is a direct application of Theorem 3.7.7 in~\cite{Vaart:Wellner:96}  noticing that the $\mathrm{APF}_k$s are uniformly bounded by $Tn_{\mathrm{max}}$ and they
form a so-called Donsker class, see Lemma~\ref{lemma APF donsker} 
and its proof in Appendix~\ref{s:proof of theorem confidence region}.

\vspace{0.5cm}

\begin{theorem}\label{th convergence KS}
 Let the situation be as described above. 
If $r \rightarrow\infty$ such that $ r_1/r \ra \lambda$ with $\lambda\in (0,1)$, 
then under $\mathcal{H}_0$
\begin{align*}
 \lim_{r \rightarrow \infty} \lim_{B\ra \infty} \mathrm P\left( KS_{r_1,r_2}> {\hat{q}^B_\alpha}  \right) = \alpha,
\end{align*}
whilst if $\mathcal{H}_0$ is not true and  $\sup_{ m \in I} \left| \mathrm E \left\{ A_{1}-A_{r_1+1}\right\}(m)\right|  >0$, then
\begin{align*}
 \lim_{r \rightarrow \infty} \lim_{B\ra \infty} \mathrm P\left( KS_{r_1,r_2}> {\hat{q}^B_\alpha}  \right) = 1.
\end{align*}
\end{theorem}

Therefore,  the test that rejects $\mathcal{H}_0$ 
if $KS_{r_1,r_2}  > \hat{q}^B_\alpha  $ is of asymptotic level $100\alpha\%$ and power $100\%$.
As remarked in~\cite{Vaart:Wellner:96}, by their Theorem~3.7.2 it is possible to present a permutation two-sample test so that the critical value $\hat q^B_\alpha$ for the bootstrap two-sample test has the same asymptotic properties as the critical value for the permutation test.

Other two-sample test statistics than \eqref{e:KSfirst} can be constructed by considering
other measurable functions of $\overline{ A_{r_1}}-\overline{ A_{r_2}}$,
e.g.\ we may consider the two-sample test statistic 
\begin{align}
 M_{r_1,r_2} &= \int_I \left| \overline{ A_{r_1}} (m) - \overline{ A_{r_2}} (m)  \right| \,\mathrm dm \label{critical value L1}.
\end{align}
Then by similar arguments as above but redefining $\theta^*$ in \eqref{critical value KS} by 
\begin{align*}
 \theta^*  &= \sqrt{\frac{r_1 r_2}{r}} \int_{ m \in I} \left|\overline{ A^*_{r_1}} (m) - \overline{ A^*_{r_2}} (m) \right|\,\mathrm dm,  
\end{align*} 
the test that rejects $\mathcal{H}_0$  
if $ M_{r_1,r_2}>\hat{q}^B_\alpha$ is of asymptotic level $100\alpha\%$ and power $100\%$.

\paragraph*{Example 4 (brain artery trees).}
 
To distinguish between male and female subjects of the brain artery trees dataset, we use the two-sample test statistic $KS_{r_1,r_2}$  
under three different settings: 
\begin{enumerate}
 \item[(A)] For $k=0,1$, we 
 let $\mathrm{PD}_k'$ be the subset of $\mathrm{PD}_k$ corresponding to the $100$ largest lifetimes.
 Then $D_1,\ldots,D_{46}$  and $E_1,\ldots,E_{49}$  are the $\mathrm{RRPD}_k$s  obtained from the
 $\mathrm{PD}_k'$s associated to female and male subjects, respectively. This is the setting used in \cite{steve:16}.
 \item[(B)] For $k=0,1$, we consider all lifetimes and let $D_1,\ldots,D_{46}$  and $E_1,\ldots,E_{49}$ be the  $\mathrm{RRPD}_k$s
  associated to female and male subjects, respectively.
 \item[(C)] The samples are as in setting (B) except that we exclude the $\mathrm{RRPD}_k$s
 where the corresponding $\mathrm{APF}_k$ was detected as an outlier in Example~2.
 Hence, $r_1=40$ and $r_2=43$ if $k=0$, and $r_1=43$ and $r_2=45$ if $k=1$.
\end{enumerate}

\cite{steve:16} perform a permutation test based on the mean lifetimes for the male and female subjects and
 conclude that gender effect is recognized when considering $\mathrm{PD}_1$ ($p$-value $=3\%$) but not $\mathrm{PD}_0$ ($p$-value $=10\%$). For comparison, 
under each setting (A)-(C), we perform the two-sample test for $k=0,1$, different intervals $I$, and $B=10000$. In each case, we estimate the $p$-value, 
i.e.\ the smallest $\alpha$ such that the two-sample test with significance level $100\alpha\%$ does not reject $\mathcal{H}_0$, 
by $\hat p=\frac{1}{B}\sum_{i=1}^B 1( \theta_i^* > KS_{r_1,r_2})$. 
Table~\ref{table:gender two sample test} shows the results. Under each setting (A)-(C), using $\mathrm{APF}_0$ we have a smaller $p$-value than in \cite{steve:16} if $I=[0,137]$ and an even larger $p$-value if $I=[0,60]$; and for $k=1$ under setting~(B), our $p$-value is about seven times larger than the $p$-value in \cite{steve:16} if $I=[0,25]$, and else it is similar or smaller. For $k=1$ and $I=[0,25]$, the large difference between our $p$-values under settings (B) and (C) indicates 
that the presence of outliers violates the result of Theorem~\ref{th convergence KS} and care should hence be taken. In our opinion we can better trust the results without outliers, where  in contrast to \cite{steve:16} we see a clear gender effect when considering the connected components. Notice also that in agreement with the discussion of Figure~\ref{fig:functional boxplot brain} in Example~2, for each setting A, B, and C and each dimension $k=0,1$, the $p$-values in Table~\ref{table:gender two sample test} are smallest when considering the smaller interval $I=[0,60]$ or $I=[15,25]$.  

Appendix~\ref{s:appendix two sample test} provides an additional Example~8 illustrating the use of two-sample test in a simulation study. 

  \begin{table}
\newcommand{\mc}[3]{\multicolumn{#1}{#2}{#3}}
\def\arraystretch{1}
\begin{center}
\begin{tabular}[c]{l|c c c c}
\mc{1}{c|}{}                          &  \mc{2}{c|}{$\mathrm{APF}_0$}       & \mc{2}{c|}{$\mathrm{APF}_1$}    \\ 
\mc{1}{c|}{}                          &  $I=[0,137]$ & \mc{1}{c|}{$I=[0,60]$}       &  $I=[0,25]$ & \mc{1}{c|}{$I=[15,25]$}    \\ \cline{1-5}
\mc{1}{l|}{Setting (A)}      & 5.26        &  \mc{1}{c|}{ 3.26}               & 3.18  &  \mc{1}{c|}{ 2.72 }      \\ 
\mc{1}{l|}{Setting (B)}      & 7.67         &  \mc{1}{c|}{ 3.64}               & 20.06  &  \mc{1}{c|}{ 1.83 }      \\ 
\mc{1}{l|}{Setting (C)}      & 4.55         &  \mc{1}{c|}{ 2.61}             &  0.92 &  \mc{1}{c|}{0.85}  
\end{tabular}
\end{center}
\caption{Estimated $p$-values given in percentage of the two-sample test based on $KS_{r_1,r_2}$  
used with $\mathrm{APF}_0$ and $\mathrm{APF}_1$ on different intervals $I$ to distinguish
between male and female subjects under settings (A), (B), and (C) described in Example 4.}\label{table:gender two sample test}
\end{table}

\subsubsection*{Acknowledgements}

Supported 
by The Danish Council for Independent Research | Natural Sciences, grant 7014-00074B, "Statistics for point processes in space and beyond", and 
by the "Centre for Stochastic Geometry and Advanced Bioimaging",
funded by grant 8721 from the Villum Foundation. Helpful discussions with Lisbeth Fajstrup on persistence homology is acknowledged. In connection to the brain artery trees dataset we thank James Stephen Marron and Sean Skwerer for helpful discussions and
the CASILab at The University of North Carolina at Chapel Hill for providing the data
distributed by the MIDAS Data Server at Kitware, Inc. We are grateful to the editors and the referees for useful comments.

\appendix
\section*{Appendix}
 
Appendix~\ref{s:sep}-\ref{s:proof of theorem confidence region} contain complements and additional examples to Sections~\ref{s:conf for APFS}-\ref{s:two sample problem}. 
Our setting and notation are as follows. 
All the examples are based on a simulated point pattern $\{x_1,\ldots,x_N\}\subset \R^2$ as described in Section~\ref{s:1:simulated dataset}, with 
$x_1,\ldots,x_N$ being IID points where
$N$ is a fixed positive integer.
As in Section~\ref{s:toy},
our setting corresponds to applications typically considered in TDA where the aim is to obtain topological information about a compact set $C\subset \R^2$ which is unobserved and where 
possibly noise appears: For specificity, we let $x_i=y_i+\epsilon_i$, $i=1,\ldots,N$, where 
$y_1,\ldots,y_N$ are IID points with support $C$, the noise $\epsilon_1,\ldots,\epsilon_N$ are IID and independent of $y_1,\ldots,y_N$, and $\epsilon_i$ follows the restriction to the square $[-10\sigma,10\sigma]^2$ of a bivariate zero-mean normal distribution with IID coordinates and 
standard deviation $\sigma\ge0$ (if $\sigma=0$ there is no noise).
We denote this distribution for $\epsilon_i$ by $N_2(\sigma)$ (the restriction to $[-10\sigma,10\sigma]^2$ is only imposed for technical reasons and is not of practical importance). 
We let $C_t$ be the union of closed discs of radii $t$ and centred at $x_1,\ldots,x_N$, 
and we study how the topological features of $C_t$ changes as $t\ge0$ grows. For this we use the Delaunay-complex mentioned in Section~\ref{s:toy}. 
Finally, we denote by $\mathcal{C}((a,b),r)$ the circle with center $(a,b)$ and radius $r$.

\section{Transforming confidence regions for persistence diagrams used for
 separating topological signal from noise}\label{s:sep}

As noted in Section~\ref{s:conf for APFS}
there exists several constructions and
results on confidence sets for persistence diagrams when the aim is to separate topological signal from noise, see~\cite{fasy:etal:14}, \cite{chazal:fasy:14}, and the references therein. 
We avoid presenting the technical description of these constructions and results, which depend on 
different choices of complexes (or more precisely so-called filtrations). For specificity, in this appendix we just consider the Delaunay-complex and discuss  the transformation of such a confidence region into one for an accumulate persistence function.

We use the following notation. As in the aforementioned references, consider the persistence 
diagram $\mathrm{PD}_k$ for an unobserved compact manifold $C\subset \R^2$ 
and obtained as in Section~\ref{s:toy} by considering the persistence as $t\ge0$ grows of $k$-dimensional topological features of the set consisting of all points in $\mathbb R^2$ within distance $t$ from  $C$. Note that $\mathrm{PD}_k$ is considered as being non-random and unknown; of course in our simulation study presented in Example~5 below we only pretend that $\mathrm{PD}_k$ is unknown. 
Let $\widehat{\mathrm{PD}}_{k,N}$ be  the random 
persistence diagram obtained as in Section~\ref{s:toy} from IID points $x_1,\ldots,x_N$ with support $C$.  
Let $\mathcal N=\{(b,d):\,b \le d,\,l\le 2  c_N\}$ be the set of points at distance $\sqrt{2}c_N$ of the diagonal in the persistence diagram. Let $S(b,d)=\{(x,y):\,|x-b|\le c_N,\,|y-d|\le c_N\}$ be the square with center $(b,d)$, sides parallel to the $b$- and $d$-axes, and of side length $2c_N$. 
Finally, let $\alpha\in(0,1)$.

\cite{fasy:etal:14} and~\cite{chazal:fasy:14} suggest various ways of constructing a bound $c_N>0$ so 
that an asymptotic conservative $100(1-\alpha)\%$-confidence region for 
$\mathrm{PD}_k$ with respect to the bottleneck distance $W_{\infty}$ is given by
\begin{equation}\label{e:confFazy}
 \liminf_{N\ra \infty} \mathrm P \left( W_{\infty} (\mathrm{PD}_k, \widehat{\mathrm{PD}}_{k,N} ) \le c_N  \right) \geq 1-\alpha
 \end{equation}
 where $W_{\infty}$ is the bottleneck distance defined in Section~\ref{s:1:APF}.
 The confidence region given by \eqref{e:confFazy} consists of those persistence diagrams 
$\mathrm{PD}_k$ which have exactly one point in each square $S(b_i,d_i)$, with $(b_i,d_i)$ a  
 point 
of $\widehat{\mathrm{PD}}_{k,N}$,
and have an arbitrary number of points in the set $\mathcal N$. 
\cite{fasy:etal:14} consider the points of $\widehat{\mathrm{PD}}_{k,N}$ falling in $\mathcal N$ as noise and the remaining points  as representing a significant topological feature of $C$.     

Using  \eqref{e:confFazy} an asymptotic conservative $100(1-\alpha)\%$-confidence 
region for the $\mathrm{APF}_k$ corresponding to $\mathrm{PD}_k$ is immediately obtained. 
This region will be bounded by two functions $\widehat{A}^{\mathrm{min}}_{k,N}$ and $\widehat{A}^{\mathrm{max}}_{k,N}$ specified by  $\widehat{\mathrm{PD}}_{k,N}$ and $c_N$. 
Due to the accumulating nature of $\mathrm{APF}_k$, the span between 
the bounds is an increasing function of the meanage. 
When using the Delaunay-complex, 
\cite{chazal:fasy:14} show that the span decreases as $N$ increases; this is illustrated in   Example~5 below.

\paragraph*{Example 5 (simulation study).} Let $C=\mathcal{C}((-1.5,0),1)\cup \mathcal{C}((1.5,0),0.8)$ and suppose each point $x_i$ is uniformly distributed on $C$. 
Figure~\ref{fig:simEx1} shows $C$ and an example of a simulated point pattern with $N=300$ points.
We use the bootstrap method implemented in the {\sf {\bf R}}-package {\tt TDA} and presented in \cite{chazal:fasy:14} to compute 
the $95\%$-confidence region for $\mathrm{PD}_1$ when $N=300$, see the top-left panel of Figure~\ref{fig:noise_APF}, 
where the two squares above the diagonal correspond to the two loops in $C$ and the other squares 
correspond to topological noise.  
Thereby $95\%$-confidence regions for  
$\mathrm{RRPD}_1$ (top-right panel) and $\mathrm{APF}_1$ (bottom-left panel) are obtained.
The confidence region for $\mathrm{APF}_1$ decreases as $N$ increases as demonstrated 
in the bottom panels where $N$ is increased from 300 to 500.  
As noticed in Section~\ref{s:1:APF}, we must be careful when using results based on the bottleneck metric, because
 small values of the bottleneck metric does not correspond to closeness of the two corresponding APFs: Although close persistence diagram with respect to the bottleneck distance imply that the two corresponding APFs are close with respect to the $L^q$-norm ($1\le q\le\infty$), the converse is not true. Hence, it is possible that an APF is in the confidence region plotted in Figure~\ref{fig:noise_APF} but that the corresponding persistence diagram is very different from the truth.

\begin{figure}[ht]
\begin{center}
\begin{tabular}{cc} 
 \includegraphics[scale=0.3]{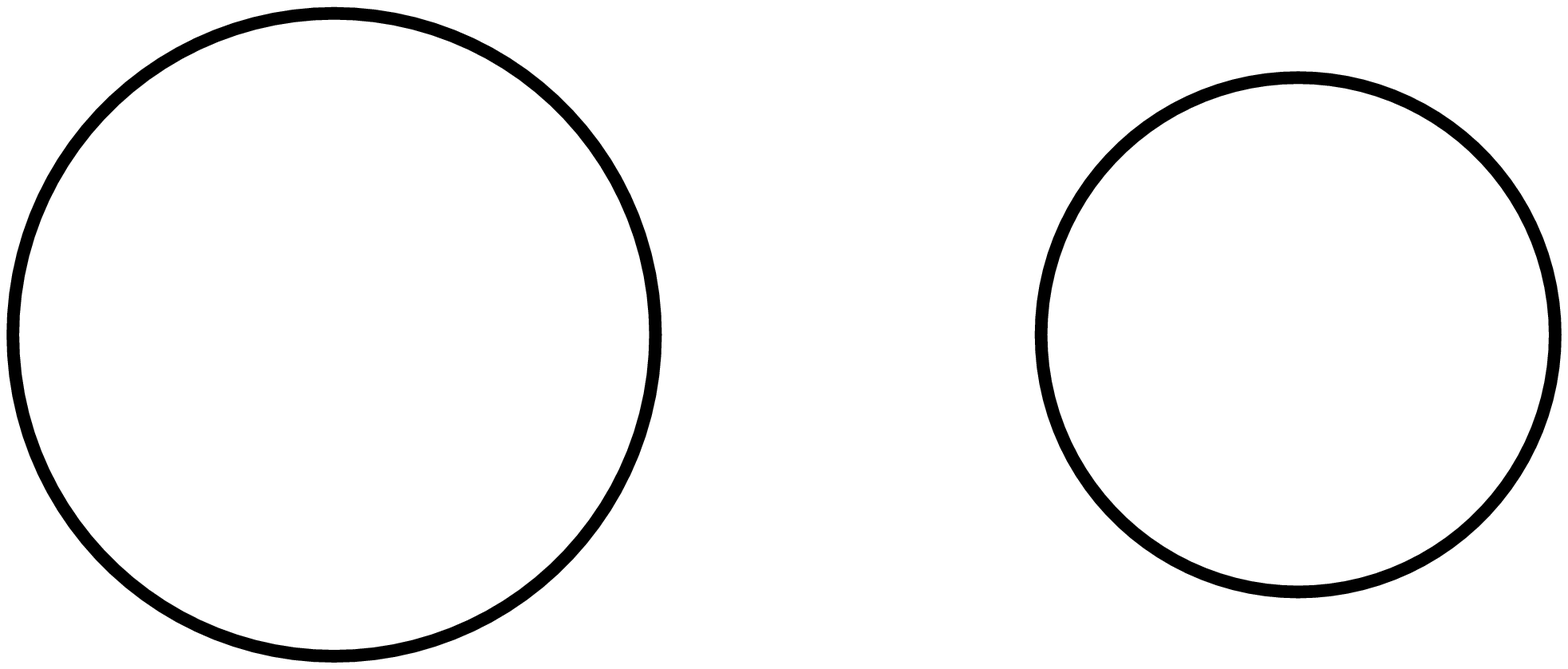} \hspace{1cm} & \hspace{1cm} \includegraphics[scale=0.3]{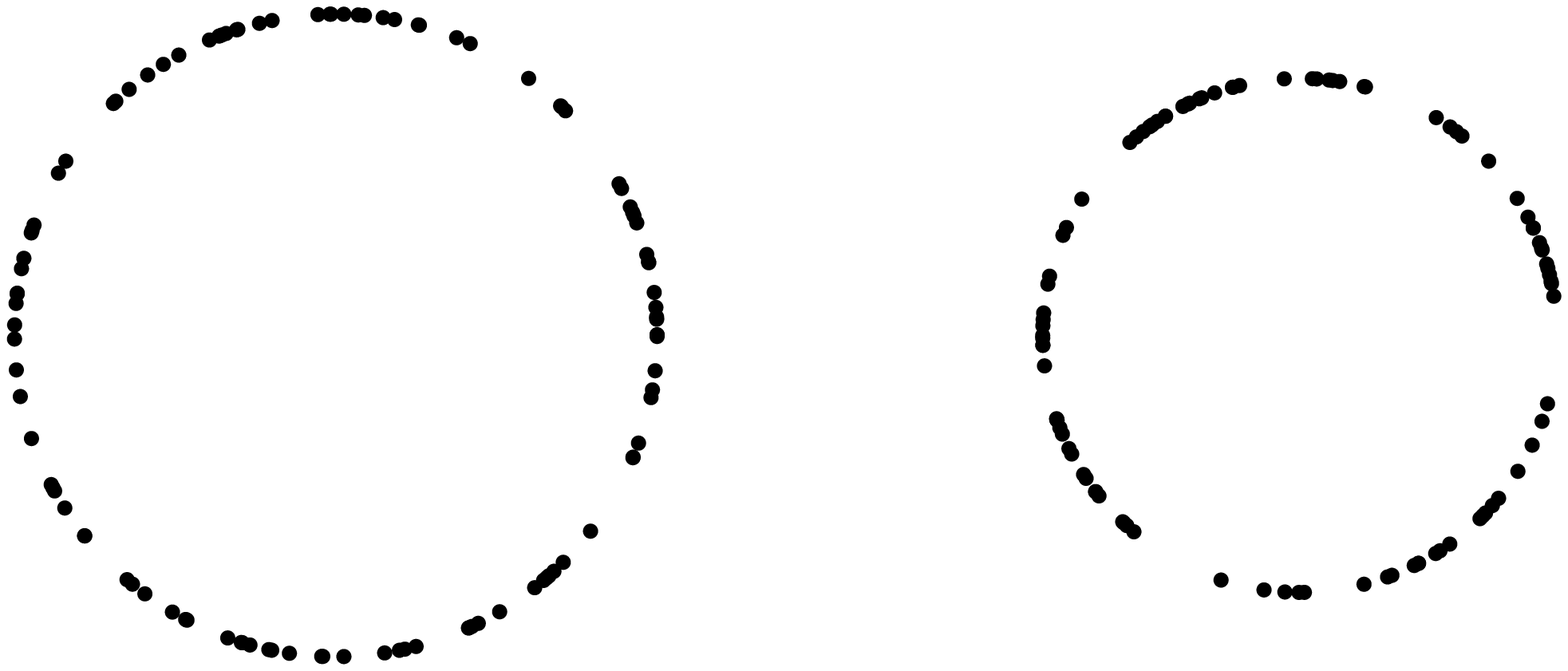} 
\end{tabular} 
\caption{The set $C$ in Example 5 (left panel) and a simulated point pattern of $N=300$  independent and uniformly 
distributed points on $C$ (right panel).}\label{fig:simEx1}
\end{center}
\end{figure}

 \begin{figure}[ht]
\begin{center}
\begin{tabular}{cc} 
\includegraphics[scale=0.2]{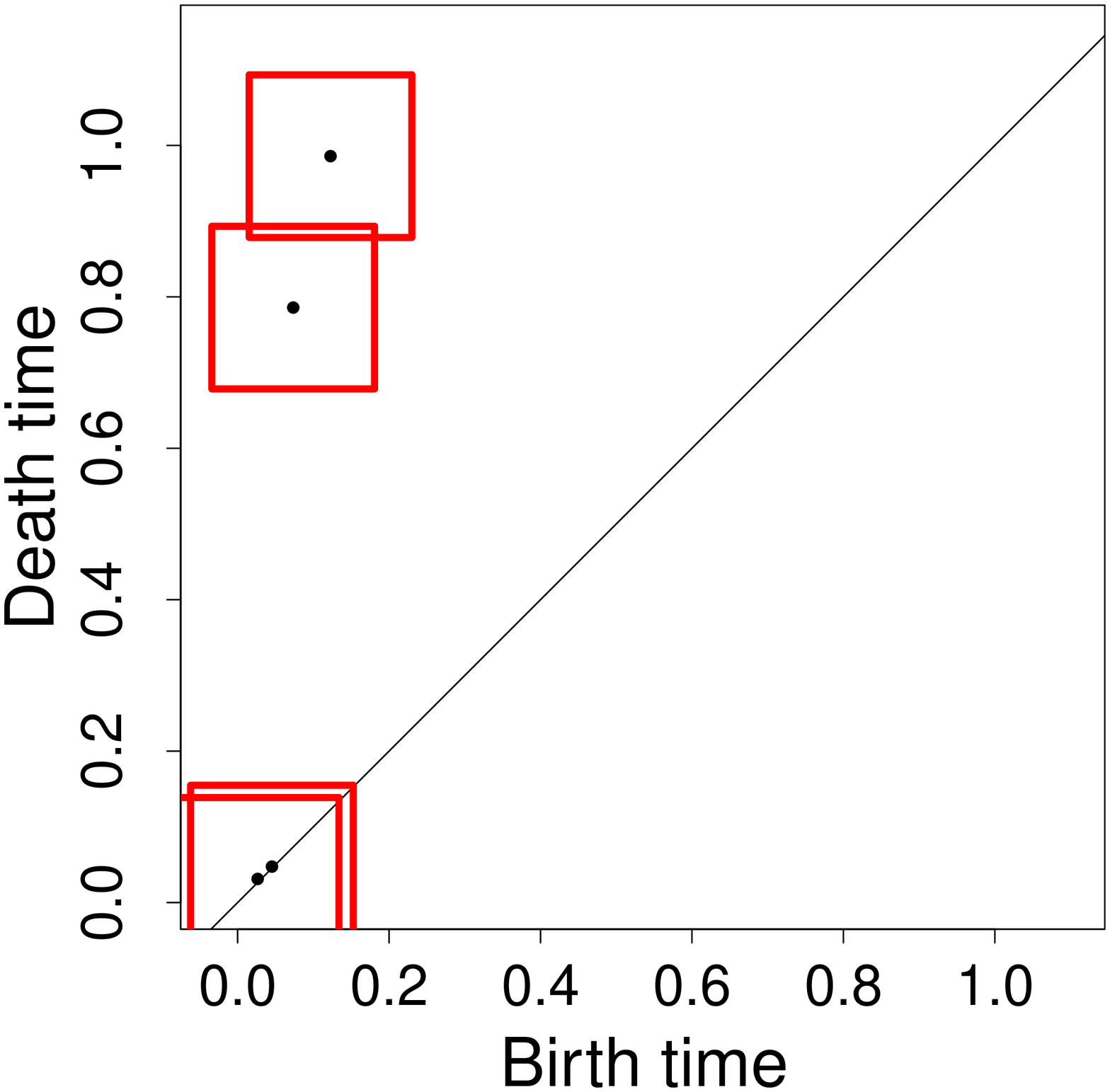}  & \includegraphics[scale=0.2]{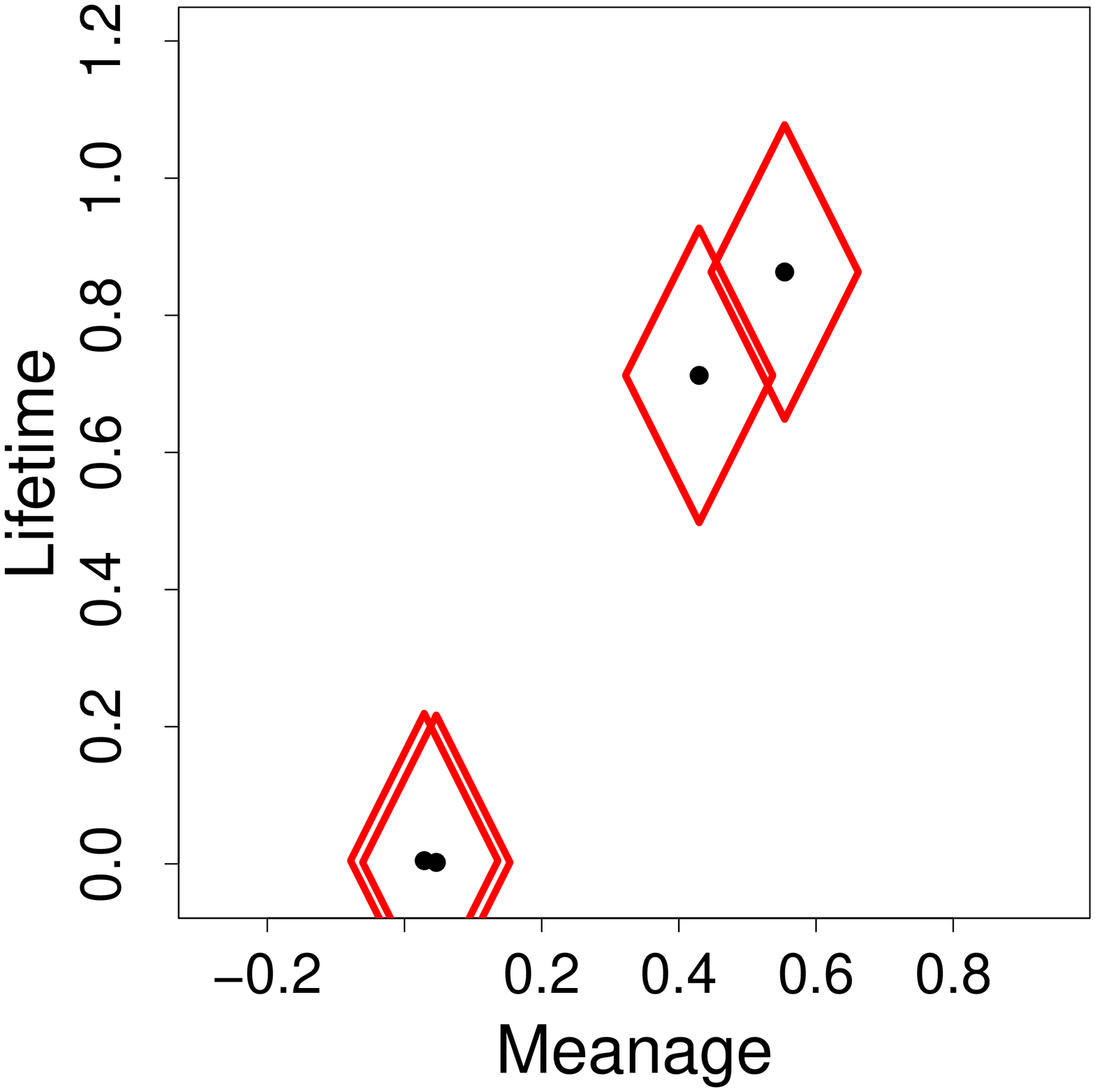} \\
\includegraphics[scale=0.2]{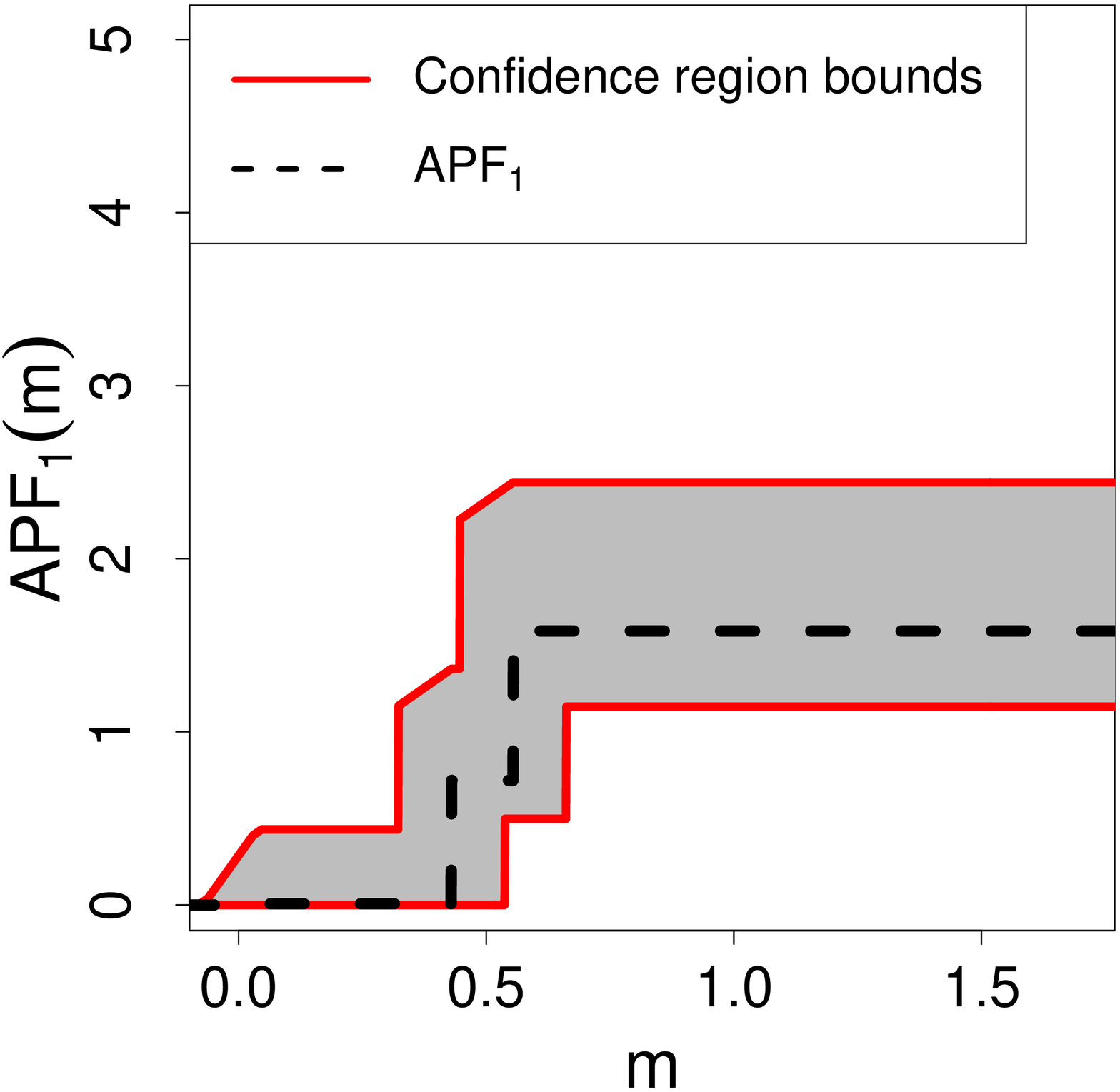} & \includegraphics[scale=0.2]{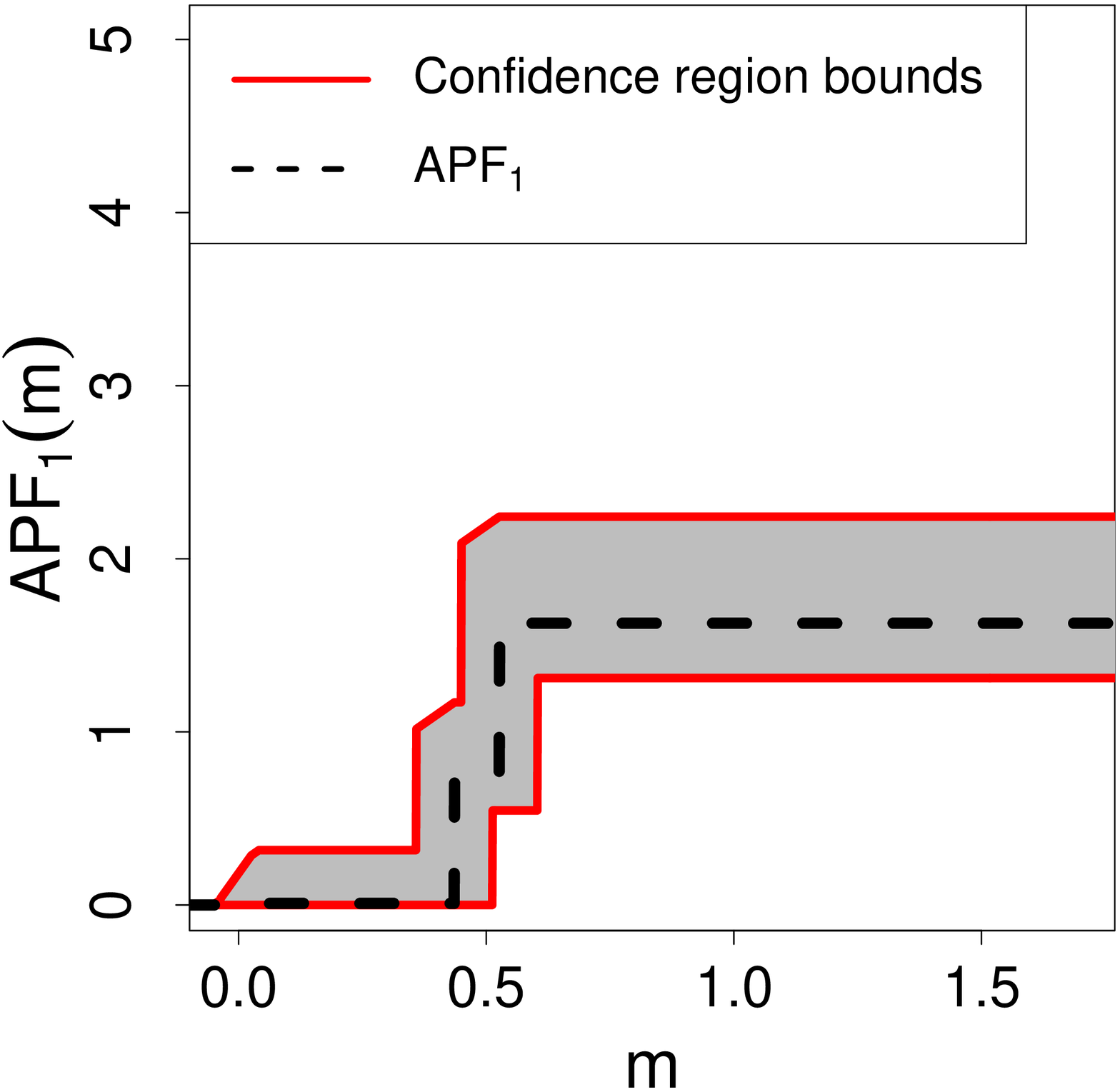}
\end{tabular} 
\caption{$95\%$-confidence regions obtained by  the bootstrap method 
for $\mathrm{PD}_1$ (top-left panel)
and its corresponding $\mathrm{RRPD}_1$ (top-right panel) when  $C$ and $x_1,\ldots,x_{300}$ are as in Figure~\ref{fig:simEx1}. 
The bottom-left panel shows the corresponding $95\%$-confidence region for $\mathrm{APF}_1$. 
The bottom-right panel shows the $95\%$-confidence region for $\mathrm{APF}_1$ when a larger point cloud with $300$ points is used.} \label{fig:noise_APF}
\end{center}
 \end{figure}

\section{Additional example related to Section 4.1 "Functional boxplot"}\label{s:appendix functional boxplot}

The functional boxplot described in Section~\ref{s:2:functional boxplot} can be used as an exploratory tool 
for the curves given by 
 a sample of $\mathrm{APF}_k$s. It
 provides a representation of the most central curve and the variation around this. It can also be used for outliers detection 
 as illustrated in the following example.

 \paragraph*{Example 6 (simulation study).}  We consider a sample of $65$ independent $\mathrm{APF}_k$s, 
  where the joint distribution of the first $50$ $\mathrm{APF}_k$s is exchangeable, whereas the last $15$  play the role of outliers. We suppose each $\mathrm{APF}_k$ corresponds to   
a point process of $100$ IID points, where each point $x_i$ follows one of the following
 distributions $\mathrm P_1,\ldots,\mathrm P_4$. 
 \begin{itemize}
\item  $\mathrm  P_1$ (unit circle):  $x_i$ is a uniform point on $\mathcal{C}((0,0),1)$ perturbed by $N_2\left(0.1 \right)$-noise.
\item  
 $\mathrm  P_2$ (Gaussian mixture): Let $y_i$ follow $N_2\left(0.2 \right)$, then $x_i=y_i$ with probability $0.5$, and $x_i=y_i+(1.5,0.5)$ otherwise. 
 \item   
 $\mathrm  P_3$ (two circles): $x_i$ is a uniform point on $\mathcal{C}((-1,-1),1)\cup \mathcal{C}((1,1),0.5)$ perturbed by $N_2\left((0,0),0.1 \right)$-noise. 
 \item  
 $\mathrm  P_4$ (circle of radius $0.7$): $x_i$ is a uniform point on $\mathcal{C}((0,0),0.7)$ perturbed by $N_2\left(0.1 \right)$-noise.     
 \end{itemize} 
We let the first  $50$ point processes be obtained from $\mathrm P_1$ (the distribution for non-outliers), 
 the  next $5$ from  $\mathrm  P_2$, the following $5$ from  $\mathrm  P_3$, and the final $5$ from $\mathrm  P_4$. Figure~\ref{fig:4patterns} shows 
 a simulated realization of each of the four types of point processes.
 
 \begin{figure}
 \begin{center}
\begin{tabular}{c}
\includegraphics[scale=0.25]{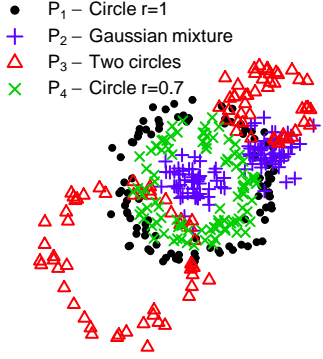}
\end{tabular}
 \caption{Simulated realizations of the four types of point processes, each consisting of 100 IID points with distribution either $\mathrm  P_1$ (black dots), 
 $\mathrm  P_2$ (blue crosses), $\mathrm  P_3$ (red triangles), or $\mathrm  P_4$ (rotated green crosses).} \label{fig:4patterns}
 \end{center}
\end{figure}
 
  Figure~\ref{fig:functionalboxplot} shows the functional boxplots when considering $\mathrm{APF}_0$ (left panel) and $\mathrm{APF}_1$ (right panel).
  The curves detected as outliers and corresponding to the distributions
  $\mathrm P_2, \mathrm P_3$, and $\mathrm P_4$ are plotted in red, blue, and green, respectively. In both panels   the outliers detected by the 1.5 criterion agree with the true outliers. 
  
  In the left panel, each curve has an accumulation of small jumps between $m=0$ and $m\approx 0.1$, corresponding to the moments where the points associated to each circle are connected by the growing discs in the sequence $\{C_t\}_{t\ge0}$. The curves corresponding to realisations of $P_3$ have a jump at $m\approx 0.38$ which corresponds to the moment where the points associated to the two circles used when defining $P_3$ are connected by the growing discs in the sequence $\{C_t\}_{t\ge0}$.
  The points following the distribution $P_4$ are generally closer to each other than the ones following the distribution $P_1$ as the radius of the underlying circle is smaller. This  
  corresponds  to more but smaller jumps in $\mathrm{APF}_0$ for small meanages, and hence the curves of $\mathrm{APF}_0$
  are lower when they correspond to realisations of $P_1$ than to realisations of $P_4$; and as expected, for large meanages, 
  the curves of $\mathrm{APF}_0$
  are larger when they correspond to realisations of $P_1$ than to realisations of $P_4$. Note that if we redefine $P_4$ so that the $N_2\left(0.1 \right)$-noise is replaced by $N_2\left(0.07 \right)$-noise, then the curves would be the same up to rescaling. 
  
  In the right panel, we observe clear jumps in all $\mathrm{APF}_1$s obtained from $P_1$, $P_3$, and $P_4$. 
 These jumps correspond to the first time that the loops of the circles in $P_1$, $P_3$, and $P_4$ are covered by the union of growing discs in the sequence $\{C_t\}_{t\ge0}$. Once again, if we have used $N_2\left(0.07 \right)$-noise in place of $N_2\left(0.1 \right)$-noise in the definition of $P_4$, the curves would be the same up to rescaling.

  If we repeat everything but with the distribution $P_4$ redefined so that $\mathcal{C}((0,0),0.7)$ is replaced by $\mathcal{C}((0,0),0.8)$, then the support of $P_4$ is closer to that of $P_1$ and it becomes harder in the case of $\mathrm{APF}_0$
  to detect the outliers with distribution $P_4$ (we omit the corresponding plot); thus further simulations for determining   
 a stronger criterion would be needed.

\begin{figure}
\centering
\begin{tabular}{cc} 
\includegraphics[scale=0.25]{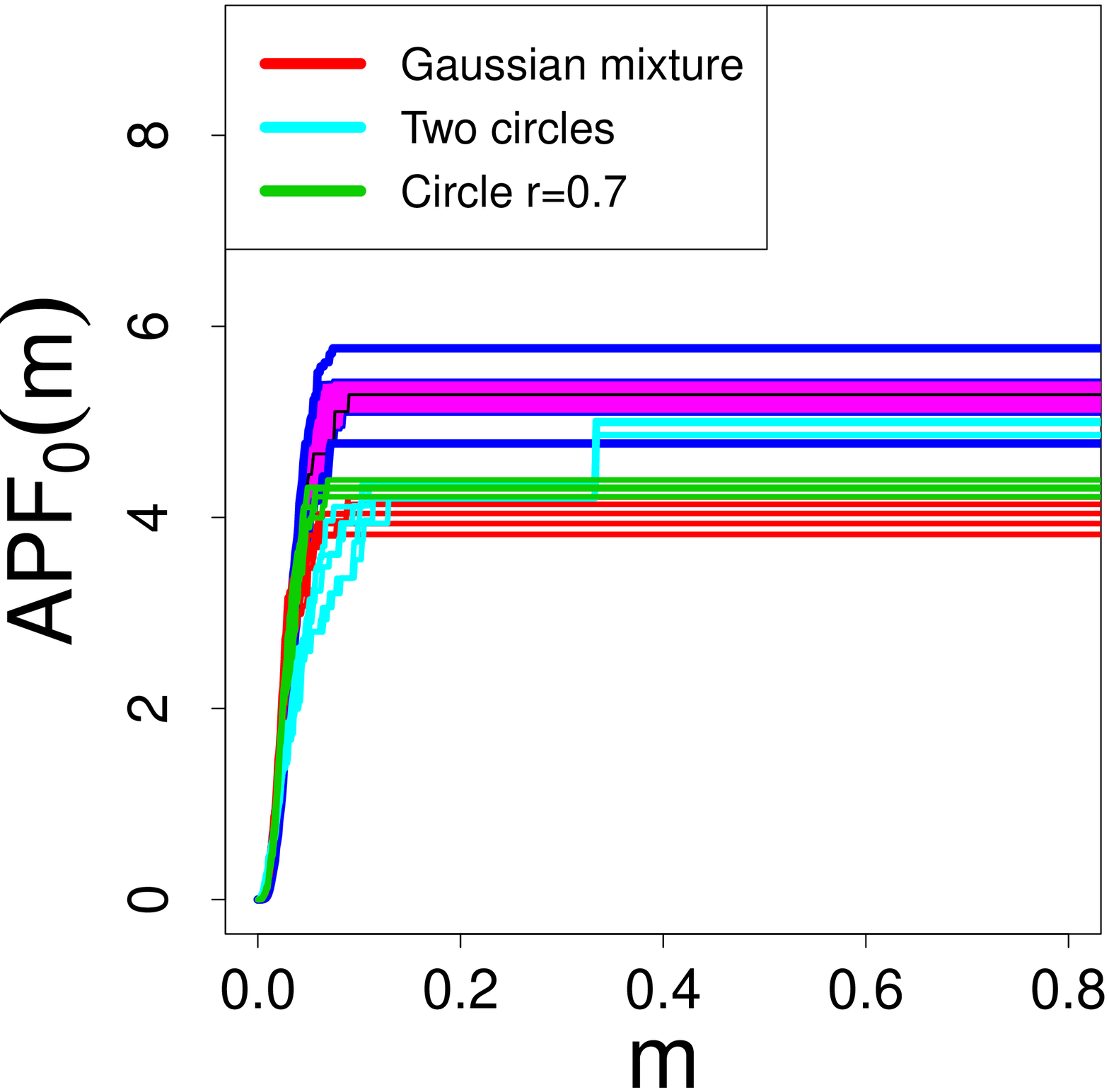} & \includegraphics[scale=0.25]{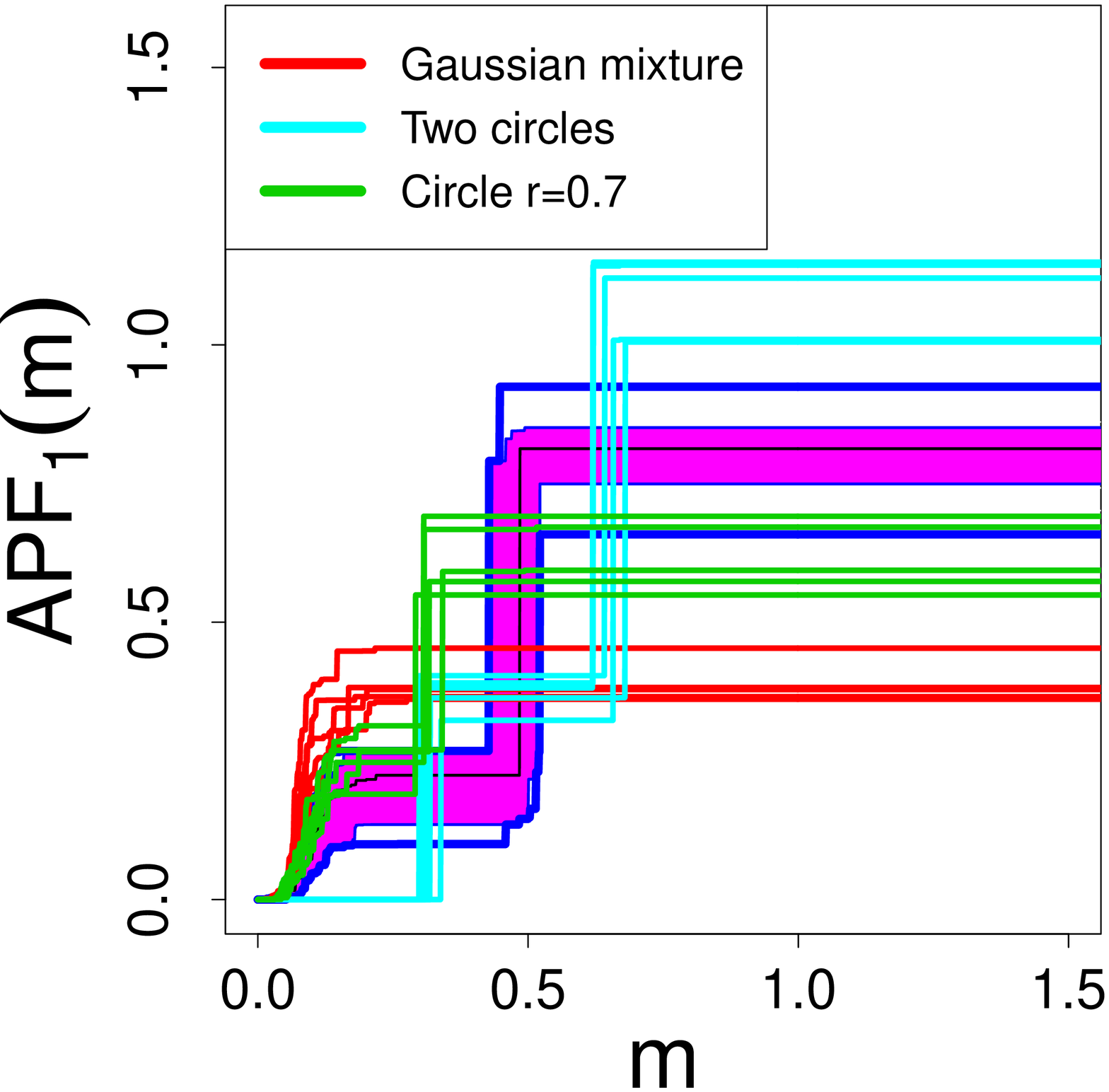}  
\end{tabular}
\caption{Functional boxplots of $65$ APFs based on the topological features of dimension $0$ (left panel) and $1$ (right panel). 
In each panel, 50, 5, 5, and 5 APFs are obtained from the Delaunay-complex of $100$ 
 IID points from the distribution $\mathrm P_1$, $\mathrm P_2$, $\mathrm P_3$, and $\mathrm P_4$, respectively.  
The APFs detected as outliers are plotted in red, blue, and green in the case of $\mathrm P_2$, $\mathrm P_3$, and $\mathrm P_4$, respectively.}\label{fig:functionalboxplot}
 \end{figure}

\section{Additional example related to Section 4.2 "Confidence region for the mean function"}\label{s:appendix crmean}

This appendix provides yet an example to illustrate the bootstrap method in Section~\ref{s:crmean} for obtaining
 a confidence region for the mean function of a sample of IID $\mathrm{APF}_k$s. 

\paragraph*{Example 7  (simulation study).}
Consider 50 IID copies of a point process consisting of 100  
independent and uniformly distributed points on the union of three circles
with radius $0.25$ and centred at $(-1,-1)$, $(0,1)$, 
 and $(1,-1)$, respectively (these circles were also considered in the example of Section~\ref{s:toy}). A simulated realization of the point process is shown in the left panel of Figure~\ref{fig:mean_apf}, and 
 the next two panels show simulated confidence regions for
 $\mathrm{APF}_0$ and $\mathrm{APF}_1$, respectively, 
when the bootstrap procedure with $B=1000$ is used. 
In the middle panel, between $m=0$ and $m=0.2$,  
there is an accumulation of small jumps corresponding to the moment when each circle is covered by the union of growing discs
from the sequence $\{C_t\}_{t\ge0}$; we interpret these small jumps as topological noise. 
The jump at $m\approx 0.25$ corresponds to the moment when the circles 
centred at $(-1,-1)$ and $(1,-1)$ are connected by the growing discs, and
the jump at $m\approx 0.3$ to when
all three circles are connected by the growing discs. 
In the right panel, at $m\approx 0.3$ there is an accumulation 
of small jumps corresponding to the moment when the three circles are connected by the growing discs and they form a
loop at $m=0.25$ in Figure~\ref{fig:TDA_tuto}. The disappearance of this loop at $m=0.69$ in Figure~\ref{fig:TDA_tuto} corresponds to the jump at $m\approx 0.7$ in Figure~\ref{fig:mean_apf}.

 \begin{figure}
\begin{center}
\resizebox{\columnwidth}{!}{
\begin{tabular}{ccc} 
\includegraphics[scale=0.25]{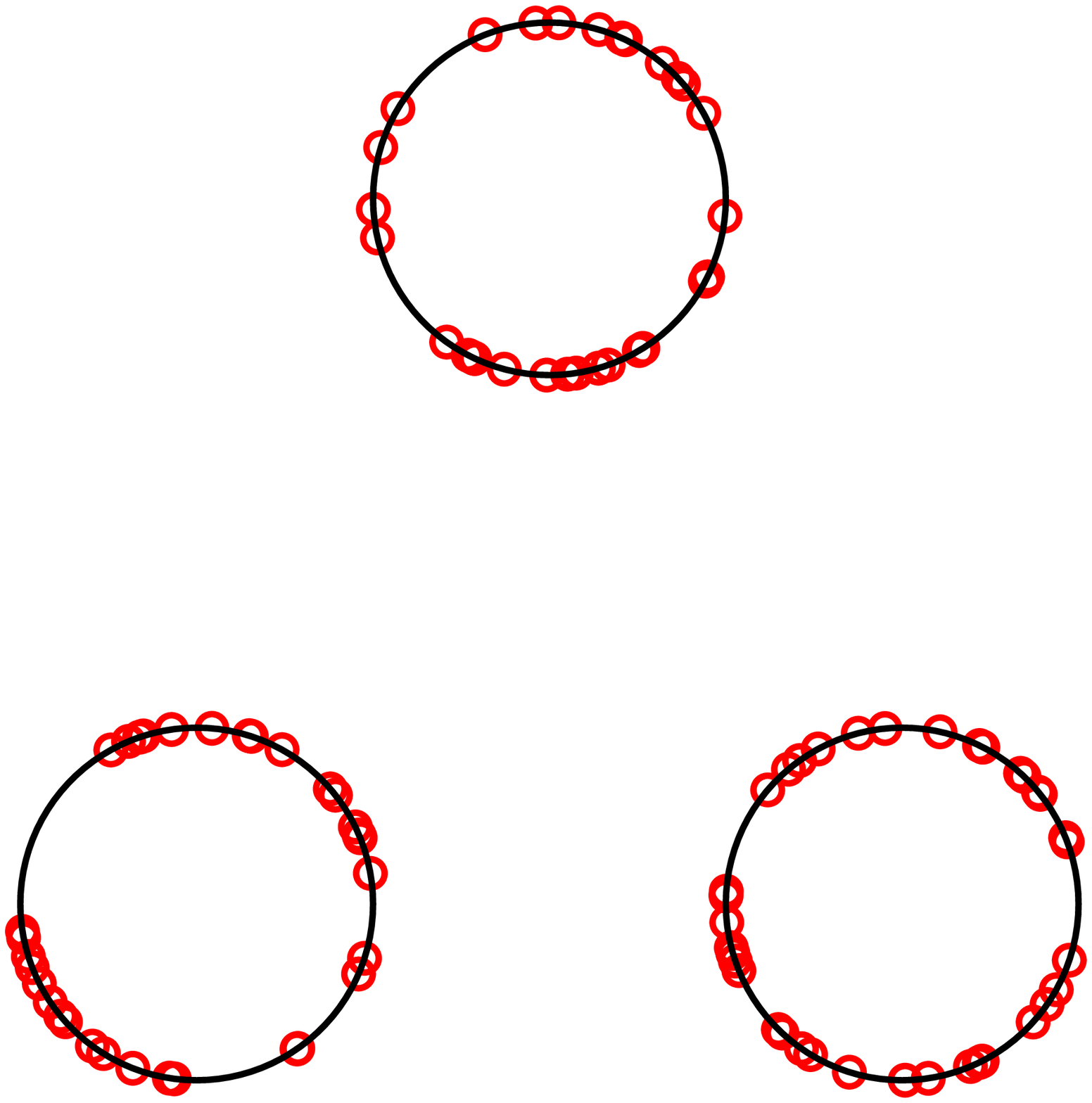} & \includegraphics[scale=0.25]{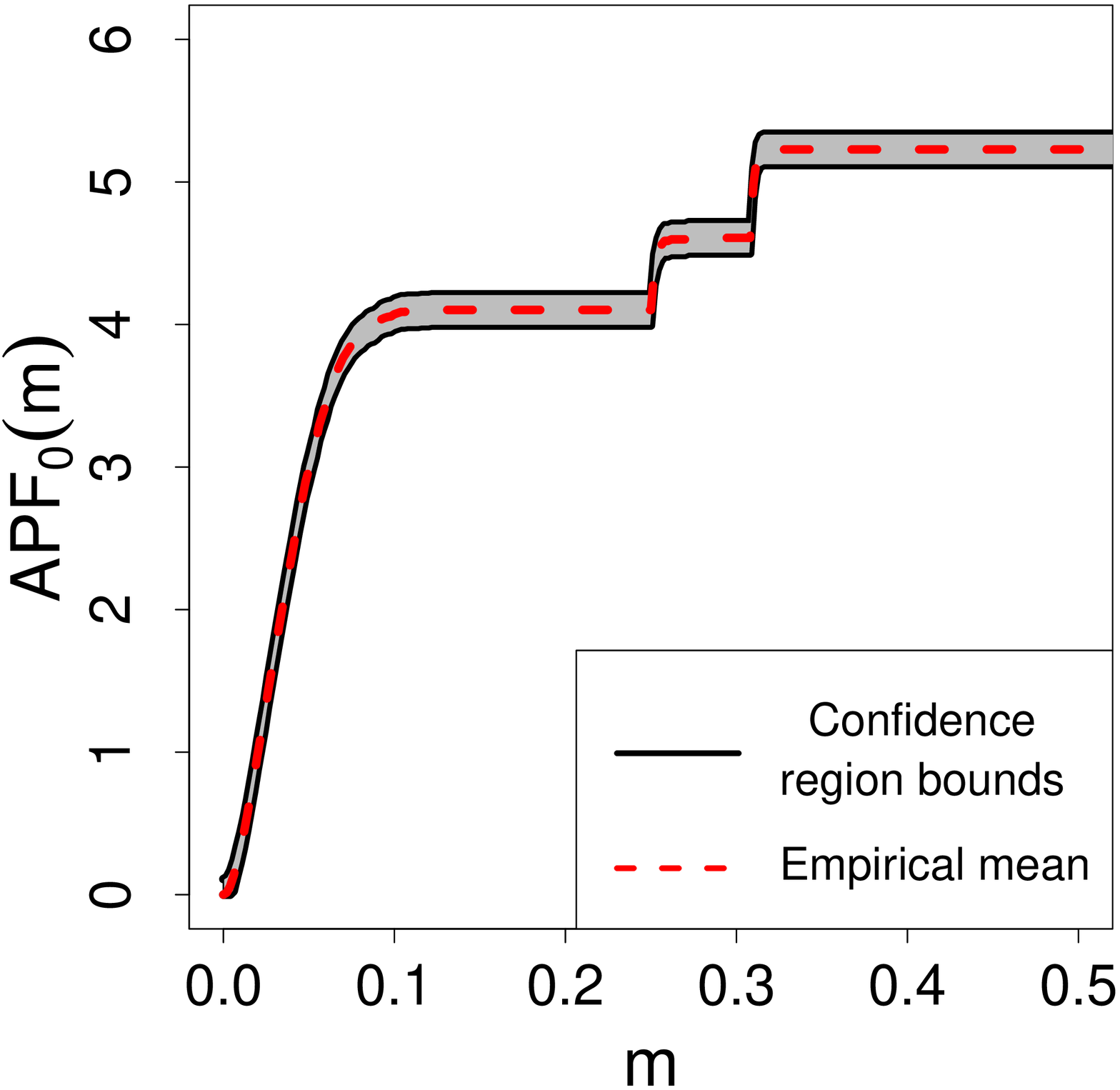}  & \includegraphics[scale=0.25]{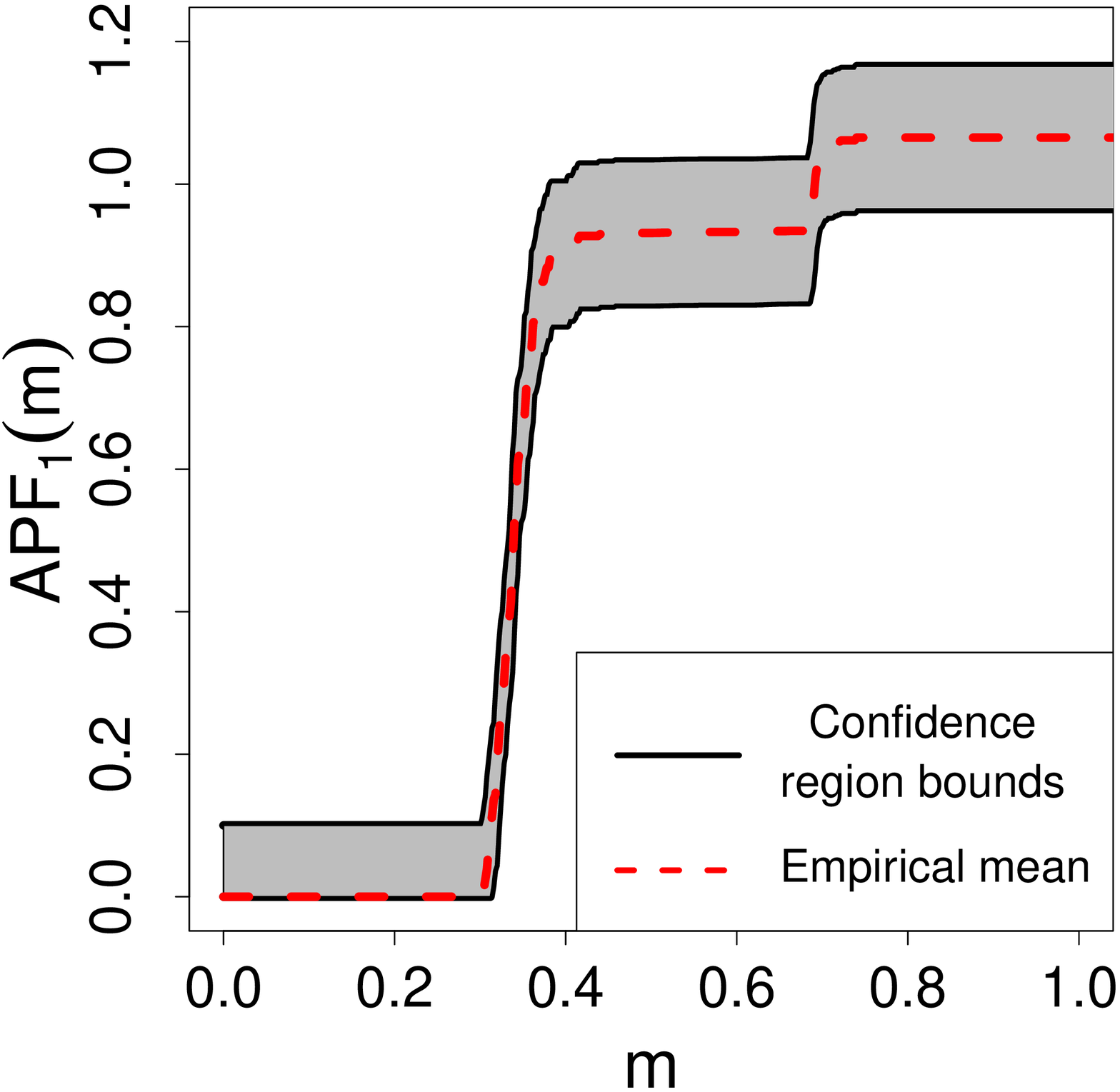} 
\end{tabular}
}
\caption{A simulation of 100 independent and uniformly distributed points on the union of three circles (dashed lines) with the same radius $r=0.5$ and centred  at $(-1,-1)$, $(0,1)$, and $(1,-1)$ (left panel). The $95\%$-confidence regions for the mean $\mathrm{APF}_0$ (middle panel) and the mean $\mathrm{APF}_1$ (right panel) are based on 50 IID simulations. }\label{fig:mean_apf}
\end{center}
 \end{figure}

\section{Additional example to Section 5 "Two samples of accumulated persistence functions"}\label{s:appendix two sample test}

Section~\ref{s:two sample problem} considered two samples of independent $\mathrm{RRPD}_k$s $D_1,\ldots,D_{r_1}$ and $E_1,\ldots,E_{r_2}$, 
where each $D_i$ ($i=1,\ldots,r_1$) has distribution $\mathrm{P}_D$ 
and each $E_j$ has distribution $\mathrm{P}_E$ ($j=1,\ldots,r_2$). Then we studied a bootstrap two-sample test to asses the null hypothesis $\mathcal{H}_0$: $\mathrm{P}_D = \mathrm{P}_E$, e.g.\ in connection to the brain artery trees.  An additional example showing the performance of the test is presented below.

 \paragraph*{Example 8  (simulation study).}
 Let $\mathrm{P}_D$ be the distribution of a $\mathrm{RRPD}_k$ obtained from $100$ independent and uniformly distributed points on $\mathcal{C}((0,0),1)$ perturbed by $N_2\left(0.2 \right)$-noise, and define
$\mathrm{P}_E$ in a similar way but with a circle of radius $0.95$.  
A simulated realisation of each point process is shown in Figure~\ref{fig:twosample}; it seems difficult to recognize  
that the underlying circles are different. Let us consider the two-sample
test statistics \eqref{e:KSfirst} and \eqref{critical value L1} 
 with  $I=[0,3]$, $r_1=r_2=50$, and $\alpha=0.05$. Over $500$ simulations of the two samples of $\mathrm{RRPD}_k$ we obtain the following
percentage of rejection:
For  $M_{r_1,r_2}$, $5.2\%$ if $k=0$, and
$24.2\%$ if $k=1$. 
For $KS_{r_1,r_2}$ much better results are observed, namely $73.8\%$ if $k=0$, and 
$93.8\%$ if $k=1$, where this high percentage is mainly caused by  
 the largest lifetime of a loop.

\begin{figure}
\begin{center}
\begin{tabular}{cc}
 \includegraphics[scale=0.25]{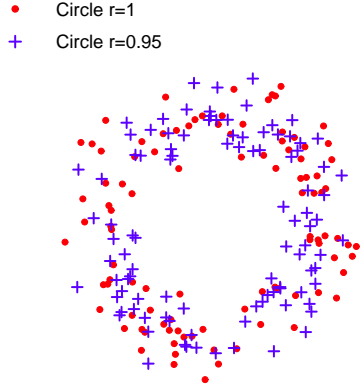}  
\end{tabular}
\end{center}
\caption{A simulation of $100$ independent and uniformly distributed points on the circle centred at $(0,0)$ with radius $1$ and perturbed by $N_2\left((0,0),0.2 \right)$-noise (red dots), together with $100$ independent and uniformly distributed points on the circle centred at $(0,0)$ with radius $0.95$ and perturbed by $N_2\left((0,0),0.2 \right)$-noise (blue crosses).
} \label{fig:twosample}
\end{figure}

\section{Further methods for two or more samples of accumulated persistence functions}\label{s:group of APFs}
\subsection{Clustering}\label{s:clusteringappendix}

Suppose  $A_1,\ldots,A_r$ are $\mathrm{APF}_k$s 
which we want to label into $K<r$ groups by using a method of 
clustering (or unsupervised classification). 
Such methods are studied many places in the literature for functional data, 
see the survey in~\cite{julien:preda:14}. In particular, 
\cite{chazal:cohen:09},~\cite{chen:etal:15}, and~\cite{robins:turner:16} consider clustering in connection to 
$\mathrm{RRPD}_k$s. 
Whereas the $\mathrm{RRPD}_k$s are two-dimensional functions, 
it becomes easy to use clustering for the one-dimensional $\mathrm{APF}_k$s as illustrated in Example 9 below.

For simplicity we just consider the standard technique known as 
the $K$-means clustering algorithm (\cite{hartigan:wong:79}). For more complicated applications than considered in Example~9 the EM-algorithm may be needed for the $K$-means clustering algorithm. 
As noticed by a referee, to avoid the use of the EM-algorithm we can modify~\eqref{critical value KS} or~\eqref{critical value L1} and thereby construct a distance/similarity matrix for different APFs which is
used to perform hierarchical clustering. However, for Example~9 the results using hierarchical clustering (omitted here) were not better than with the $K$-means algorithm.

Assume that $A_1,\ldots,A_r$ are pairwise different 
and square-integrable functions on $[0,T]$, where $T$ is a user-specified parameter.
For example, if $\mathrm{RRPD}_k \in \mathcal{D}_{k,T,n_{\mathrm{max}}}$ 
(see Section~\ref{s:confidence region mean}), then $\mathrm{APF}_k \in L^2([0,T])$. 
The $K$-means clustering algorithm works as follows.
\begin{itemize}
\item Chose  uniformly at random  a subset of $K$  functions from $\{A_1,\ldots,A_r \}$; call these functions centres and label them by $1,\ldots,K$. 
\item Assign each non-selected $\mathrm{APF}_k$ the label $i$ if it is closer to the centre of label $i$
than to any other centre with respect to the $L^2$-distance on  $L^2([0,T])$. 
\item In each group, reassign the centre by the mean curve of the group (this may not be 
an $\mathrm{APF}_k$ of the sample). 
\item Iterate these steps until the assignment of centres does not change.
\end{itemize}
The  algorithm is known to be convergent, 
however, it may have several drawbacks as discussed in~\cite{hartigan:wong:79} and \cite{bottou:bengio:95}.
   
\paragraph*{Example 9  (simulation study).} Consider $K=3$ groups, each consisting of 50 $\mathrm{APF}_0$s 
and associated to point processes consisting of
 $100$ IID  points, where each point $x_i$ follows one of the following 
 distributions $\mathrm P_1$, $\mathrm P_2$, and $\mathrm P_3$ for groups 1, 2, and 3, respectively.
\begin{itemize}
\item  $\mathrm P_1$ (unit circle): $x_i$ is a uniform point on $\mathcal{C}((0,0),1)$ perturbed by $N_2\left(0.1 \right)$-noise. 
\item  
$\mathrm  P_2$ (two circles): $x_i$ is a uniform point on $\mathcal{C}((-1,-1),0.5)\cup \mathcal{C}((1,1),0.5)$ perturbed by $N_2\left(0.1 \right)$-noise. 
\item
$\mathrm P_3$ (circle of radius 0.8): $x_i$ is a uniform point on $\mathcal{C}((0,0),0.8)$ perturbed by $N_2\left(0.1 \right)$-noise. 
\end{itemize}
We start by simulating a realization of each of the $3\times 50=150$ point processes. 
The left panel of Figure~\ref{fig:clustering} shows one realization of each type of point process; it seems difficult to distinguish the underlying circles for groups 1 and 3, but  
the three $\mathrm{APF}_0$s associated to these three point patterns  are in fact
 assigned to their right groups. The right panel of Figure~\ref{fig:clustering} shows the result of the $K$-means clustering algorithm.
 Here we are
using the
{\sf \bf R}-function ``kmeans'' for the $K$-means algorithm and it takes only a few seconds when 
evaluating each $A_i(m)$  at  $2500$ 
equidistant values of $m$ between $0$ and $T=0.5$. As expected we see more overlap between the curves of the $\mathrm{APF}_0$s assigned to groups 1 and 3. 

We next repeat 500 times the simulation of the 150 point processes. 
A clear distinction between the groups is obtained by the $K$-means algorithm applied for connected components: 
The  percentage of wrongly assigned
 $\mathrm{APF}_0$s among the $500\times3\times50=75000$ $\mathrm{APF}_0$s has an average of $4.5\%$ 
 and a standard deviation of $1.6\%$. The assignment error is in fact mostly caused by 
incorrect labelling  of  $\mathrm{APF}_0$s associated to $\mathrm P_1$ or $\mathrm P_3$.  
This is expected as the  underlying circles used  in the definitions of $P_1$ and $P_3$ are rather close,  whereas the underlying set in the definition of $P_2$ is different with two connected components as represented by the jump at $m\approx 0.4$ in the middle panel of Figure~\ref{fig:clustering}.
 
Even better results are obtained when considering loops
instead of connected components: The  percentage of wrongly assigned
 $\mathrm{APF}_1$s among the $75000$ $\mathrm{APF}_1$s has an average of $1.6\%$ 
 and a standard deviation of $1.0\%$. This is mainly due to the sets underlying $P_1$, $P_2$, and $P_3$ which have distinctive loops that results in clear distinct jumps in the $\mathrm{APF}_1$s as seen in the right panel of Figure~\ref{fig:clustering}.  
 
\begin{figure}[ht]
\begin{center}
\begin{tabular}{ccc}
\includegraphics[scale=0.2]{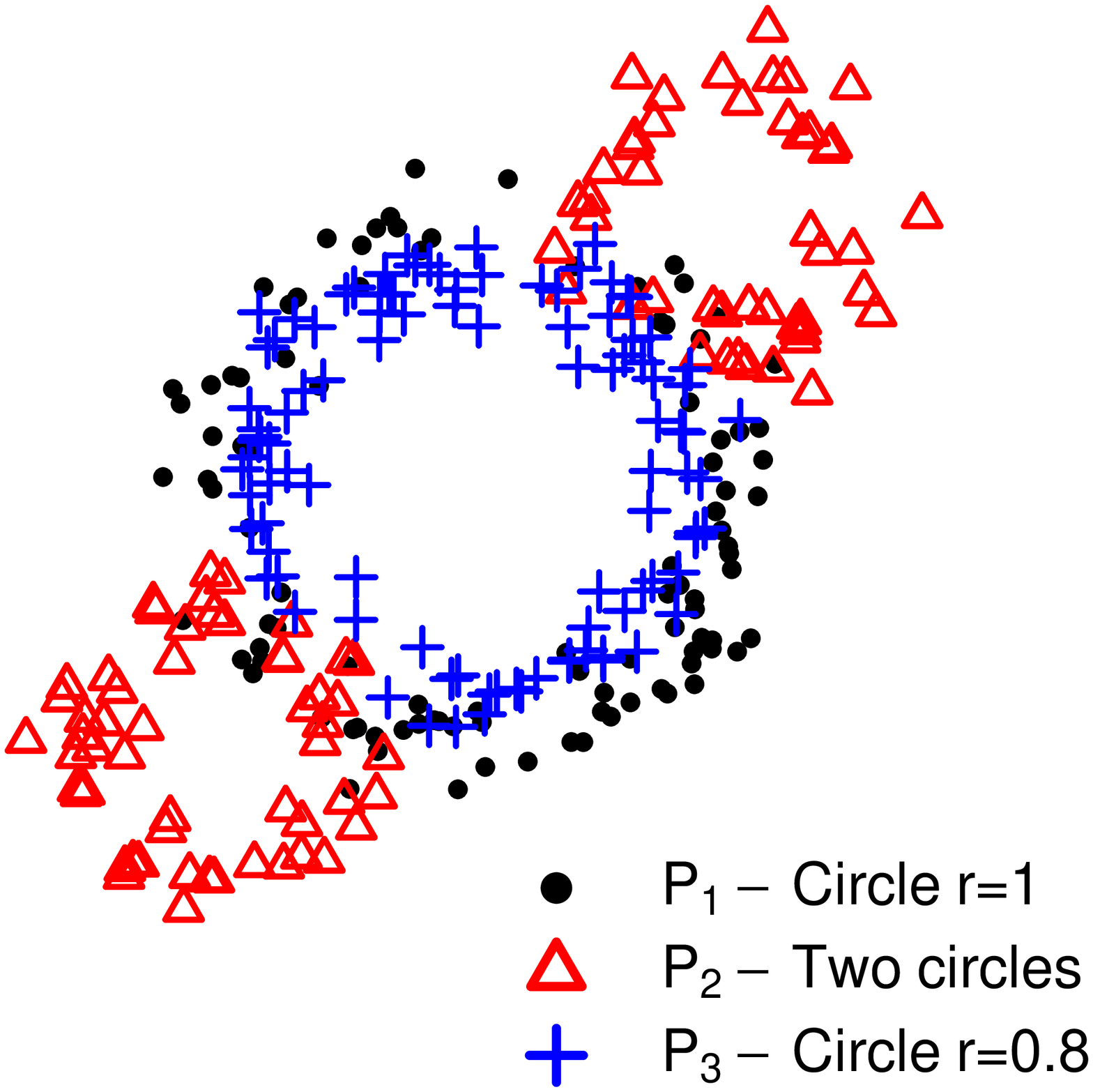} &  \includegraphics[scale=0.2]{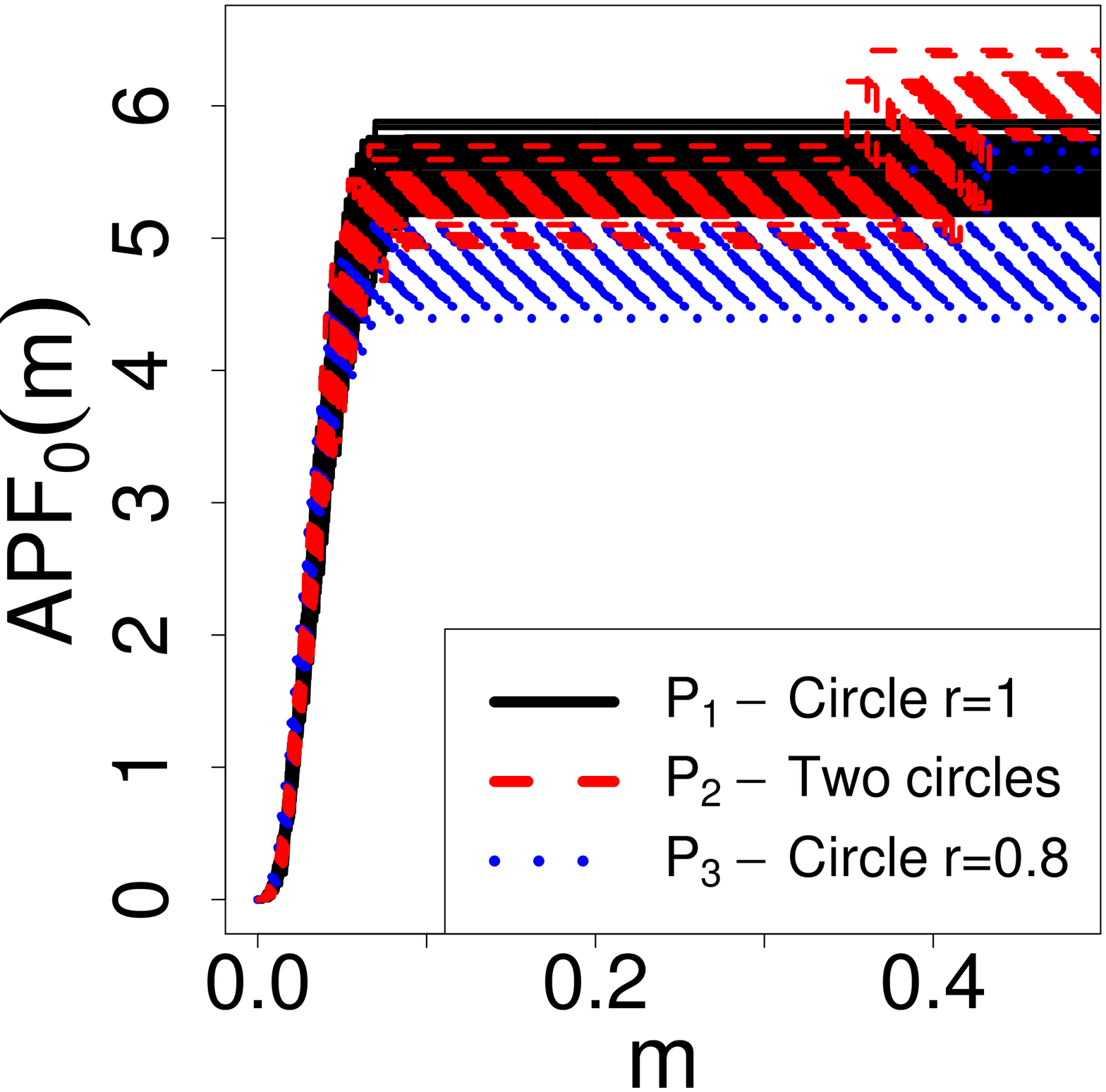} & \includegraphics[scale=0.2]{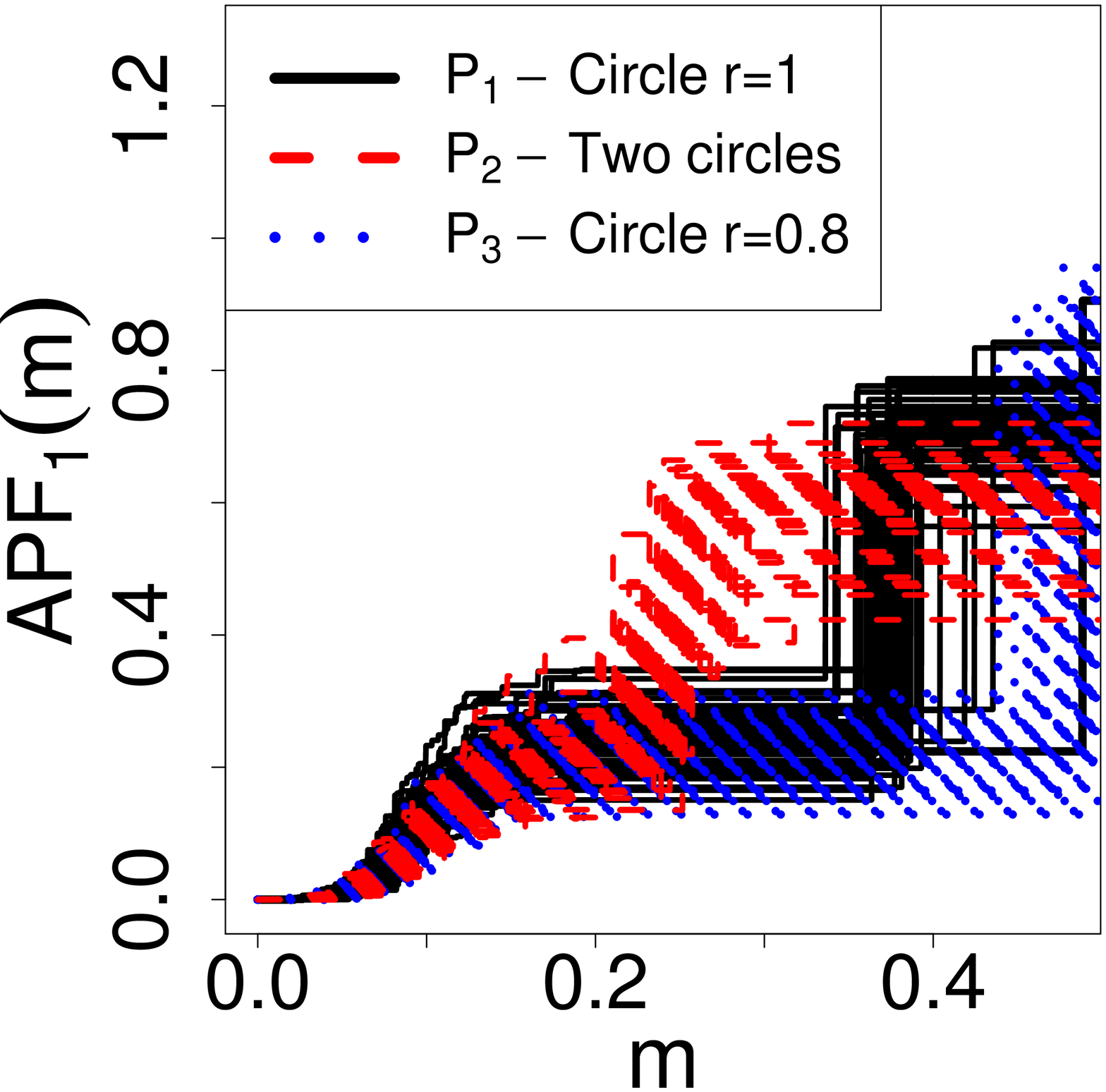}
\end{tabular}
\end{center}
\caption{Left panel: Simulated example of the three point processes, each consisting of $100$ IID points drawn from the distribution $P_1$ (black dots), $P_2$ (red triangles), or $P_3$ (blue crosses). 
Middle  panel: The 150 $\mathrm{APF}_0$s obtained from the simulation of the $150$ point processes associated to $\mathrm P_1$, $\mathrm P_2$, or $\mathrm P_3$, where the colouring in black, red, or blue specifies whether the  $K$-means algorithm assigns an $\mathrm{APF}_0$ to the group associated to $P_1$, $P_2$, or $P_3$. Right panel: As the middle panel but for the 150 $\mathrm{APF}_1$s.}
\label{fig:clustering}
\end{figure}

\subsection{Supervised classification}\label{s:supervised classification}

Suppose we want to assign an $\mathrm{APF}_k$ to a training set of $K$ different groups $ \mathcal G_1 ,\ldots, \mathcal G_K$,
where $\mathcal{G}_i$ is a sample of $r_i$ independent $\mathrm{APF}_k$s $A^i_1, \ldots, A^i_{r_i}$. For this purpose  supervised classification methods for functional data
may be adapted. 

We just consider 
a particular method by~\cite{lopez:romo:10}:
Suppose $\alpha\in[0,1]$ and we believe that at least $100(1-\alpha)\%$ 
of the $\mathrm{APF}_k$s in each group are IID, whereas the remaining $\mathrm{APF}_k$s in each group follow a different distribution and are considered as outliers (see Section~\ref{s:2:functional boxplot}). 
For a user-specified parameter $T>0$ and $i = 1, \ldots , K$,  
define the $100\alpha\%$-trimmed mean $\overline{A}^{\, \alpha}_{i} $ with respect to $\mathcal G_i$
as the   mean function on $[0,T]$  of the $100(1-\alpha)\%$  $\mathrm{APF}_k$s in $\mathcal G_i$ with the largest
$\MBD_{r_i}$,  see~\eqref{e:definition MBD}.
Assuming $\cup_{i=1}^K \mathcal{G}_i \subset L^2([0,T])$,
an  $\mathrm{APF}_k$ $A\in  L^2([0,T])$ is assigned  to  $\mathcal G_{i}$
if
\begin{align}\label{e:assign}
i  = \argmin_{j \in \{1,\ldots, K\} } \| \overline{A}^{\, \alpha}_{j} - A \|  ,
\end{align}
where $\| \cdot \|$ denotes the $L^2$-distance. Here, the trimmed mean is used for robustness and allows a control over the curves we may like to omit because of outliers, but e.g.\ the median could have been used instead.

\paragraph*{Example 10 (simulation study).} 
Consider the following distributions $\mathrm P_1,\ldots,\mathrm P_4$ for a point $x_i$.
\begin{itemize}
\item $\mathrm P_1$ (unit circle): $x_i$ is a uniform point on $\mathcal{C}((0,0),1)$
which is perturbed by $N_2\left(0.1 \right)$-noise.
\item
$\mathrm P'_1$ (two circles, radii 1 and 0.5): $x_i$ is a uniform point on $\mathcal{C}((0,0),1)\cup \mathcal{C}((1.5,1.5),0.5)$  and perturbed by  $N_2\left(0.1 \right)$-noise.
\item
$\mathrm P_2$ (circle of radius 0.8): $x_i$ is a uniform point on $\mathcal{C}((0,0),0.8)$
which is perturbed by $N_2\left(0.1 \right)$-noise. 
\item
$\mathrm P'_2$ (two circles, radii 0.8 and 0.5): $x_i$ is a uniform point on $\mathcal{C}((0,0),0.8)\cup \mathcal{C}((1.5,1.5),0.5)$  and perturbed by  $N_2\left(0.1 \right)$-noise. 
\end{itemize}
For $k=0,1$ we consider the following simulation study: $K=2$ and  
$r_1=r_2=50$;
$\mathcal{G}_1$ consists of $45$ $\mathrm{APF}_k$s associated to simulations of point processes consisting of $100$ IID
points with distribution $\mathrm P_1$ (the non-outliers)
and $5$ $\mathrm{APF}_k$s 
obtained in the same way but from  $\mathrm P'_1$ (the outliers); 
$\mathcal{G}_2$ is specified in the same way as $\mathcal{G}_1$ but replacing
$\mathrm P_1$ and  $\mathrm P'_1$ with $\mathrm P_2$ and $\mathrm P'_2$, respectively; and we have correctly specified that $\alpha=0.2$.
Then we simulate 100 $\mathrm{APF}_k$s associated to $P_1$ and 100 $\mathrm{APF}_k$s associated to $P_2$, i.e.\ they are all non-outliers. Finally, we use
\eqref{e:assign} to assign each of these 200 $\mathrm{APF}_k$s to either $\mathcal{G}_1$ or $\mathcal{G}_2$.

The top panels in Figure~\ref{fig:trimmed mean and PP supervised} show the $20\%$-trimmed means $\overline{A}^{\, 0.2}_{1}$ and $\overline{A}^{\, 0.2}_{2} $  when $k=0$
(left) and  $k=1$ (right). The difference between the $20\%$-trimmed means is clearest when $k=1$ and so we expect that the assignment error is lower in that case. In fact  
  wrong assignments happen mainly
  when the support of  $\mathrm P_1$ or $\mathrm P_2$  is not well covered by the point pattern as illustrated 
in the bottom panels.

   \begin{figure}[t]
\begin{center}
\begin{tabular}[c]{c c}
\includegraphics[scale=0.25]{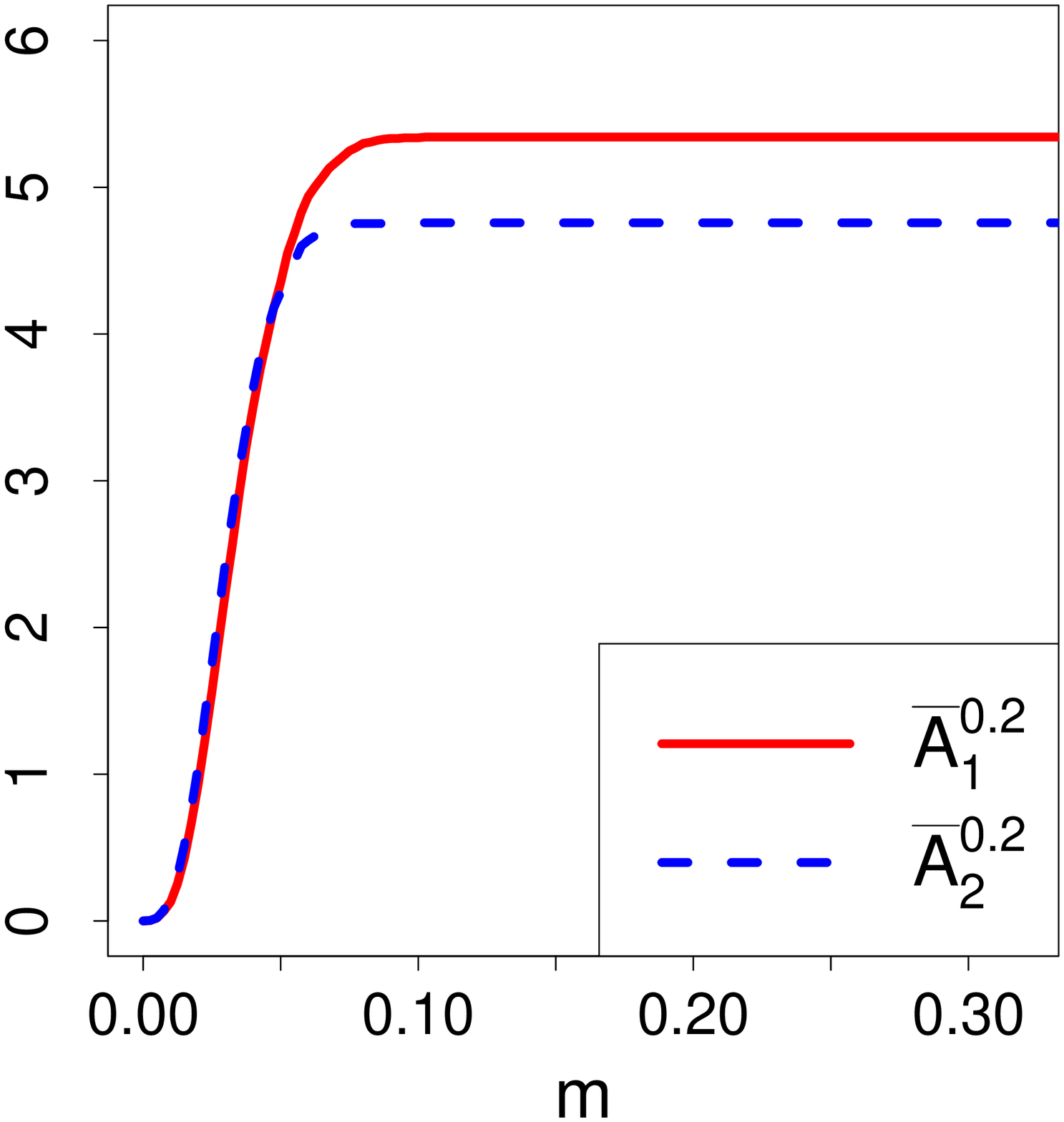}    &  \includegraphics[scale=0.25]{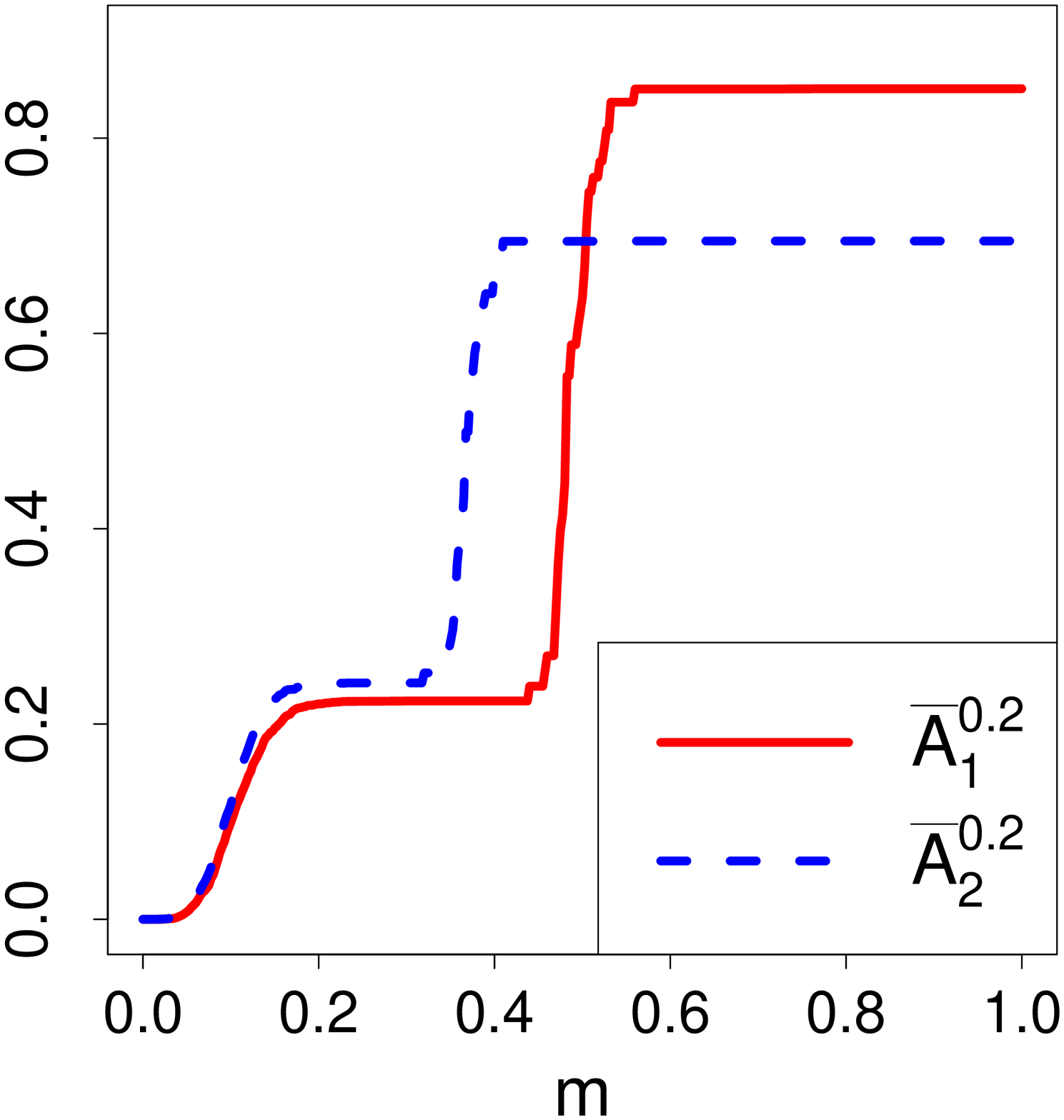} \\
 \includegraphics[scale=0.25]{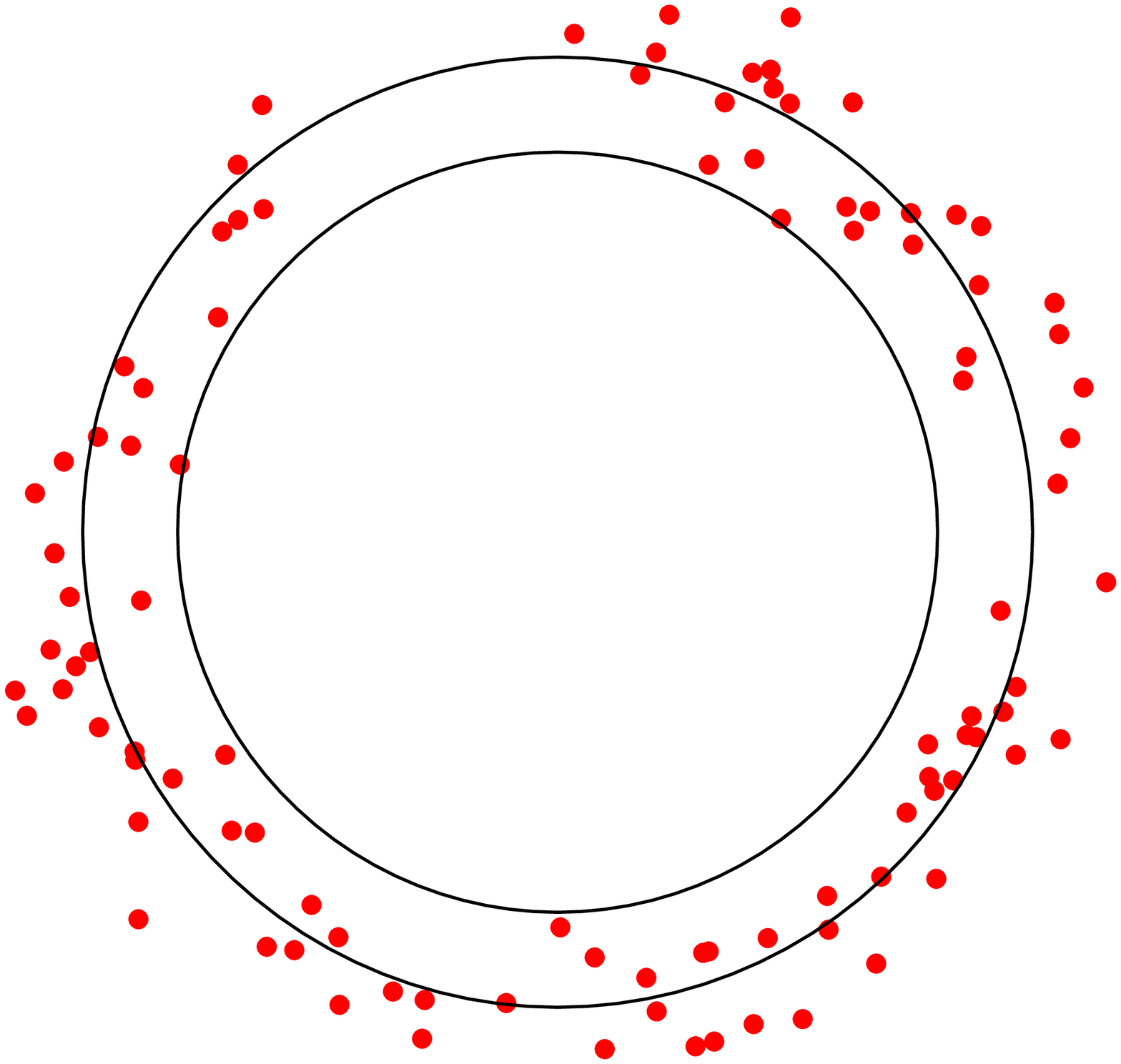}  & \includegraphics[scale=0.25]{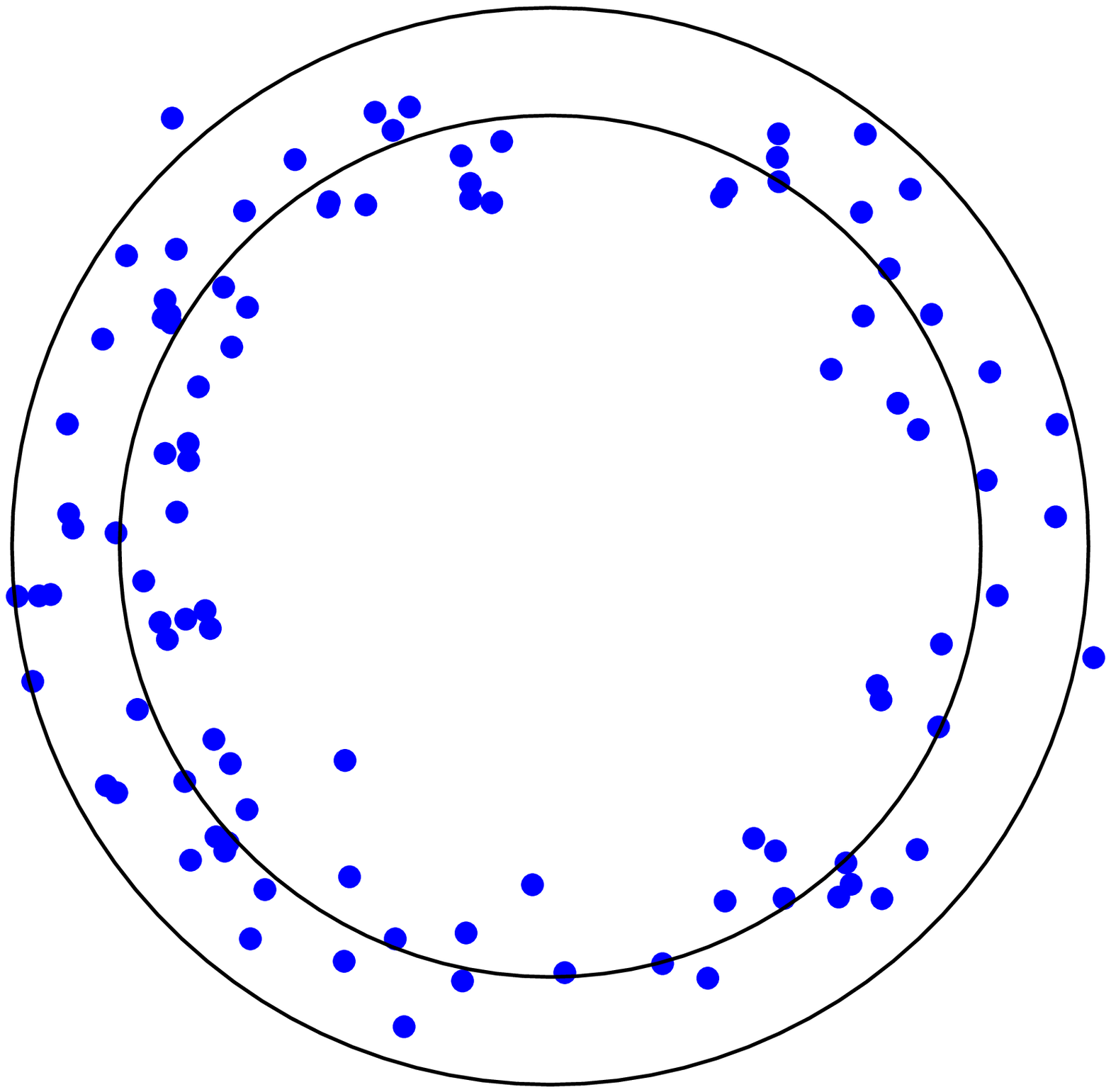}
\end{tabular}
\end{center}
\caption{Top panels: The $20\%$-trimmed mean functions with respect to
$\mathcal G_1$ and $\mathcal G_2$ when considering 
$\mathrm{APF}_0$s (left) and $\mathrm{APF}_1$s (right)
obtained from the Delaunay-complex and based on $100$ IID points following the distribution $\mathrm P_1$ (solid curve) or $\mathrm P_2$ (dotted curve).
Bottom panels:  Examples of point patterns with associated $\mathrm{APF}_0$s assigned to the wrong group, together with the circles of radius $0.8$ and $1$.}
\label{fig:trimmed mean and PP supervised}
\end{figure}

Repeating this simulation study $500$ times,
the percentage of $\mathrm{APF}_0$s wrongly assigned among the $500$ repetitions
has a mean of $6.7\%$ and a standard deviation of $1.7\%$, whereas for the $\mathrm{APF}_1$s
the mean is $0.24\%$ and the standard deviation is $0.43\%$. 
To investigate how the results depend on the radius of the smallest circle,
we repeat everything but with radius $0.9$ in place of $0.8$ when defining the distributions
$\mathrm{P}_2$ and $\mathrm{P}_2'$.
Then for the $\mathrm{APF}_0$s, the proportion of wrong assignments has a mean of $23.2\%$ and a standard deviation of $2.9\%$, and for the $\mathrm{APF}_1$s, a mean of $5.7\%$ and a standard deviation of $1.9\%$.
Similar to Example~9, the error is lowest when $k=1$ and this is due to the largest lifetime of a loop.

\section{Proof of Theorem~\ref{t:1}}\label{s:proof of theorem confidence region}

The proof of Theorem~\ref{t:1} follows along similar lines as in \cite{chazal:etal:13} as soon as we have verified Lemma~\ref{lemma APF donsker} 
below. Note that the proof of Lemma~\ref{lemma APF donsker} 
is not covered by the approach in~\cite{chazal:etal:13}. 

We first need to recall the following definition, 
where $\mathcal{B}_T$ denotes the topological space of  bounded real valued Borel functions defined on $[0,T]$ and  
its topology is induced by the uniform norm. 

\vspace{0.5cm}

\begin{definition}
A sequence  $\left \lbrace  X_r \right\rbrace_{r=1,2,\ldots}$  of random elements in $\mathcal{B}_T$  
converges in distribution to a random element $X$ in $\mathcal{B}_T$ if for any bounded  continuous function 
$f:\mathcal{B}_T \mapsto \R $, 
$\mathrm E f(X_r)$ converges to  $\mathrm E f(X)$ as  $r\rightarrow\infty$.
\end{definition}

\vspace{0.5cm}

\begin{lemma}\label{lemma APF donsker} Let the situation be as in Section~\ref{s:crmean}. 
As $r\rightarrow\infty$, $\sqrt{r} \left(\overline{A}_r-\mu\right)$ converges in distribution 
towards a zero-mean Gaussian process on $[0,T]$ 
with covariance function $c(m,m') = \mathrm{Cov}\left(  A_1 (m),  A_1(m') \right)$, $m,m' \in [0, T]$. 
\end{lemma}

\begin{proof}[Proof] 
 We need some notation and to recall some concepts of empirical process theory. 
For $D \in  \mathcal D_T^{k, n_{\mathrm{max}}}$, denote $A_D$ the $\mathrm{APF}_k$ of $D$. 
Let $\mathcal F = \lbrace f_m:0\leq m\leq T \rbrace$ be the class of functions 
$f_m:\mathcal D_T^{k, n_{\mathrm{max}}}\mapsto[0,\infty)$ given by $f_m(D)=A_{D}(m)$.
To see the connection with empirical process theory, we consider 
\begin{equation*}
\mathbb{G}_r(f_m) = \sqrt{r}\left( \frac{1}{r}\sum_{i=1}^r f_{m}(D_i)  - \mu(m) \right)
\end{equation*}
as an empirical process. 
 Denote $\|\cdot\|$ the $L^2$ norm on $\mathcal F$ with respect to the distribution of $D_1$,
 i.e.\ $\|f_m(\cdot)\|^2= \mathrm E\left\{A_{D_1}(m)^2\right\}$.  
 For $u,v\in \mathcal F $, the bracket $[u,v]$ is the set of all  functions $f\in\mathcal F$ with $u \leq f \leq v$. 
 For any $\epsilon>0$, $N_{[]}( \epsilon, \mathcal F , \|\cdot\| )$ 
 is the smallest integer $J\ge1$ such that $\mathcal{F} \subset \cup_{j=1}^J [u_j,v_j]$ 
 for some functions $u_1,\ldots,u_J$ and $v_1,\ldots,v_J$ in $\mathcal F$ 
 with $\|v_j-u_j\| \leq \epsilon$ for $j=1,\ldots,J$.  
 We show below that $\int_0^1 \sqrt{ \log \left( N_{[]} \left( \epsilon, \mathcal F ,  \|\cdot\|\right)  \right)} \ \mathrm d\epsilon$ is finite. 
Then,  by  Theorem 19.5 in  \citet{Vaart:98}, $\mathcal F$ is a so-called Donsker class which 
implies the convergence in distribution of $ \mathbb{G}_r(f_m)$ to a Gaussian process as in the statement of Lemma~\ref{lemma APF donsker}.

For  any sequence $-\infty=t_1 < \ldots<t_J=\infty$ with $J\geq 2$, for $j=1,\ldots,J-1$,  and for 
 $D  = \lbrace (m_1,l_1,c_1),\ldots,(m_n,l_n,c_n) \rbrace \in \mathcal D_T^{k,n_{\mathrm{max}}}$, 
let $u_j(D) = \sum_{i=1}^n c_i  l_i 1(m_i\le t_j)$ and $v_j( D ) = \sum_{i=1}^n c_i  l_i 1(m_i < t_{j+1})$  
(if  $n=0$, then $D$ is empty and we set $u_j(D) =  v_j( D ) = 0$). 
Then, for any $m \in [0,T]$, there exists a $j=j(m)$ 
such that  $u_j( D) \leq f_m(D) \leq  v_j(  D) $, i.e.\ $f_m(D) \in [u_j,v_j]$. 
Consequently,  $\mathcal{F} \subset \cup_{j=1}^{J-1} [u_j,v_j]$. 

We prove now that for any $\epsilon \in (0,1)$, the sequence $\lbrace t_j \rbrace _{ 1\leq j\leq J}$  can be chosen 
such that for $j=1,\ldots,J-1$, we have $\|v_j-u_j\| \leq \epsilon$. 
Write $D_1 =  \lbrace (M_1,L_1,C_1),\ldots,(M_N,L_N,C_N) \rbrace$, 
where $N$ is random and should not to be confused with $N$ in Sections~\ref{s:1:simulated dataset} and~\ref{s:sep} (if $N=0$, then $D_1$ is empty). 
Let $n\in\{1,\ldots, n_{\mathrm{max}}\}$ and conditioned on $N=n$, let $I$ be uniformly selected from $\{ 1,\ldots,n \}$. Then
\begin{align*} 
\mathrm E \left\{  \left(v_j \left( D_1 \right) -u_j \left( D_1 \right)  \right)^2  1 \left(N=n\right)  
\right\} 
                           &=  n^2 \mathrm E \left\{ 1\left(N=n \right)
                                  \frac{1}{n} \sum_{i=1}^n C_i L_i 1\left( M_i \in \left(t_j,t_{j+1} \right)  \right)    \right\}^2 \\
 &\leq T^2 n_{\mathrm{max}}^4 \mathrm E \left\{ 1\left(N=n\right)  
                                  1\left(M_I  \in \left(t_j,t_{j+1} \right)  \right)    \right\}^2   \\
                           &\leq T^2 n_{\mathrm{max}}^4  \mathrm P \left( M_I \in \left(t_j,t_{j+1} \right)  | N=n  \right),  
\end{align*} 
as $n\leq n_{\mathrm{max}}$, $C_i \leq n_{\mathrm{max}}$, and $L_i \leq T$. 
Further, 
\begin{align*}
\mathrm E \left\{  \left(v_j \left( D_1 \right) -u_j \left( D_1 \right)  \right)^2  1 \left(N=0\right)  
\right\} &=0.
\end{align*}
 Hence
\begin{align}\label{e:majoration bracket}
  \mathrm E \left\{v_j \left( D_1 \right) -u_j \left( D_1 \right) \right \}^2  
  &= \sum_{n=0}^{n_{\mathrm{max}}} \mathrm E \left\{  \left(v_j \left( D_1 \right) -u_j \left( D_1 \right)  \right)^2  1 \left(N=n\right)  \right\} \notag \\
   &\leq   T^2 n_{\mathrm{max}}^5 \max_{n = 1,\ldots,n_{\mathrm{max}}} 
  \mathrm P \left(  M_I \in \left(t_j,t_{j+1}\right)  | N=n \right).  
\end{align}
Moreover,  by Lemma~\ref{lemma:subdivision}  below, there exists
a finite sequence $\lbrace t_{n,j} \rbrace _{ 1\leq j\leq J_{n}}$ such that 
$\mathrm P (M_I \in  (t_{n,j},t_{n,j+1}) | N=n ) \leq  \epsilon^2/\left(T^2 n_{\mathrm{max}}^5\right)$ 
and  $ J_{n} \leq 2+ T ^2n_{\mathrm{max}}^5/\epsilon^2 $.  
Thus, by choosing 
\begin{align*}
\lbrace t_j \rbrace _{ 1\leq j\leq J} = \bigcup_{n=1,\ldots,n_{\mathrm{max}}}  \lbrace t_{n,j} \rbrace_{ 1\leq j\leq 
J_{n}},
\end{align*}
we have  $J \leq 2 n_\mathrm{max}  + T ^2n_{\mathrm{max}}^6/\epsilon^2$ and 
\begin{align*}
  \max_{n = 1,\ldots,n_{\mathrm{max}}} \mathrm P\left(  M_I \in \left(t_j,t_{j+1}\right) | N=n \right) 
\leq \frac{\epsilon^2}{T^2 n_{\mathrm{max}}^5}.
\end{align*}
Hence by~\eqref{e:majoration bracket},  $\|v_j-u_j\| \leq \epsilon$,  and so by definition, 
$N_{[]} \left( \epsilon, \mathcal F , \|\cdot\|\right) \leq 2n_\mathrm{max} + T ^2 n_{\mathrm{max}}^6/\epsilon^2 $.  
Therefore
\begin{align*}
 \int_0^1 \sqrt{ \log \left( N_{[]} \left( \epsilon, \mathcal F ,  \|\cdot\|\right)  \right)} \ \mathrm d\epsilon
        \leq  \int_0^1 \sqrt{ \log\left( 2n_{\mathrm{max}}+ T ^2n_{\mathrm{max}}^6/\epsilon^2 \right)} \ \mathrm 
d\epsilon<\infty.
\end{align*}
This completes the proof. 
\end{proof}
 
\paragraph*{Proof of Theorem~\ref{t:1}.} By the Donsker property established in the proof of Lemma~\ref{lemma APF donsker} and Theorem 2.4 in~\cite{gine:90},  $\sqrt{r} (\overline{A_r}- \overline{A^*_r})$ and $\sqrt{r} \left(\overline{A}_r-\mu\right)$ converge in distribution to the same process as $r\rightarrow\infty$, so the quantile of $\sup_{m\in[0,T]} \sqrt{r} \left|\overline{A_r}(m)- \overline{A^*_r}(m)\right|$ converges to the quantile of $\sup_{m\in[0,T]} \sqrt{r} \left|\overline{A_r}(m)- \mu(m)\right|$. Therefore,  $\hat{q}^B_\alpha$  provides the bounds for the asymptotic $100(1-\alpha)\%$-confidence region stated in Theorem~\ref{t:1}.

\vspace{0.5cm}
  
\begin{lemma}\label{lemma:subdivision}
Let $X$ be a positive random variable. For any $\epsilon \in (0,1)$, 
there exists a finite sequence $-\infty=t_1 < \ldots<t_J=\infty$  
such that $J\leq 2 + 1/\epsilon$ and for $j=1,\ldots, J-1$, 
 \begin{align*}
\mathrm  P \left(  X \in (t_j,t_{j+1}) \right) \leq \epsilon.
 \end{align*}
\end{lemma}

\begin{proof}
 Denote by $F$ the cumulative distribution function of $X$, 
 by $F( t-)$ the left-sided limit of $F$ at $t\in \R$,   
 and by $F^{-1}$ the generalised inverse of $F$, i.e.\ 
 $F^{-1}(y) = \inf \lbrace x \in \R : F(x) \geq y \rbrace$ for $y\in \R$.  
 We verify the lemma with $J= 2 +  \lfloor 1/\epsilon   \rfloor$,   $t_J=\infty$,
 and $t_j=  F^{-1} ( (j-1) \epsilon )$ for $j=1,\ldots,J-1$. 
  Then, for $j=1,\ldots, J-2$, 
 \begin{align*}
   \mathrm P \left(  X \in \left( t_j,t_{j+1} \right) \right)  &=   
	F \left( F^{-1} ( j \epsilon ) - \right) - F\left(  F^{-1} \left( \left(j-1\right) \epsilon \right) \right) 
     \leq  j\epsilon - \left(j-1\right)\epsilon=\epsilon.
 \end{align*}
 Finally, 
 \begin{align*}
    \mathrm P \left(  X \in \left(  t_{J-1},t_{J} \right) \right)  
                =   \mathrm P  (X > F^{-1} ( (J-2) \epsilon ) ) &= 1-  F \left(F^{-1} ( (J-2) \epsilon )  \right)
                \leq 1- \lfloor 1/\epsilon   \rfloor \epsilon  < \epsilon.
 \end{align*}
\end{proof}

\bibliographystyle{royal}
\bibliography{bibliography}
\end{document}